\documentclass{imsart} 

\usepackage{amsmath,amssymb}
\usepackage{amsthm}

\usepackage[margin=1.6in]{geometry}

\usepackage[numbers]{natbib}

\usepackage{booktabs}       
\usepackage{hyperref}       

\hypersetup{breaklinks=true,
    linkcolor=red,
    citecolor=blue,
    colorlinks=true
}

\newcommand{\pb}[1]{{ #1}}

\newcommand\defas{\stackrel{\text{def}}{=}}

\usepackage{tikz}
\usetikzlibrary{graphs}

\usepackage{thmtools}
\usepackage{mathtools}
\usepackage{thm-restate}
\usepackage{mathrsfs}
\usepackage{tensor}
\usepackage{bm}
\usepackage{lineno} 
\usepackage{pgfplotstable}
\pgfplotsset{compat=1.16}

\usepackage{cleveref}
\usepackage{soul}

\declaretheorem[name=Theorem,numberwithin=section]{theorem}
\declaretheorem[name=Proposition,sibling=theorem]{proposition}
\declaretheorem[name=Lemma,sibling=theorem]{lemma}
\declaretheorem[name=Corollary,sibling=theorem]{corollary}
\declaretheorem[name=Assumption,numberwithin=section]{assumption}

\declaretheorem[name=Remark,style=remark,numberwithin=section]{remark}

\usepackage{xr,refcount}
\usepackage{tcolorbox}
\usepackage{bold-preambule}
\usepackage{math-preambule}

\def\valueOfUstar{\bigl[ K^{2}
		(1 - 1/n)
		(L \tau^{-1} \| \bSigma \|_{\op} 
		(2 + \sqrt{p/n} + 2\eta_n)^{2}
		+1)^{-1}
		-
		3L/n\bigr]_{+}
            }
\def\veps{\varepsilon}

\numberwithin{equation}{section}

\begin{document}

\title{Asymptotic normality of robust $M$-estimators with convex penalty}
\runtitle{Asymptotic normality of robust $M$-estimators with convex penalty}
\author{
    Pierre C Bellec\thanksref{t1},
    Yiwei Shen 
    and
    Cun-Hui Zhang\thanksref{t2}
}
\runauthor{Bellec, Shen and Zhang}
  \thankstext{t1}{
    Research partially supported by the NSF Grants DMS-1811976
    and DMS-1945428.
  }
  
  \thankstext{t2}{
      Research partially supported by the NSF Grants DMS-1513378,
  IIS-1407939, DMS-1721495, IIS-1741390 and CCF-1934924.  }

\affiliation{Rutgers University}
\address{Department of Statistics, Rutgers University
}

\begin{abstract}
    This paper develops asymptotic normality
    results for individual coordinates of
    robust M-estimators with convex penalty
    in high-dimensions, where the dimension $p$ is at most
    of the same order as the sample size $n$, i.e,
    $p/n\le\gamma$ for some fixed constant $\gamma>0$.
    The asymptotic normality requires a bias correction 
    and holds for most coordinates
    of the M-estimator 
    for a large class of loss functions including
    the Huber loss and its smoothed versions
    regularized with a strongly convex penalty.

    The asymptotic variance that characterizes the width
    of the resulting confidence intervals
    is estimated with data-driven quantities.
    This estimate of the variance adapts
    automatically to low ($p/n\to0)$ or high ($p/n \le \gamma$)
    dimensions and does not involve the proximal operators
    seen in previous works on asymptotic normality of M-estimators.
    For the Huber loss, the estimated variance has a simple expression
    involving
    an effective degrees-of-freedom as well
    as an effective sample size.
    The case of the Huber loss with Elastic-Net penalty
    is studied in details and
    a simulation study confirms the theoretical findings. 
    The asymptotic normality results follow from
    Stein formulae for high-dimensional random vectors
    on the sphere developed in the paper which are of independent interest.
\end{abstract}

\maketitle


\section{Introduction}

\subsection{Robust inference}
In his seminal paper on robustness,
\citet{huber1964} introduced $M$-estimators for 
an unknown location parameter $\mu\in\R$ from observations
$Y_i = \mu + \veps_i$, {$i=1,\ldots,n$, where $\veps_i$} 
are iid noise random variables distributed as a mixture
$F=(1-\eps) N(0,1) + \eps H$ 
with $H$ being the distribution of the contaminated samples,
possibly chosen by an adversary.
Given a differentiable loss function $\rho:\R\to\R$ and its derivative
$\psi=\rho'$, Huber defined $M$-estimators as minimizers
$\hat \mu =\argmin_{b\in \R} \sum_{i=1}^n\rho(Y_i - b)$,
or equivalently as solutions to $\sum_{i=1}^n\psi(Y_i - b) =0$.
\citet{huber1964} went on to prove consistency
and asymptotic normality of such $M$-estimators, obtaining among
other results that if $\rho$ is convex and $\psi$ is absolutely continuous,
then
under mild assumptions on $F$, the convergence 
$\hat \mu\to \mu$ in probability holds as well as the asymptotic normality
$$
n^{1/2}(\hat \mu - \mu) \to^d {N\bigl(0, \E[\psi^2(\veps_1)]\big/\E[\psi'(\veps_1)]^2\bigr).}
$$

Huber's $M$-estimators were extended to regression models,
where a design matrix $\bX\in\R^{n\times p}$ is observed
together with responses $y_i = \bx_i^\top \bbeta + \veps_i$
where $\bx_1,...,\bx_n$ are the rows of $\bX$ and $\veps_1,...,\veps_n$
are possibly contaminated noise random variables as in the previous
paragraph.
For fixed or slowly growing dimension $p$ as $n\to+\infty$,
consistency and asymptotic normality
of $M$-estimators of the form $\hbbeta = \argmin_{\bb\in\R^p}
\sum_{i=1}^n \rho(y_i - \bx_i^\top\bb)$ were obtained,
see \cite[Section 7]{huber2004robust} or \cite{portnoy1985asymptotic}
among others. Explicitly,
if $\be_j\in\R^p$ is a canonical basis vector and one is interested
in the asymptotic normality of $\hat\beta_j-\beta_j$
for the purpose of confidence intervals,
\cite{portnoy1985asymptotic} shows that
\begin{equation}
    \bigl(\be_j^\top(\bX^\top\bX/n)^{-1} \be_j\bigr)^{-1/2}
    ~ \sqrt n(\hat \beta_j - \beta_j)
 ~
 \to^d
 ~
 N\Bigl(0, \frac{\E[\psi^2(\veps_1)]}{\E[\psi'(\veps_1)]^2}\Bigr)
 \label{eq:intro-asymptotic-normality-low-dimensions}
\end{equation}
if $(p\log n)^{3/2}/n\to0$ and under mild assumptions on $\bX$.
As in the location parameter problem of the previous paragraph,
the asymptotic variance is characterized by the ratio
$\E[\psi^2(\veps_1)]/\E[\psi'(\veps_1)]^2$.  

The last decade has seen striking developments of similar
asymptotic normality results in high-dimensions,
where $p/n\to \gamma$ for some constant $\gamma<1$, cf.
\cite{el_karoui2013robust,bean2013optimal,karoui2013asymptotic,donoho2016high,el_karoui2018impact}.
In terms of asymptotic normality, these works show that
if $\bX$ has iid $N(0,\bSigma)$ rows,
the unregularized $M$-estimator
$\hbbeta = \argmin_{\bb\in\R^p}
\sum_{i=1}^n \rho(y_i - \bx_i^\top\bb)$
satisfies asymptotic normality 
of the form
\begin{equation}
    \label{eq:asymptotic-normality-r2}
(\be_j^\top\bSigma^{-1} \be_j)^{-1/2}
~
\sqrt p (\hat\beta_j - \beta_j) \to^d N(0,r^2)
\end{equation}
where $r>0$
is a deterministic 
constant that captures the high-dimensionality of the problem
\cite[Lemma 1]{el_karoui2013robust}.
The constant $r>0$ is determined by solving a system of nonlinear equations
with two unknowns:
In the unregularized setting, \cite[S2]{el_karoui2013robust} describes
this system of nonlinear equations
with unknowns $(r,c)$ as
\begin{equation}
    \label{system-karoui}
    \left\{
    \begin{aligned}
        \E\bigl[1-[\text{prox}_c(\rho)]'(\veps_1 + r Z)\bigr] 
        &= \gamma,
        \\
        \E\bigl[ \bigl( \veps_1 + rZ - [\text{prox}_c(\rho)](\veps_1 + r Z)\bigr)^2\bigr] 
        &= \gamma r^2
    \end{aligned}
    \right.
    \qquad
    \text{ as }~
    p/n\to\gamma,
\end{equation}
where $Z\sim N(0,1)$ is independent of $\veps_1$,
and $\prox_c(\rho)(t)=\argmin_{u\in\R} \rho(u) + (t-u)^2/(2c)$
denotes the proximal operator of the convex function $t\to c\rho(t)$ 
with derivative $[\text{prox}_c(\rho)]'(t)$.
The optimality conditions
$$
c^{-1}(t - [\text{prox}_c(\rho)](t) )
= \psi([\text{prox}_c(\rho)](t))
$$
of the proximal minimization problem
leads to the expressions
$c^{-2}\gamma r^2 = \E[\psi([\text{prox}_c(\rho)](\veps_1+rZ))^2]$
and
$c^{-1}\gamma = \E[\frac{d}{dt}\psi([\text{prox}_c(\rho)](t)) |_{t=\veps_1+rZ}]$
for the solutions $(r,c)$ to the above system.
Hence \eqref{eq:asymptotic-normality-r2} can be rewritten as
\begin{equation}
    (\be_j^\top\bSigma^{-1}\be_j)^{-1/2}
    \sqrt n(\hat \beta_j - \beta_j)
 ~
 \to^d
 ~
 N\Bigl(0, \frac{
    \E[\psi([\text{prox}_c(\rho)](\veps_1+rZ))^2]
 }{
    \E[\frac{d}{dt}\psi([\text{prox}_c(\rho)](t)) |_{t=\veps_1+rZ}]^2
 }\Bigr),
 \label{eq:intro-asymptotic-normality-p-n-to-gamma}
\end{equation}
see, e.g., \cite[Theorem 4.1 and Corollary 4.6]{donoho2016high}.
These results embody that when $p$ and $n$ are of the same order,
the asymptotic variance in
\eqref{eq:intro-asymptotic-normality-low-dimensions}
must be modified to account for the high-dimensionality of the problem
by (a) replacing $\psi$ in the numerator and $\psi'$ in the denominator
by their compositions with the proximal operator $\text{prox}_c(\rho)$,
and (b) adding the extra Gaussian term $rZ$ to the initial noise $\veps_1$.
The distribution of $\veps_1 + rZ$ is sometimes referred to
as the effective noise.
The Gaussian assumption can be relaxed and some of the above
results are still valid if $\bX$ has iid centered entries with variance one
\cite{karoui2013asymptotic,el_karoui2018impact}.
Despite the subtle introduction of the proximal operator
and the constants $(r,c)$,
it is remarkable that the
informal ratio $\frac{\text{average}[\psi^2]}{\text{average}[(d/dt)\psi]^2}$
unifies the results \eqref{eq:intro-asymptotic-normality-low-dimensions}
and \eqref{eq:intro-asymptotic-normality-p-n-to-gamma}
in both low and high-dimensions.

All results of the previous section are applicable when $p/n\to\gamma$
with $\gamma<1$. For $\gamma > 1$ regularization is required
to ensure the uniqueness of $\hbbeta$,
for instance through an additive penalty which leads to
regularized $M$-estimators of the form
\begin{equation}
    \hbbeta 
    =
    \argmin_{\bb\in\R^p}
    \frac 1n\sum_{i=1}^n \rho(y_i - \bx_i^\top\bb)
    + g(\bb)
    \label{eq:M-estimators-regularized-intro}
\end{equation}
for some convex penalty function $g:\R^p\to\R$.
The case of Ridge regularization with $g(\bb)=\tau\|\bb\|_2^2/2$
for some constant $\tau > 0$ is treated in
\cite{karoui2013asymptotic,el_karoui2018impact}.
In this case, the two solutions $(r,c)$ of
a system of two nonlinear equations similar to \eqref{system-karoui}
characterize the error $\|\hbbeta-\bbeta\|_2$,
asymptotic normality and asymptotic variance
of $\sqrt n((1+a)\hat\beta_j - \beta_j)$
where $a$ is a constant capturing the bias induced by
regularization \cite[Proposition 3.30]{el_karoui2018impact}
and $a$ is a function of $(\gamma,r,c)$.
\cite{thrampoulidis2018precise} characterize the
error $\|\hbbeta-\bbeta\|_2$ for a large class of $(\rho,g)$ pairs
using a technique known as the Convex Gordon Min-Max theorem
pioneered by \cite{stojnic2013framework},
and the recent paper \cite{huang2020} on Approximate Message Passing
focused on $g(\bb)=\lambda\|\bb\|_1$
and $\rho$ either the Huber loss or the absolute value.
Little is known, however, on asymptotic normality
of the regularized estimators \eqref{eq:M-estimators-regularized-intro}
for penalty functions different from the square norm $\bb\mapsto \tau\|\bb\|_2^2$. 
The theories developed in \cite{thrampoulidis2018precise,huang2020} do
not readily provide asymptotic normality results
and regularized $M$-estimators of the form
\eqref{eq:M-estimators-regularized-intro} lack confidence interval
capabilities.
One goal of the present paper is to fill this gap.

A separate line of research develops asymptotic normality results
and confidence intervals
for regularized least-squares estimators of the form
\begin{equation}
    \label{eq:least-squares}
    \hbbeta = \argmin_{\bb\in\R^p}
    \frac{1}{2n} \sum_{i=1}^n (y_i - \bx_i^\top\bb)^2
    + g(\bb)
    =
    \argmin_{\bb\in\R^p}
    \|\by - \bX\bb\|_2^2/(2n)
        + g(\bb)
\end{equation}
where $\bX$ has rows $\bx_1,...,\bx_n$.
Early results studied the Lasso with $g(\bb)=\lambda\|\bb\|_1$
\citep{ZhangSteph14,JavanmardM14a,GeerBR14}
under a sparsity condition $s\log(p)=o(\sqrt n)$ where $s=\|\bbeta\|_0$,
or Ridge regression \citep{buhlmann2013statistical}.
For the Lasso the sparsity condition was later improved
to $s\log^2(p)/n \to 0$ \citep{javanmard2018debiasing},
to $s \log(p/s)/n\to 0$ \citep{bellec_zhang2019debiasing_adjust}
and $p/n\to\gamma\in(0,\infty)$ with
$s\lesssim n/\log(p/s)$ (\citep{JavanmardM14b,miolane2018distribution}
for isotropic Gaussian design and 
\citep{celentano2020lasso}
\citep[Theorem 3.2]{bellec2019second_poincare} for non-isotropic Gaussian design).
For penalty functions beyond the Lasso and Ridge regularization,
\cite[Proposition 4.3(iii)]{celentano2019fundamental} provides
asymptotic normality on average over the coordinates
for permutation invariant penalty function $g$ in \eqref{eq:least-squares},
and \cite[Theorem 3.1]{bellec2019second_poincare} proves asymptotic normality for individual
coordinates of \eqref{eq:least-squares} under a strong convexity assumption.
A high-level message of these works is that one must de-bias
the regularized estimator \eqref{eq:least-squares} in order
to obtain asymptotic normality at the $\sqrt n$-adjusted rate
and construct confidence intervals. In the $p/n\to\gamma$ regime
that is the focus of the present paper,
this bias correction 
takes the following form. 
Under a strong convexity assumption and for $\bX$ with iid $N(\mathbf{0}_p,\bSigma)$ rows,
\cite{bellec2019second_poincare} proves that for most coordinates $j=1,...,p$,
\begin{equation}
\frac{
(n-\df)(\hat\beta_j - \beta_j) + \be_j^\top\bSigma^{-1}\bX^\top(\by-\bX\hbbeta)}{\|\by-\bX\hbbeta\|_2}
\Omega_{jj}^{-1/2}
\to^d N(0,1)
\label{eq:intro-adjustment-asymptotic-normality-df-square-loss}
\end{equation}
where $\Omega_{jj}=\be_j^\top\bSigma^{-1}\be_j$ 
and $\df$ is the effective degrees of freedom of $\hbbeta$
defined as the Jacobian of the map $\by\mapsto \bX\hbbeta$ for fixed $\bX$.
For $\bSigma=\bI_p$ and consequently $\Omega_{jj}=1$,
a similar bias correction proportional to
$\be_j^\top\bX^\top(\by-\bX\hbbeta)$ is visible in the asymptotic normality
result \cite[Proposition 4.3(iii)]{celentano2019fundamental},
although there the coefficients $(n-\df)$ and $\|\by-\bX\hbbeta\|_2$
in \eqref{eq:intro-adjustment-asymptotic-normality-df-square-loss}
are replaced with deterministic scalar
counterparts obtained by solving a system of nonlinear equations
of a similar nature as \eqref{system-karoui}.
Another goal of the present paper is to equip robust $M$-estimators
\eqref{eq:M-estimators-regularized-intro},
with $\rho$ different than the squared loss, with de-biasing
and asymptotic normality results similar to the previous display,
allowing for general robust loss functions $\rho$ coupled with general
convex penalty functions $g$.

\subsection{Contributions}
Our contributions are the following.

\begin{enumerate}
    \item 
        We provide de-biasing and asymptotic normality results
        for robust $M$-estimators with convex penalty functions
        when $p$ and $n$ are of the same order.
        This leads to confidence intervals for the $j$-th coordinate
        $\beta_j$ of the unknown coefficient vector $\bbeta$.
        Asymptotic normality holds for a large class of robust loss functions,
        including the Huber loss and its smoothed versions.
    \item
        Although this 
        bias correction 
        required for 
        asymptotic normality resembles the one-step estimators
        recommended in the theory of classical $M$-estimator
        to improve efficiency
        (e.g., \cite{bickel1975one}), a notable difference 
        from the low-dimensional case is the requirement of a degrees-of-freedom
        adjustment to amplify the one-step correction.
        For the squared loss, this degrees-of-freedom adjustment
        takes the form of multiplication by $(n-\df)$ in
        \eqref{eq:intro-adjustment-asymptotic-normality-df-square-loss};
        one contribution of this paper is to identify the 
        degrees-of-freedom adjustment that leads to asymptotic normality
        for robust and regularized $M$-estimators, beyond the squared loss.
    \item 
        The asymptotic variance is estimated by random,
        data-driven
        quantities, as opposed to the deterministic scalars
        $(r,c)$ that determine the asymptotic variance
        for unregularized estimators in \eqref{eq:intro-asymptotic-normality-p-n-to-gamma}.
        The fact that the asymptotic variance is estimated
        by observable quantities makes this results more suitable
        for confidence intervals (case in point: computing the solution
        $(r,c)$ of \eqref{system-karoui} and the asymptotic variance
        in \eqref{eq:intro-asymptotic-normality-p-n-to-gamma}
        requires the knowledge
        of the noise distribution subject to contamination).
        The asymptotic normality result takes the form
        $$
        \hat V^{-1/2} \Omega_{jj}^{-1/2}
        \sqrt n (\hat\beta_j - \beta_j)
        + [\text{observable bias correction}]
        \approx N(0,1)
        $$
        where the data-driven variance estimate $\hat V$ again
        is a ratio of the form $\frac{\text{average}[\psi^2]}{\text{average}[(d/dt)\psi]^2}$
        for a particular 
        sense of average to be defined in \eqref{eq:asymptotic-variance} below.
        Interestingly, the expression
        for this average and
        $\hat V$ does not involve the proximal mapping 
        in \eqref{eq:intro-asymptotic-normality-p-n-to-gamma}.
        This informal statement will be made precise 
        in \Cref{sec:variance} below.
    \item 
        In order to derive these new asymptotic normality results,
        we develop in \Cref{sec:stein} new identities for random vectors
        uniformly distributed on the Euclidean sphere.
        Although the argument of the present paper for asymptotic normality
        is closely related to that used in \cite{bellec2019second_poincare}
        for the squared loss, this previous theory for the squared loss
        for functions of standard normal vector
        does not extend to robust loss functions due to the lack
        of strong convexity of $\rho$ for robust losses,
        and consequently the lack
        of explicit lower bounds on $\frac1n\sum_{i=1}^n \psi(y_i-\bx_i^\top\hbbeta)^2$.
        Developing
        these new identities and the corresponding asymptotic normality
        results for random vectors
        uniformly distributed on the sphere is a crucial step
        to overcome the lack of global strong convexity of $\rho$ for robust losses
        and to obtain the asymptotic normality results of the previous bullet points.
        These new identities provide novel Stein formulae
        for random vectors on the sphere and may be used
        more broadly for elliptical distributions.
\end{enumerate}

\section{Model and main results}

\subsection{Model and assumptions}

We consider a linear model
\begin{equation} \label{eq:model}
\by = \bX \bbeta + \bep,
\end{equation} 
where $\by \in \R^{n}$, $\bX \in \R^{n \times p}$ and $\bep \in \R^{n}$,
with a regularized M-estimator
\begin{equation} \label{eq:bbeta}
\hbbeta = \argmin_{\bb \in \R^p} 
\frac{1}{n} \sum_{i=1}^n \rho (y_i - \bx_i^\top \bb) + g(\bb),
\end{equation} 
where $\rho: \R \to \R$ is the loss and $g: \R^{p} \to \R$
is the penalty.
\begin{assumption} [Assumptions on the loss] \label{as:rho}
Let $\rho: \R \to \R$ be convex and continuously differentiable on $\R$,
with derivative $\psi = \rho'$ being $L$-Lipschitz with
\begin{equation}
    K^2 \le \psi'(x) + \psi(x)^2 \qquad \text{ for almost every } x \in \R
\end{equation}
for some positive constant $K>0$ independent of $n,p$.
\end{assumption}

Two families of robust losses
that do not satisfy \Cref{as:rho} are
non-differentiable losses such as 
the least absolute deviations $\rho(x)=|x|$,
and $\delta$-insensitive losses such as 
$\rho(x) = (|x|-\delta)_+^2$
as $\psi(x)^2 + \psi'(x)=0$ in a neighborhood of 0.
\Cref{as:rho} is verified by the Huber loss $\rho(x)=\int_0^{|x|}\min(1,t)dt$
with $K=1$,
as well as by any smooth version of the Huber loss, for instance $\rho(x) = \sqrt{1+x^2}$ with $K^2=\frac{23}{27}\approx 0.852$.
The one-sided logistic loss $\rho(x) = \log(1+e^x)$ also satisfies
\Cref{as:rho}.
As $\psi$ is increasing, \Cref{as:rho} implies
that $\rho$ is $K^2/2$-strongly convex in
the interval $\{x\in \R: \psi(x)^2\le K^2/2\}$,
and conversely if $\rho$ is $\mu$-strongly convex
in the interval $\{x\in\R: \psi(x)^2\le C\}$ then
\Cref{as:rho} is satisfied with $K^2 = \min(\mu, C)$.


\begin{table}[ht]
    \centering
\begin{tikzpicture}
    \begin{axis}[
    height=2in,
    width=2.6in,
        ]
        \addplot+[no marks,domain=-3:3,samples=301] {  ( abs(x)<1 ? x^2/2 : abs(x)-1/2 ) };

        \addplot+[no marks,domain=-3:3,samples=301] {  sqrt(1+x^2) - 1 };

        \addplot+[no marks,domain=-3:3,samples=301] {  ( abs(x)<1 ? x^2/2 : ( abs(x)<2 ? 1/6 - abs(x)/2 +x^2 - abs(x)^3/6  : -7/6 + 3*abs(x)/2) ) };
\end{axis}
\end{tikzpicture}
\begin{tikzpicture}
    \begin{axis}[
    height=2in,
    width=2.6in,
        ]
        \addplot+[no marks,domain=-3:3,samples=301] {  ( abs(x) ) };
        \addplot+[no marks,domain=-3:3,samples=301] {  ( abs(x)<1 ? 0 : (abs(x)<2 ? (abs(x)-1)^2/2 : abs(x) - 3/2 )) };
\end{axis}
\end{tikzpicture}

\caption{
    Left: robust loss functions satisfying \Cref{as:rho}:
    the Huber loss $x\mapsto H(x) =  \int_0^{|x|}\min(t,1)dt$,
    its smoothed versions
    $x\mapsto \sqrt{1+x^2}$
    and \\
    $
    x\mapsto 
    \frac{x^2}{2} I\{|x|\le 1\}+
    (\frac{1}{6} - \frac{|x|}{2} + x^2 - \frac{|x|^3}{6})I\{|x|\in(1,2)\}
    +
    (    \frac{-7}{6} + \frac{3|x|}{2})I\{(|x|\ge2)\}
    .
    $
    Right: two loss functions that do not satisfy \Cref{as:rho}:
    the absolute deviation loss $x\mapsto |x|$ and
    the $1$-insensitive loss
    $x\mapsto H[(|x|-1)_{+}]$ where $H(\cdot)$ is the Huber loss.
}
\label{table:rho}
\end{table}

\begin{assumption} [Strongly convexity of $g$] \label{as:g}
For some constant $\tau>0$ independent of $n,p$, the penalty $g: \R^{p} \to \R$ is $\tau$-strongly convex in the sense that $\bx\mapsto g(\bx) - \tau \|\bx\|_2^2/2$ is convex.
\end{assumption}
Some useful characterizations of strong convexity are the following.
Throughout, $\partial g(\bb)$ denotes the subdifferential of $g$ at $\bb\in\R^p$. 
Then $g$ is $\tau$-strongly convex if and only if
\begin{equation}
    \label{strongly-convex-1}
    g (\ba) - g(\bb) \geq \bu^{\top} (\ba - \bb) + (\tau/2) \| \ba - \bb \|_{2}^{2} 
\quad
\text{ for all }
\bu \in \partial g(\ba) \text{ and } \ba,\bb \in \R^{p}.
\end{equation}
Similarly $g$ is $\tau$-strongly convex if and only
for any $\ba,\bb \in \R^{p}$
\begin{equation}
    \label{strongly-convex-2}
    (\bu - \bv)^\top(\ba - \bb)
\ge 
\tau \| \ba - \bb \|_{2}^{2}
\quad
\text{ for all }
\bu \in \partial g(\ba),
\bv \in \partial g(\bb).
\end{equation}

\begin{assumption}
    \label{as:feature}
    The rows of the design matrix $\bX$ are iid $N(\bzero, \bSigma)$ random vectors
    and all the eigenvalues of $\bSigma \in \R^{p\times p}$ are in $[\kappa, 1/\kappa]$ 
for some constant $\kappa\in (0,1)$ independent of $n,p$.
The noise $\bep$ is independent of $\bX$ and admits a density
with respect to the Lebesgue measure.

\end{assumption}

\begin{assumption} \label{as:asymp}
    $p/n\le \gamma$ for some constant $\gamma>0$ independent of $n,p$.
\end{assumption}

\subsection{Notation}
Throughout the paper,
we use the notation
\begin{equation} \label{eq:bpsi}
\bpsi = (\psi(y_{i} - \bx_i^\top \hbbeta))_{i \in [n]},
\qquad 
\bpsi' = (\psi'(y_i - \bx_i^\top \hbbeta))_{i \in [n]},
\qquad
\bh = \hbbeta - \bbeta.
\end{equation} 
For each $j \in [p]$, let $\be_{j}$ denote the $j$-th vector in the standard basis of $\R^{p}$,
and let  
\begin{equation}
\label[definition]{def:Quz}
\Omega_{jj}
=
\be_{j}^{\top} \bSigma^{-1} \be_{j}
,\qquad
\bz_{j}
=
\bX\bSigma^{-1} \Omega_{jj}^{-1} \be_{j}
, \qquad
\bQ_{j}=\bI_{p} - \bSigma^{-1} \Omega_{jj}^{-1} \be_{j}\be_{j}^{\top}.
\end{equation}
We remark that the above definition implies the following properties:
\begin{itemize}
\item
$\bX = \bX \bQ_{j} + \bz_{j} \be_{j}^{\top}$
and 
$\bX \bbeta = \bX \bQ_{j} \bbeta + \bz_{j} \beta_j.$ 

\item
$\bz_{j} \sim N(\bzero, \Omega_{jj}^{-1}\bI_{n})$ is independent of $\bX \bQ_{j}$ (cf. Proposition \ref{prop:indep_zj}).

\item 
Under Assumption \ref{as:feature}, $\Omega_{jj} \in [\kappa, 1/\kappa]$. 
\end{itemize} 
By construction of $\bz_j$ and $\bQ_j$, the response $\by$
can be decomposed as
$\by = \beta_j \bz_j + \bX\bQ_j \bbeta + \bep$
where $\beta_j\in\R$ is the scalar parameter of interest
for a fixed covariate $j\in[p]$,
the vector $\bQ_j \bbeta$ is a high-dimensional nuisance parameter
and $\bep$ is independent noise.
Under the additional assumption of $\veps_i\sim N(0,1)$, 
$\Omega_{jj}^{-1}$ is the Fisher information for the estimation of $\beta_j$. 

In the proof, it will be useful to treat $\bpsi = \bpsi(\bep,\bX)$ as a map from $\R^{n\times (p+1)} \to\R^n$, \pb{formally defined as
\begin{equation*}
\begin{aligned}
\hbbeta (\bep, \bX) & = \argmin_{\bb \in \R^{p}}
\sum _{i \in [n]} 
\frac{\rho ( \ep_{i}  - \bx_{i}^{\top} ( \bb - \bbeta ))}{n}
+
g ( \bb ),
\\\bpsi (\bep, \bX) & = \psi ( \bep + \bX\bbeta - \bX \hbbeta (\bep, \bX) ).
\end{aligned}
\end{equation*}
Since $(\bbeta,\bep)$ are unknown, we cannot compute the derivatives
of $\bpsi$ a priori. However, for a fixed $\bX$, the quantity
$\bpsi(\by-\bX \bbeta ,\bX)$ is observable since
all terms in $(\bbeta,\bep)$ cancel out
($\bpsi(\by-\bX \bbeta ,\bX)$
is simply $\psi(\by-\bX\hbbeta)$ with $\hbbeta$ in \eqref{eq:bbeta}).
We can thus define the observable matrix of size $n\times n$}
\begin{equation}
    [ \nabla_{\by}\bpsi ]^\top 
    \defas
    (\partial/\partial\by)\bpsi(\by-\bX \bbeta,\bX)
\label{hbpsi-y}
\end{equation}
at every point $\by\in\R^n$ where 
$\by\mapsto
\bpsi(\by-\bX \bbeta ,\bX)$
is differentiable holding $\bX$ fixed. 
By Proposition \ref{prop:Lipschitz_in_X} below, 
the map $\by\mapsto {\bpsi(\by-\bX \bbeta,\bX)}$ is $L$-Lipschitz
and $\nabla_{\by}\bpsi$ exists at Lebesgue almost every $\by$,
and with probability one since $\by$ has continuous distribution
under \Cref{as:feature}. 
Furthermore, the gradient \eqref{hbpsi-y} at the currently observed
data $(\by,\bX)$ does not depend on any unknown quantity.
It can be computed from $(\by,\bX)$ either by finding
a closed form expression for \eqref{hbpsi-y} for a given
penalty function, or by approximation using finite difference
or other numerical methods.
Finally, throughout the paper
$\Phi(t)=(2 \pi)^{-1/2} \int_{-\infty}^t \exp(-{u^2/2})du$ is the
standard normal cumulative distribution function.

\subsection{Main result}

In the following result, we consider implicitly a sequence
of integer pairs $(n,p)$, regression problems \eqref{eq:model}
and $M$-estimator \eqref{eq:bbeta}. For instance, one can think
of $p=p_n$ as a nondecreasing function of $n$ and $(g,\bbeta,\hbbeta,\rho)$
are also implicitly indexed by $n$ with values possibly changing with $n$.

\begin{theorem} [Asymptotic Normality result for M-estimator]  \label{thm:main_thm}
Consider the linear model \eqref{eq:model} 
and the M-estimator 
 $\hbbeta$ in \eqref{eq:bbeta}.
Assume $\E [\| \bSigma^{1/2} \bh \|_2^2] \leq \mathscr{R} < + \infty$. 
Let Assumptions~\ref{as:rho}, \ref{as:g},
\ref{as:feature} and
\ref{as:asymp} be fulfilled for constants
$\mathscr{R}, \tau,\kappa,K,L, \gamma>0$ independent of $n,p$.
Define the map
$\by\mapsto \bpsi(\by-\bX \bbeta,\bX)$ and its Jacobian $[\nabla_{\by}\bpsi]^\top$ in \eqref{hbpsi-y} holding $\bX$ fixed.
For each $j\in[p]$ let
\begin{align}
	\xi_{j} 
        &=
	\| \bpsi \|^{-1}
        \bigl[
        \bpsi^{\top} \bz_{j}	
	-
	n^{-1} \| \bz_{j} \|_{2}^{2}  (\beta_{j} - \hbeta_{j}) \trace 
	( 
	{\nabla_{\by}\bpsi}
	)
        \bigr],
        \label{xi_j}
        \\
	\xi_{j}'
        &=
	\| \bpsi \|^{-1}
        \bigl[
        \bpsi^{\top} \bz_{j}	
	-
        \Omega_{jj}^{-1}  (\beta_{j} - \hbeta_{j}) \trace 
	( 
	{\nabla_{\by}\bpsi}
	)
        \bigr].
        \label{xi_j-prime}
\end{align}
Then for any positive sequence $(a_{p})$ with $\lim_{p \to +\infty} a_{p} = +\infty$,
\begin{equation}
    \max_{j \in {J_{n,p}}}	
	\Big|\P \bigl(\Omega_{jj}^{1/2} \xi_{j} \leq t\bigr) - \Phi (t)\Big|
        +
	\Big|\P \bigl(\Omega_{jj}^{1/2} \xi_{j}' \leq t\bigr) - \Phi (t)\Big|
	\to 0,
    \label{eq:J_p-main-theorem}
\end{equation}
for some ${J_{n,p}} \subseteq [p]$ satisfying $|{J_{n,p}}| / p \geq 1 - a_{p}/ p $.
\end{theorem}

The proof is given in \Cref{sec:proof_main}.
To interpret the above result, 
we remark that for any slowly increasing sequence
$a_p$ such as $a_p = \log p$ or 
$a_p = \log\log p$,
the asymptotic normality \eqref{eq:J_p-main-theorem}
holds uniformly over all coordinates $j=1,...,p$ except $a_p$ of them, 
so that both $\xi_j$ and $\xi_j'$ are asymptotically pivotal for an overwhelmingly majority of $\beta_j$.
Another interpretation is given in the following Corollary:
For any given precision threshold $\upsilon>0$,
there exist at most $a_*$ coordinates $j=1,...,p$ such that
$|\P(\Omega_{jj}^{1/2}\xi_j\le t) - \Phi(t)|\ge \upsilon$
where $a_*$ is a constant independent of $n,p$.

\begin{corollary}
    \label[corollary]{cor:main}
    Let the setting and assumptions of Theorem~\ref{thm:main_thm}
    be fulfilled. For any arbitrarily small
    constant $\upsilon >0$ independent of $n,p$, define
    $$
    J_{n,p}^\upsilon= \Bigl\{j\in[p]:
        \Big|\P(\Omega_{jj}^{1/2}\xi_j \le t)
        - \Phi(t)
        \Big| > \upsilon
    \Bigr\}.
    $$
    Then, 
    $\sup_{n,p}|J_{n,p}^\upsilon| \le a_*$ for a certain 
    constant $a_*$ depending on $\{\upsilon,\tau,\mathscr{R},\gamma,L,K, t\}$ only.
    In other words, for any $(n,p)$ with $p/n\le\gamma$
    there are at most a constant number of coordinates $j=1,...,p$
    such that
    $\big|\P(\Omega_{jj}^{1/2}\xi_j \le t)
        - \Phi(t)
        \big| > \upsilon
    $.
    The same conclusion holds with $\xi_j$ replaced by $\xi_j'$.
\end{corollary}

\begin{proof}
    [of \Cref{cor:main}]
    We proceed by contradiction. If the claim does not hold,
    there exists a constant $\upsilon_*>0$ such that 
    $|J_{n,p}^{\upsilon_*}|\ge 2a_p$ for an unbounded 
    sequence $a_p$.
    By extracting a subsequence if necessary, we may assume
    without loss of generality 
    that $a_p$ is monotonically increasing with $a_p\to+\infty$.
    By Theorem~\ref{thm:main_thm} there exists $J_{n,p}\subset [p]$
    such that \eqref{eq:J_p-main-theorem} holds.
    By definition of $J_{n,p}^{\upsilon_*}$ and ${J_{n,p}}$, we 
    have ${J_{n,p}} \cap {J_{n,p}^{\upsilon_*}} = \emptyset$ for $p$ large enough.
    This implies that
    $p \ge
    |{J_{n,p}}|
    + |{J_{n,p}^{\upsilon_*}}|
    \ge (p - a_ p) + 2 a_p$
    for $p$ large enough, a contradiction.
\end{proof}

\subsection{Data-driven variance estimate}
\label{sec:variance}

Except for at most a constant number of coordinates $j\in[p]$, the
approximation
\begin{equation} \label{eq:normality}
\hat V^{-1/2} \Omega_{jj}^{-1/2}
\sqrt n (\hat\beta_j - \beta_j)
+ [\text{{bias correction}}]
\approx N(0,1)
\end{equation}
{holds} where the observable {bias correction} is given
by the first term in \eqref{xi_j-prime} and the data-driven
variance estimate
that characterizes the length of confidence intervals
for $\beta_j$ is
\begin{equation}
\hat V
=
\frac{\|\bpsi\|_2^2/n}{(\trace({\nabla_{\by}\bpsi})/n)^2}
= \frac{
    n^{-1} \sum_{i=1}^n \psi_i^2
}{
    \bigl(n^{-1} \sum_{i=1}^n (\partial/\partial y_i) \psi_i \bigr)^2
}. 
\label{eq:asymptotic-variance}
\end{equation}
This ratio of an average of $\psi^2$ by a squared average of a derivative of 
$\psi$
mirrors the robust asymptotic results in \eqref{eq:intro-asymptotic-normality-low-dimensions} and \eqref{eq:intro-asymptotic-normality-p-n-to-gamma}
as discussed in the introduction. 
Confidence intervals can be readily constructed from
    \Cref{thm:main_thm} or the informal approximation \eqref{eq:normality}:
    a 95\%-confidence interval for $\beta_j$ is given by
    $\hat\beta_j +(\Omega_{jj}\hat V/n)^{1/2}([\text{bias correction}] \pm 1.96)$,
    that is,
    $$
    \Bigl[
    \hat\beta_j
    + \frac{\Omega_{jj} \bpsi^\top\bz_j}{\trace[\nabla_{\by} \bpsi]}
    -
    \Bigl(\frac{\Omega_{jj}\hat V}{n}\Bigr)^{1/2}1.96,~~
    \hat\beta_j
    + \frac{\Omega_{jj} \bpsi^\top\bz_j}{\trace[\nabla_{\by} \bpsi]}
    +
    \Bigl(\frac{\Omega_{jj}\hat V}{n}\Bigr)^{1/2}1.96
    \Bigr].
    $$
In contrast with the asymptotic variance in \eqref{eq:intro-asymptotic-normality-p-n-to-gamma}, the above $\hat V$ involves observable quantities.
In particular, while the asymptotic variance
in \eqref{eq:intro-asymptotic-normality-p-n-to-gamma}
depends on the distribution of the noise through the solutions
$(r,c)$ of the system \eqref{system-karoui},
the knowledge of the noise distribution is not required to compute
$\hat V$ and construct confidence intervals for $\beta_j$.

\Cref{thm:main_thm} is valid for $p/n\le \gamma$,
without requiring a specific limit for the ratio $p/n$.
\Cref{thm:main_thm} is also valid for $p=o(n)$, so that
\Cref{thm:main_thm} and the estimated asymptotic variance
\eqref{eq:asymptotic-variance}
unifies both low- and high-dimensional asymptotic
normality results.

For the Huber loss
\begin{equation}
    \label{eq:Huber-loss-H}
H(u) 
=
\int_0^{|u|}\min(1,t) dt
=
\begin{cases}
    u^2/2 & \text{ if } |u|\le 1, 
    \\
    |u|-1/2 & \text{ if } |u|>1,
\end{cases}
\end{equation}
the quantity $\trace[\nabla_{\by}\bpsi]$ has a simpler form:
By the chain rule
$$\trace[\nabla_{\by}\bpsi] = \trace [\diag \bpsi']
-
\trace[\diag(\bpsi') (\partial/\partial\by) \bX\hbbeta]
$$
where the differentiation is understood holding $\bX$ fixed
as in \eqref{hbpsi-y}.
With $\hat n = \trace[\diag \bpsi']$ the number of observations
such that $y_i-\bx_i^\top\hbbeta$ falls in the range where the Huber loss
is quadratic and $\df = \trace[\diag(\bpsi') (\partial/\partial\by) \bX\hbbeta]$
representing the degrees-of-freedom of the M-estimator $\hbbeta$,
the quantity $\trace[\nabla_{\by}\bpsi]$
appearing in the dominator of $\hat V$ simplifies to
$\trace[\nabla_{\by}\bpsi] = \hat n - \df$.
In this case, the asymptotic normality for $\xi_j'$  in \Cref{thm:main_thm}
takes the form
\begin{equation}
    \Omega_{jj}^{1/2}
    \xi_j'
    =
\frac{(\hat n - \df)(\hat\beta_j - \beta_j) + 
    \Omega_{jj}\bz_j^\top \psi(\by-\bX\hbbeta)
}{\|\psi(\by-\bX\hbbeta)\|_2} \Omega_{jj}^{-1/2}
\quad 
\to^d
\quad
N(0,1)
\label{eq:xi_j_prime-ENet-N01}
\end{equation}
uniformly over all $j\in J_{n,p}$.
This extends the asymptotic normality result 
\eqref{eq:intro-adjustment-asymptotic-normality-df-square-loss}
to the Huber loss with the following important modifications:
the sample size $n$ is replaced by $\hat n$
and the residuals $\by-\bX\hbbeta$
are replaced by $\bpsi = \psi(\by-\bX\hbbeta)$.
The variance $\hat V$ that determines the length of the confidence
interval resulting from \eqref{eq:normality} presents a trade-off
between $\|\bpsi\|^2/n$, an effective sample size $\hat n$
and the degrees-of-freedom $\df$: For confidence intervals with small length,
the M-estimator $\hbbeta$ should have small residuals measured
as $\|\psi(\by-\bX\hbbeta)\|_2$, small degrees-of-freedom $\df$,
and large effective sample size $\hat n$.

\subsection{Example}
This section specializes the above results to scaled versions
of the Huber loss
\eqref{eq:Huber-loss-H}
and the Elastic-Net penalty
$g(\bb) = \lambda \|\bb\|_1 + \tau\|\bb\|_2^2/2$
for tuning parameters $\lambda,\tau>0$.
We consider the $M$-estimator
\begin{equation}
\hbbeta =
\argmin_{\bb\in\R^p}
\frac 1 n \sum_{i=1}^n \sigma^{2} H(\sigma^{-1}(y_i - \bx_i^\top\bb))
+
\lambda \|\bb\|_1 + \tau\|\bb\|_2^2/2,
\label{eq:huber-e-net}
\end{equation}
which corresponds to the scaled Huber
loss $\rho(u) = \sigma^2 H(\sigma^{-1} u)$
where $\sigma>0$ is a scaling parameter.
Since the derivative $H'$ is 1-Lipschitz, so is $\psi=\rho'$.
Furthermore, 
$\psi'(u) = H''(\sigma^{-1}u)= 1$ for $|u| \le \sigma$ and 
$|\psi(u)| = \sigma |H'(\sigma^{-1} u )|
=\sigma$ for $|u| > \sigma$,
so that $\min_{x\in \R}[\psi^2(x) + \psi'(x)]\ge \min(1,\sigma^2)$
and \Cref{as:rho} is satisfied with $L=1$ and $K^2=\min(1,\sigma^2)$.
\Cref{as:g} is also satisfied as the penalty is the sum of the $\ell_1$
norm plus $\tau \|\bb\|^2/2$.
The quantity
$\trace[\nabla_{\by}\bpsi]$ appearing in \Cref{thm:main_thm} in the denominator
of estimated variance $\hat V$ is computed in \cite[Proposition 2.3]{bellec2020out}: For almost every $(\bX,\by)$,
\begin{align*}
    \trace[\nabla_{\by}\bpsi] &=
\hat n - \df,
\\
\df &=\trace\bigl[\diag(\bpsi')\bX_{\hat S}(\bX_{\hat S}^\top\diag(\bpsi')\bX_{\hat S} + n \tau \bI_{|\hat S|})^{-1} \bX_{\hat S}^\top\diag(\bpsi')\bigr],
\end{align*}
where $\hat n = \trace[\diag(\bpsi')]= |\{i\in[n]: \psi'(y_i-\bx_i^\top\hbbeta)=1\}|$
is the number of observations $i=1,...,n$ such that
$y_i - \bx_i^\top\hbbeta$ falls in the range where $\rho$ is quadratic,
$\hat S =\{j\in[p]: \hat\beta_j\ne 0 \}$ 
and $\bX_{\hat S}\in\R^{n\times |\hat S|}$
is the submatrix of $\bX$ containing columns indexed in $\hat S$.

\subsection{Simulation study}
\label{sec:simulation}
We provide here simulations to showcase the asymptotic normality
in the Huber loss and Elastic-Net penalty example of the previous section.

\begin{table}[p]
    \centering
\begin{tabular}{@{}|c|c|c|c|c|@{}}
\toprule
$ (\lambda,\tau) $      & $(n^{-1/2},0.1)$    & $(n^{-1/2},0)$     &  $(2n^{-1/2},0.1)$  & $(2n^{-1/2},0)$   \\ \midrule
$ \hat{n} $             & $81.2 ~~\pm 5.8$ & $111.9 ~~\pm 6.4$ & $38.3 ~~\pm 4.6$ & $49.3 ~~\pm 5.5$ \\\midrule
$ \df $                 & $53.6 ~~\pm 3.9$  & $81.4 ~~\pm 5.1$  & $14.2 ~~\pm 2.3$ & $24.8 ~~\pm 4.1$  \\\midrule
$ |\hat{S}| $           & $96.4 ~~\pm 6.8$ & $81.4 ~~\pm 5.1$  & $26.5 ~~\pm 4.6$  & $24.8 ~~\pm 4.1$  \\\midrule
$ \sqrt{\hat{V} / n} $  & $0.44 ~~\pm 0.06$  & $0.37 ~~\pm 0.07$   & $0.56 ~~\pm 0.09$  & $0.55 ~~\pm 0.11$  \\\midrule
$\Omega_{jj}^{1/2}\xi_j'$                                       &
\includegraphics[width=23mm]{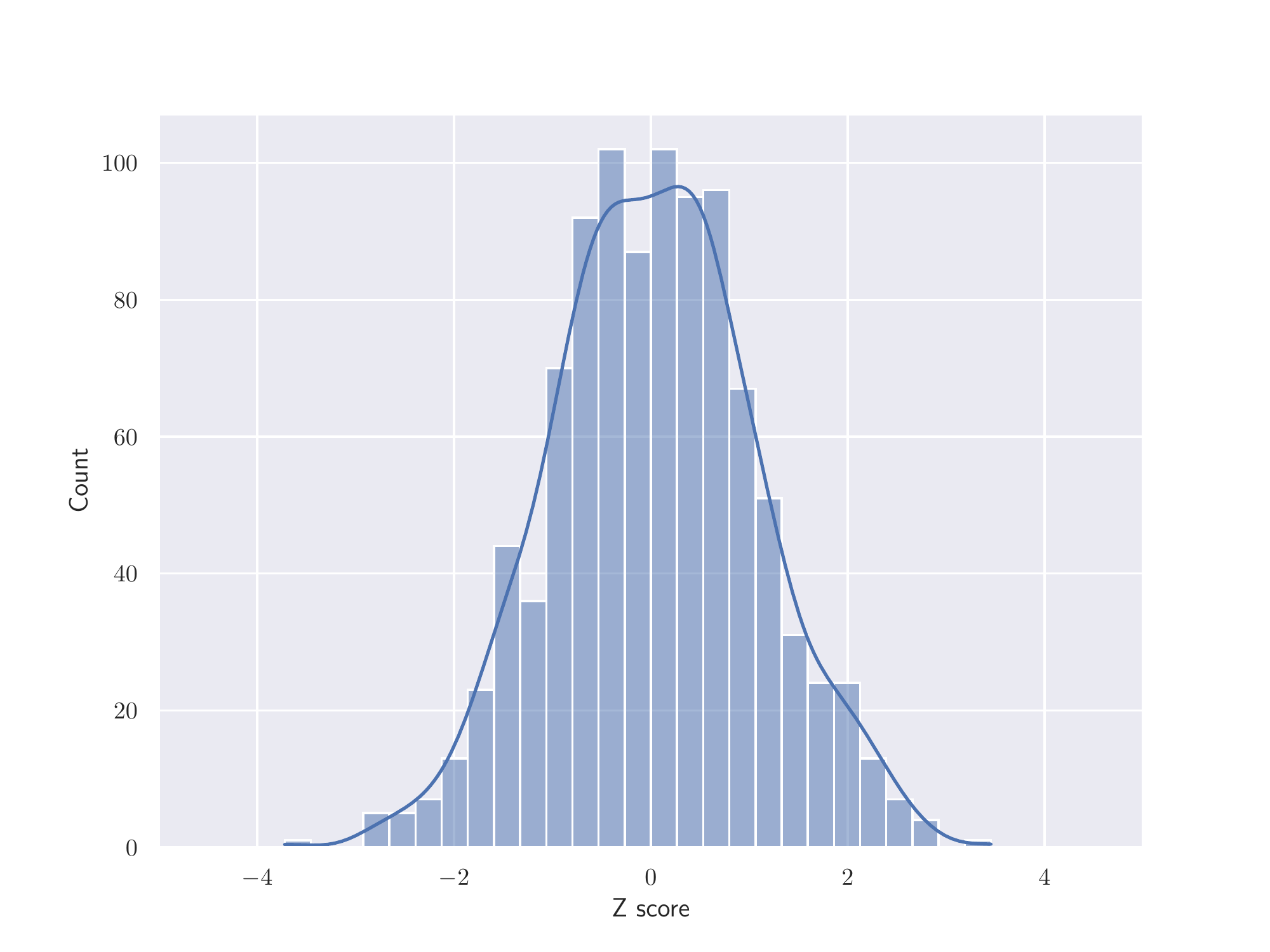}   &
\includegraphics[width=23mm]{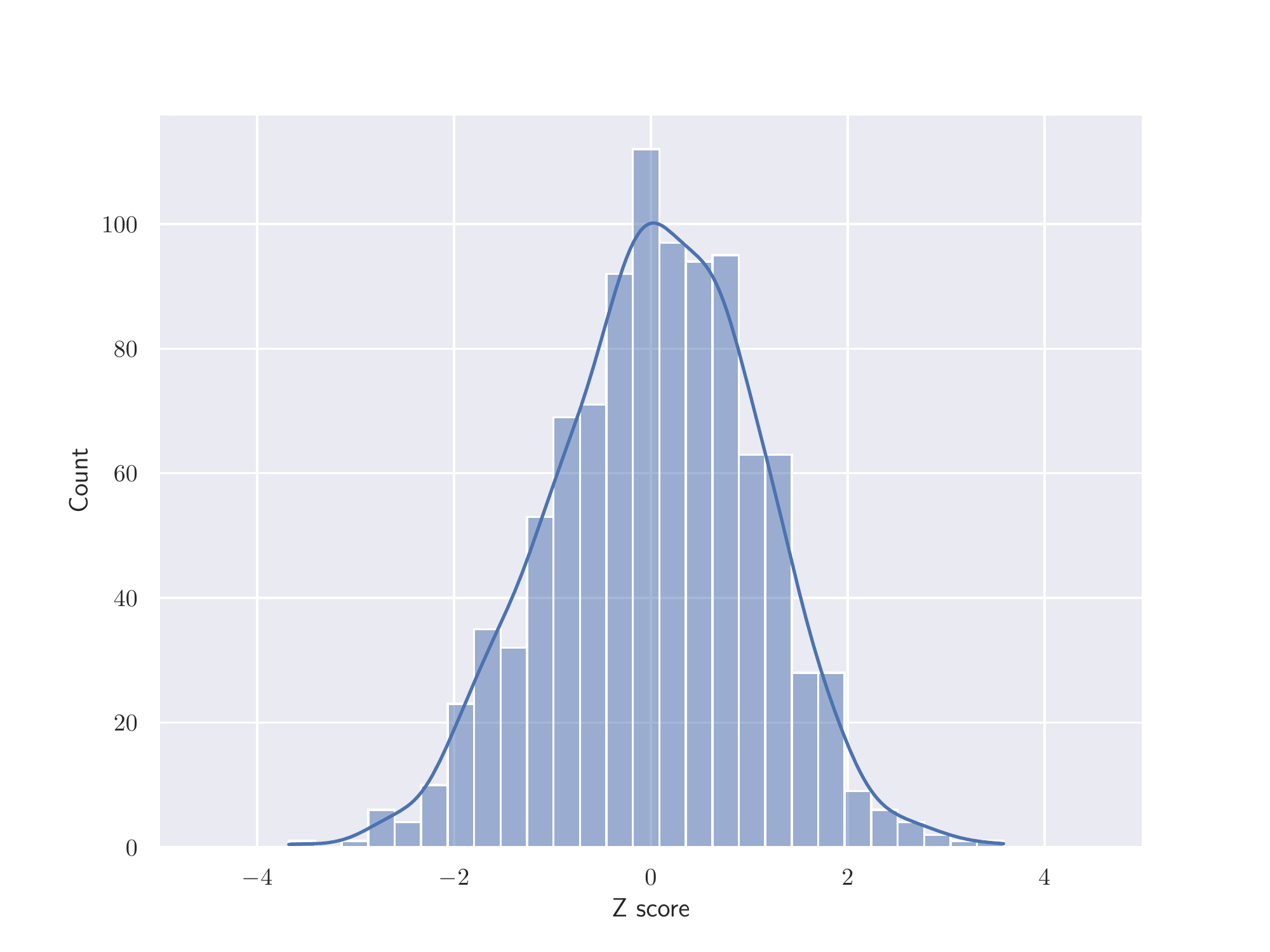}   &
\includegraphics[width=23mm]{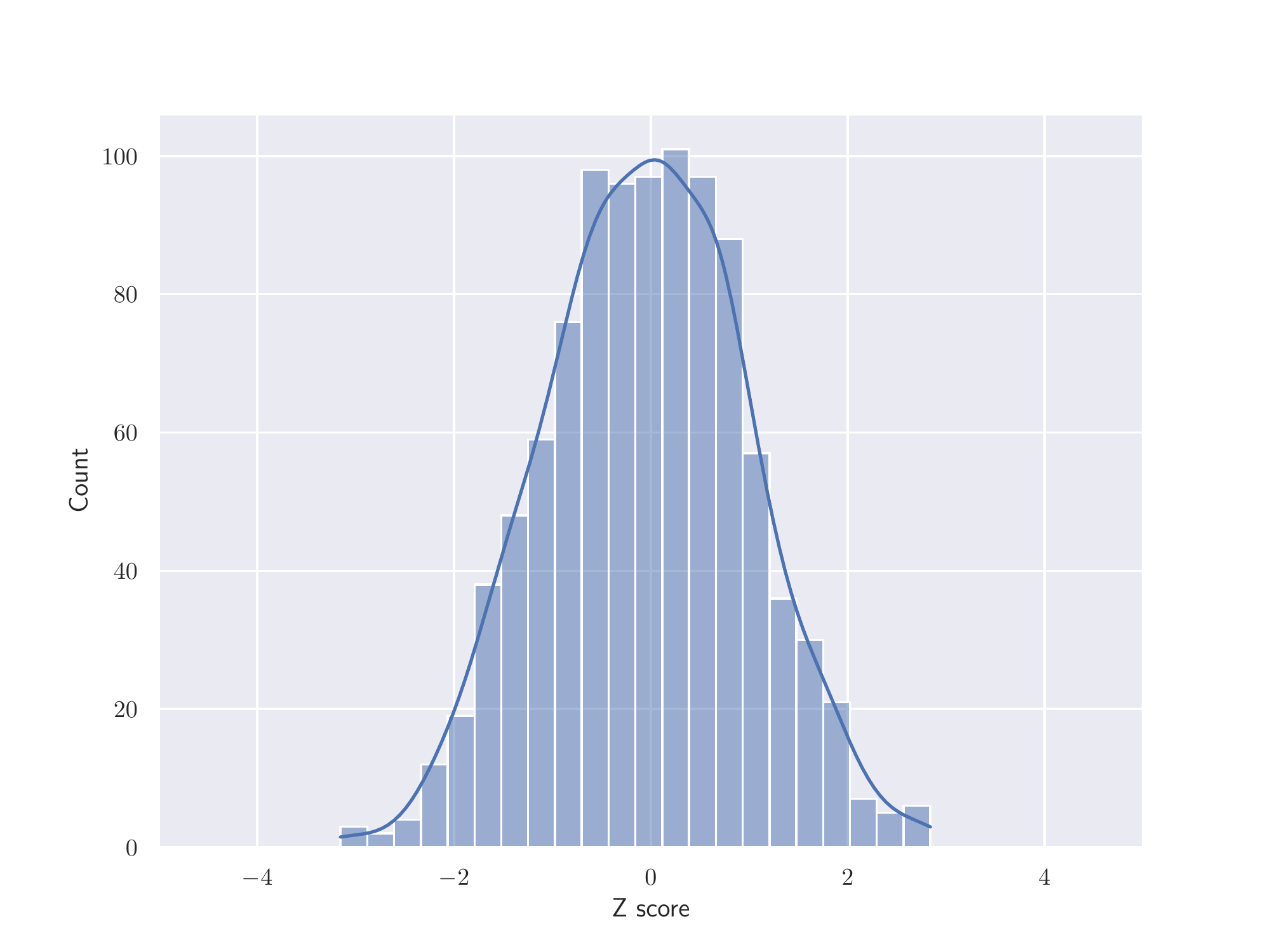} &
\includegraphics[width=23mm]{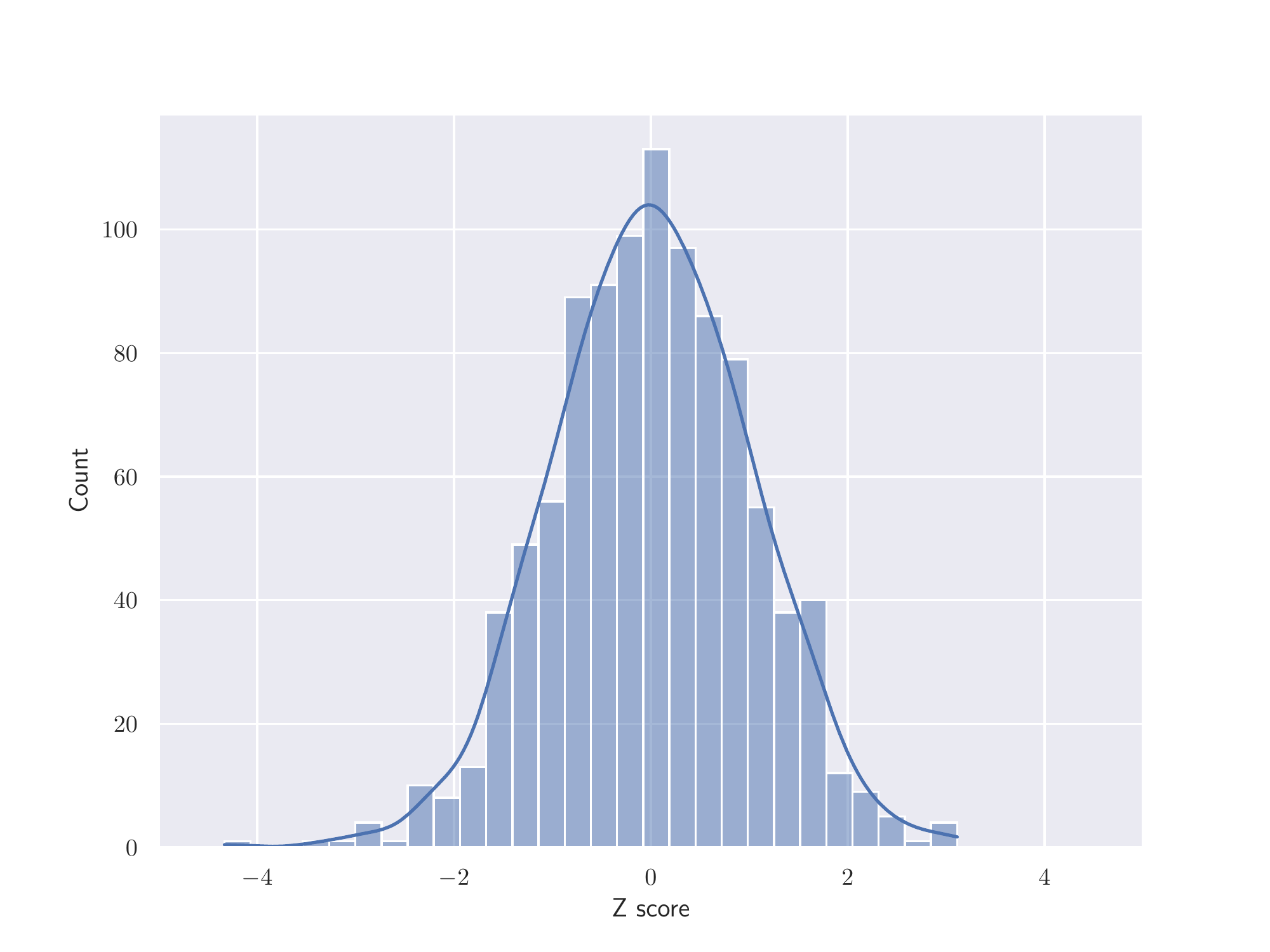} \\ &
\includegraphics[width=23mm]{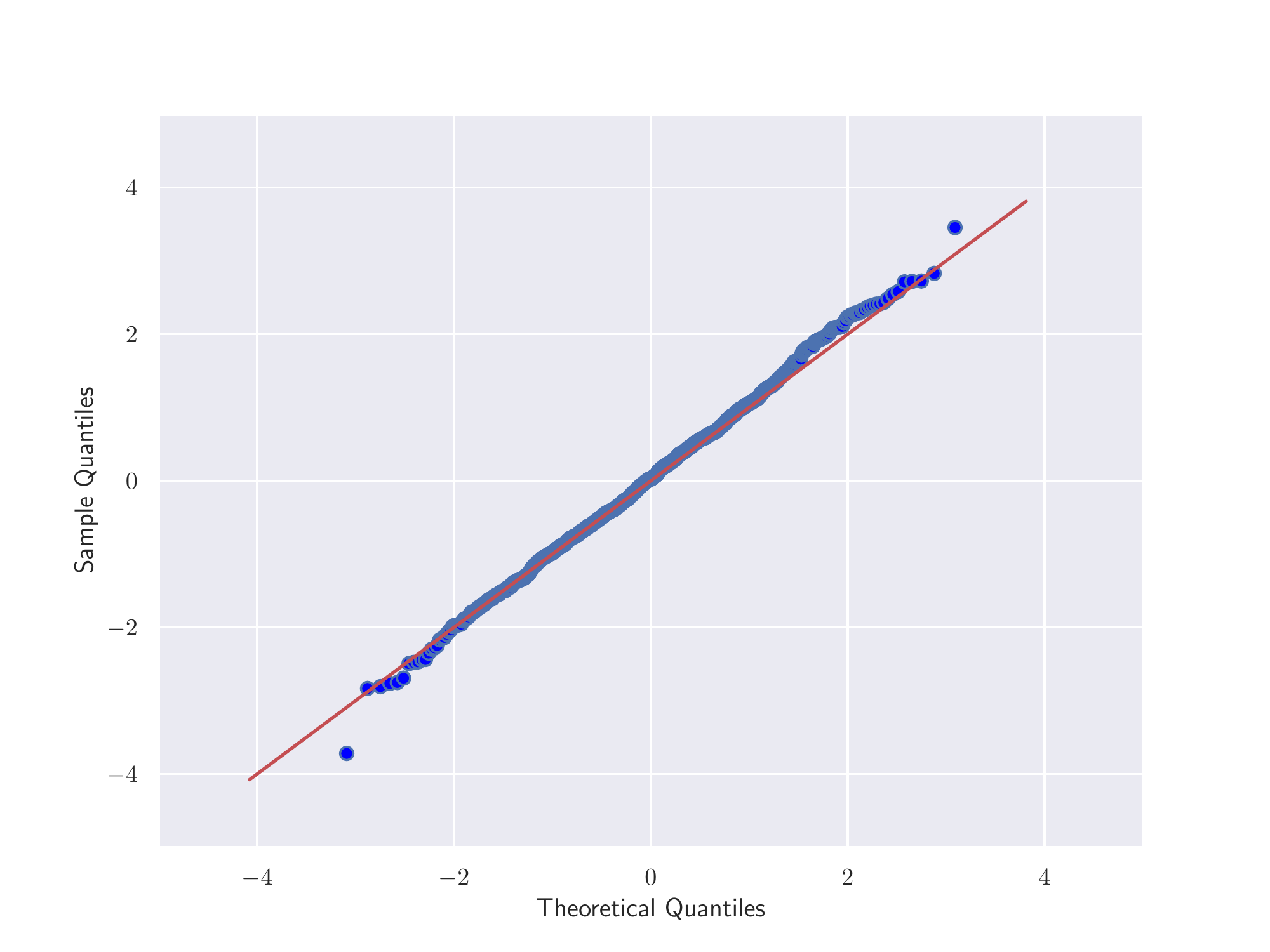}     &
\includegraphics[width=23mm]{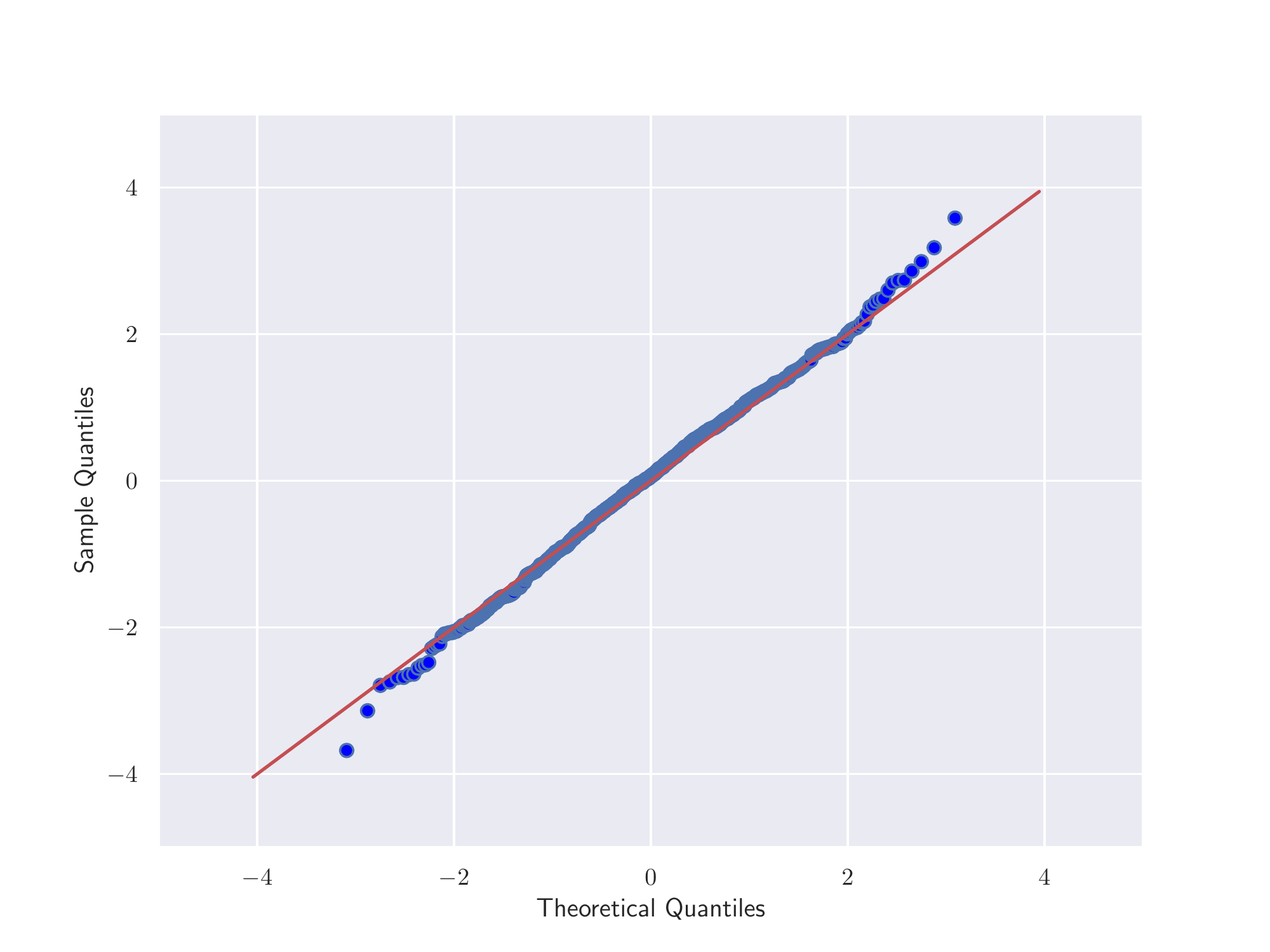}     &
\includegraphics[width=23mm]{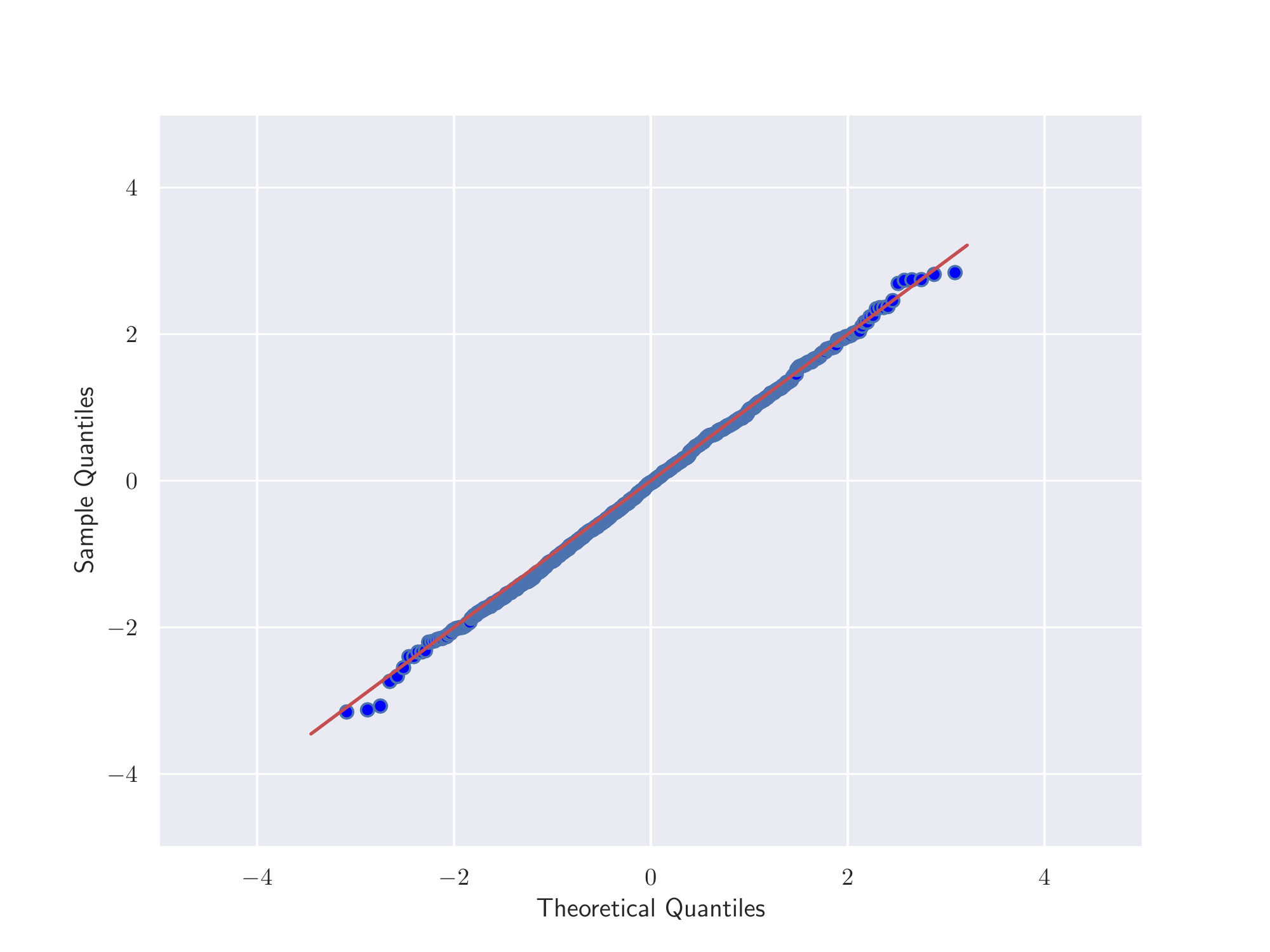}   &
\includegraphics[width=23mm]{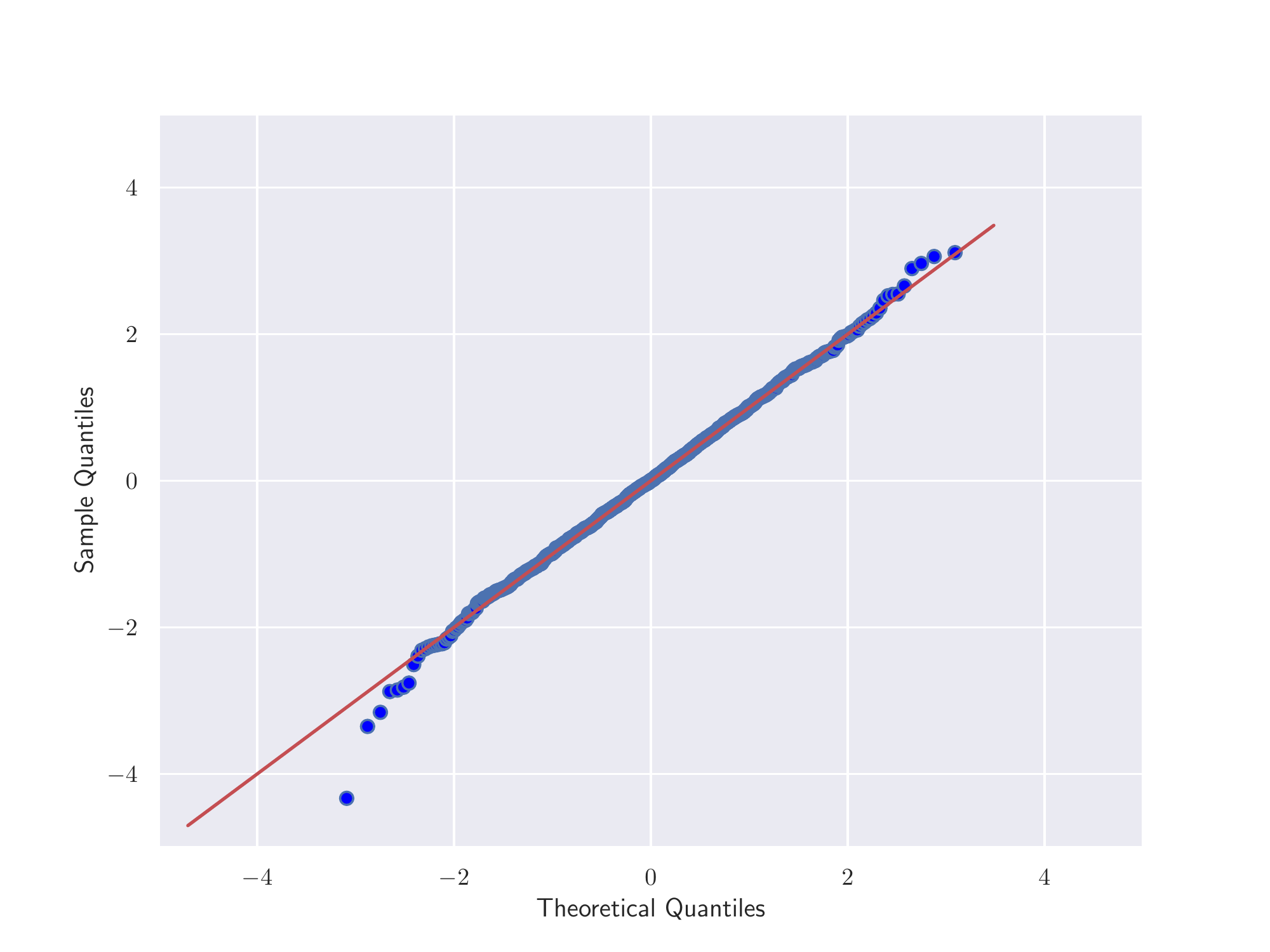}   \\ \bottomrule
\end{tabular}
\caption{Averages and standard errors over 1000 simulations
    of $(\hat n,\df,|\hat S|, \hat{V}^{1/2}n^{-1/2})$, 
    as well as histograms and QQ-plots for $\Omega_{jj}^{1/2}\xi_j'$
    given in \eqref{xi_j-prime}. 
    The coordinate $j$ is always $j=1$.
    The noise $\varepsilon_i$ is set as iid standard Cauchy.
    The red line on the QQ-plots is the diagonal line with equation $x=y$.
}
\label{table:simulation_cauchy}
\end{table}

\begin{table}[p]
    \centering
\begin{tabular}{@{}|c|c|c|c|c|@{}}
\toprule
$(\lambda,\tau)$           & $(n^{-1/2},0.1)$  & $(n^{-1/2},0)$ & $(2n^{-1/2},0.1)$ & $(2n^{-1/2},0)$ \\ \midrule
$ \hat{n} $                & $90.0 ~~\pm 5.8$ & $125.7 ~~\pm 6.4$ & $42.5 ~~\pm 4.8$ & $55.4 ~~\pm 5.7$ \\\midrule
$ \df $                    & $57.8 ~~\pm 3.8$ & $82.1 ~~\pm 4.9$  & $16.5 ~~\pm 2.4$ & $28.1 ~~\pm 4.3$ \\\midrule
$ |\hat S| $               & $97.8 ~~\pm 6.5$ & $82.1 ~~\pm 4.9$  & $29.7 ~~\pm 4.5$  & $28.1 ~~\pm 4.3$ \\\midrule
$ \sqrt{\hat{V} / n} $     & $0.37 ~~\pm 0.04$  & $0.25 ~~\pm 0.04$   & $0.51 ~~\pm 0.08$  & $0.48 ~~\pm 0.09$  \\\midrule
$\Omega_{jj}^{1/2}\xi_j'$ &
\includegraphics[width=23mm]{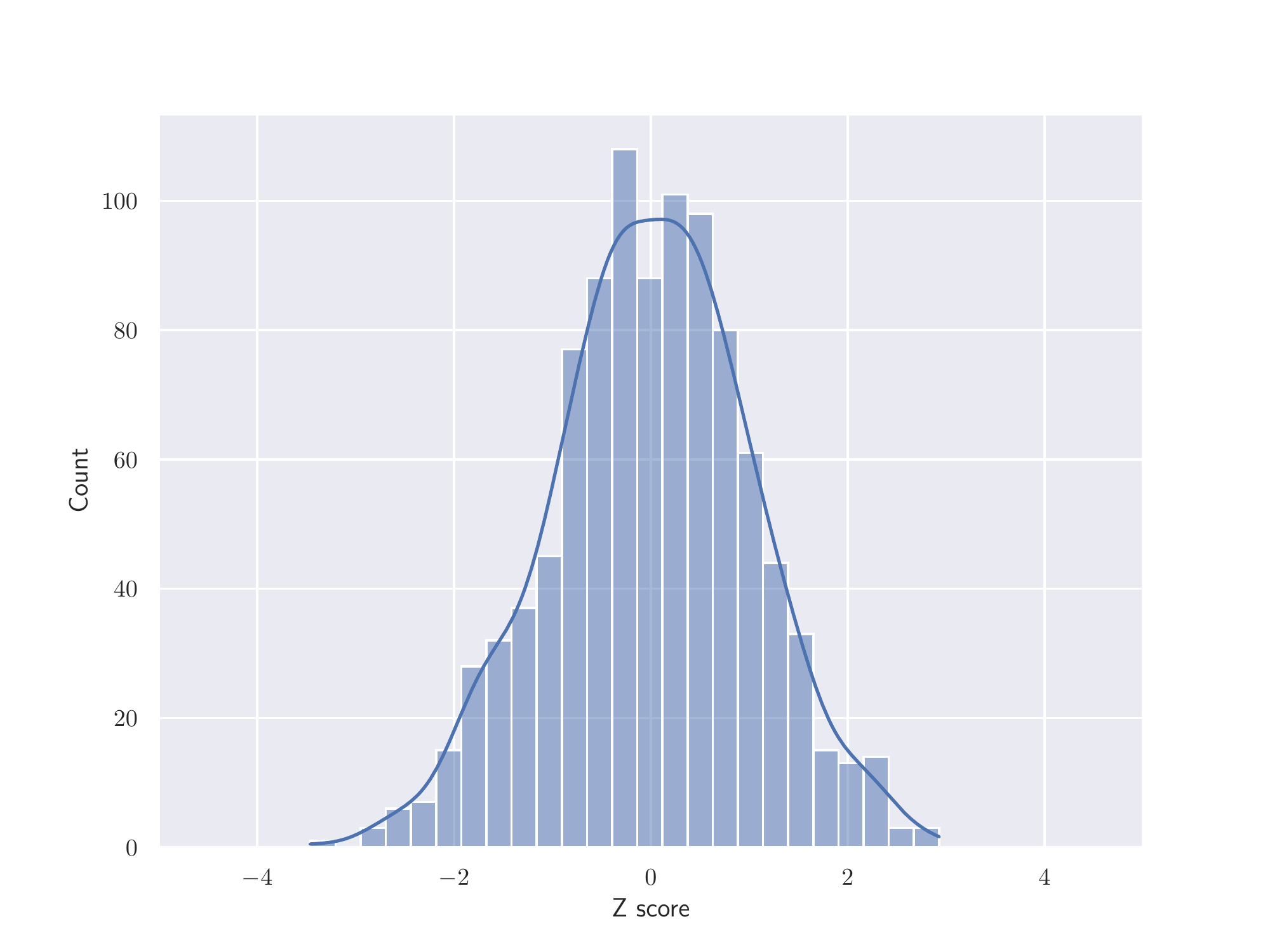}       &
\includegraphics[width=23mm]{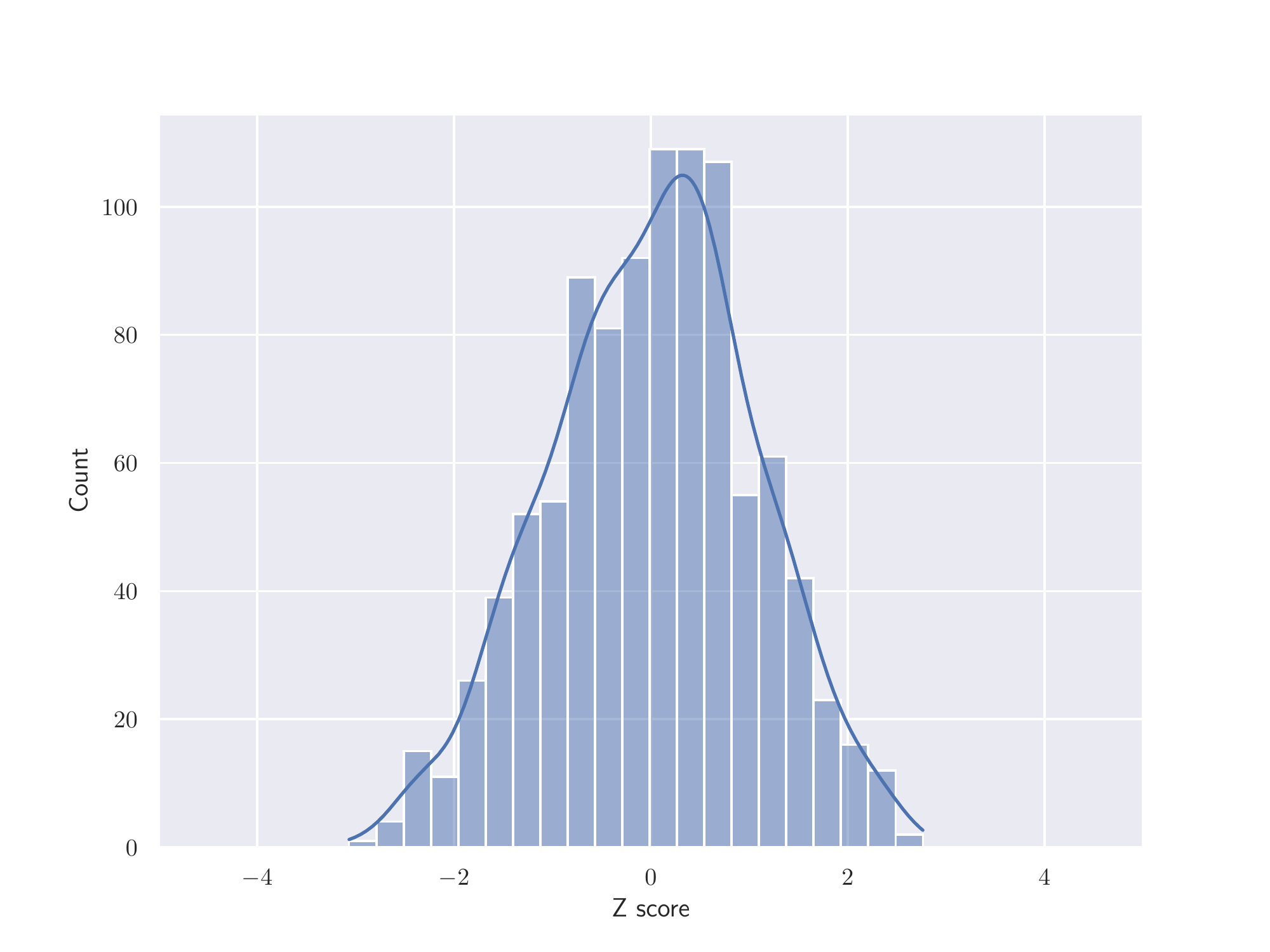}       &
\includegraphics[width=23mm]{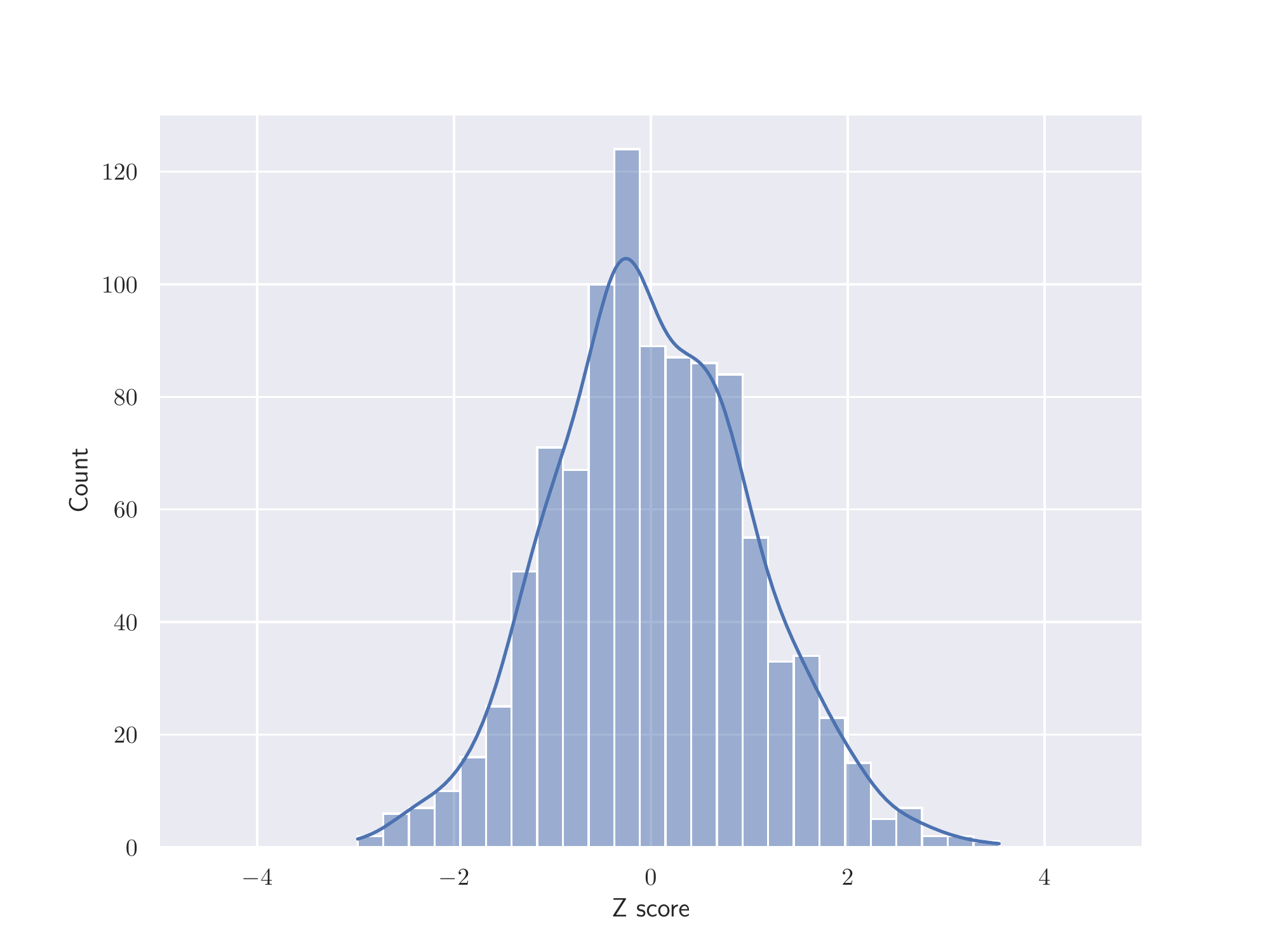}     &
\includegraphics[width=23mm]{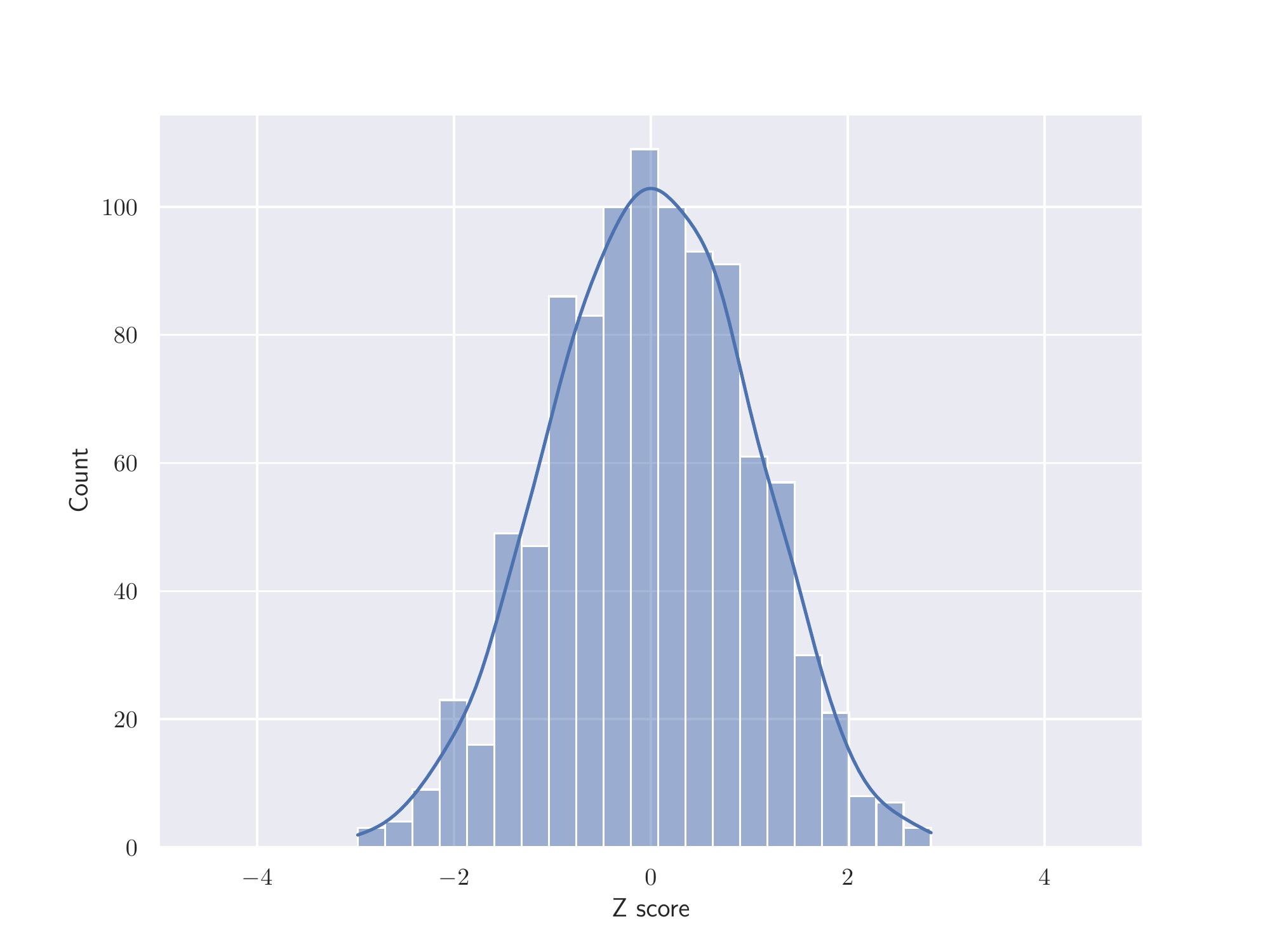} \\  &
\includegraphics[width=23mm]{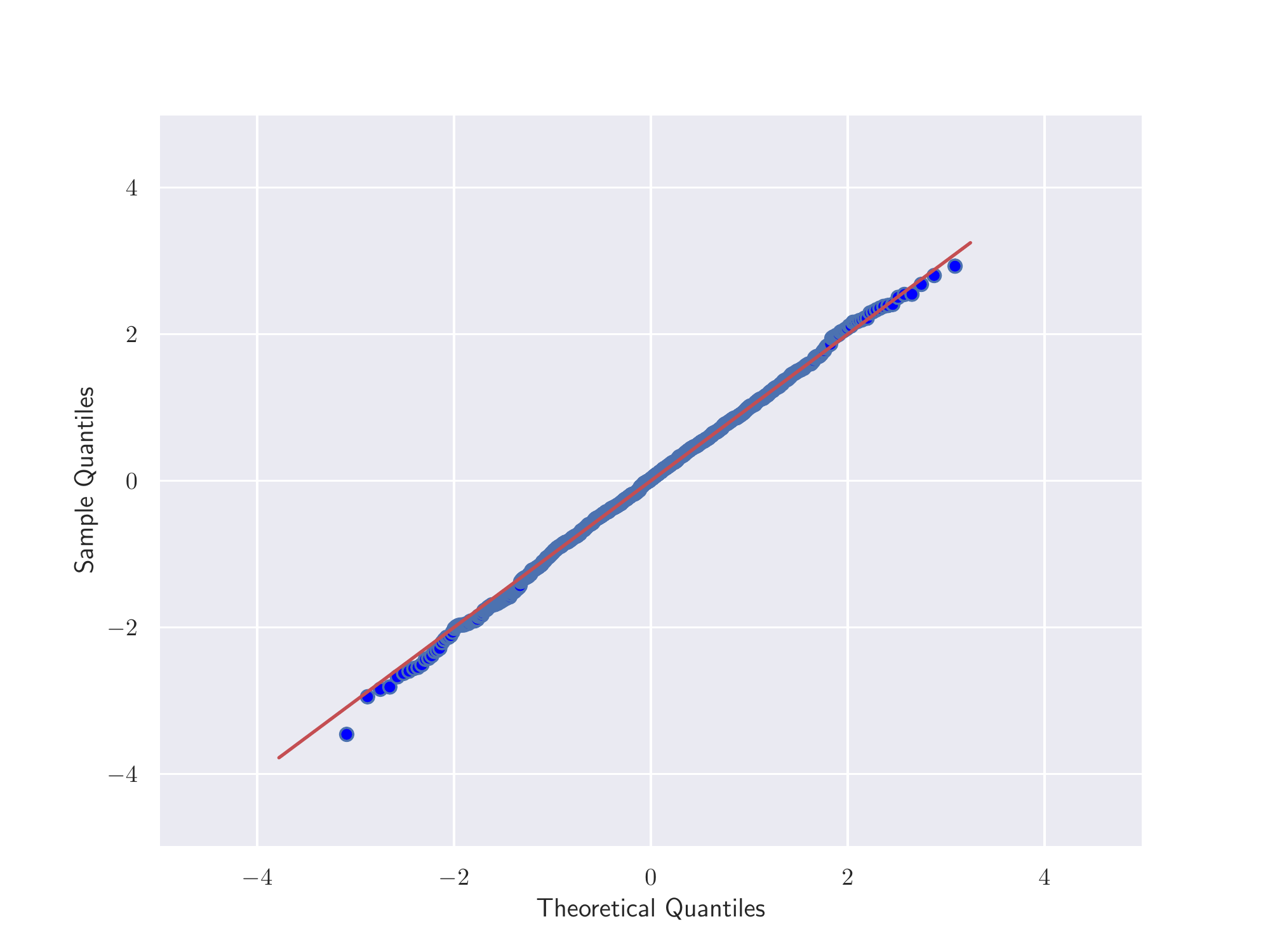}         &
\includegraphics[width=23mm]{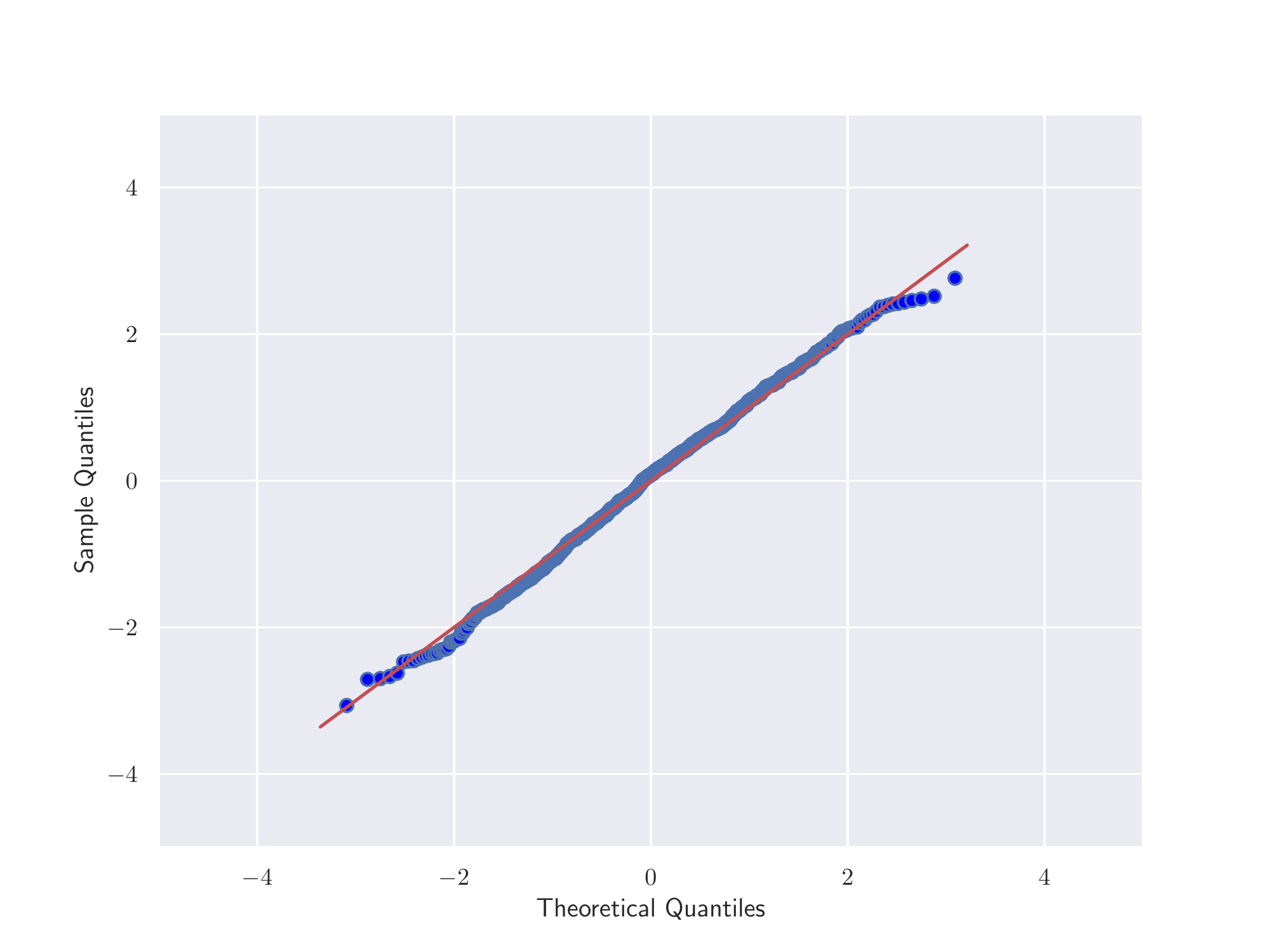}         &
\includegraphics[width=23mm]{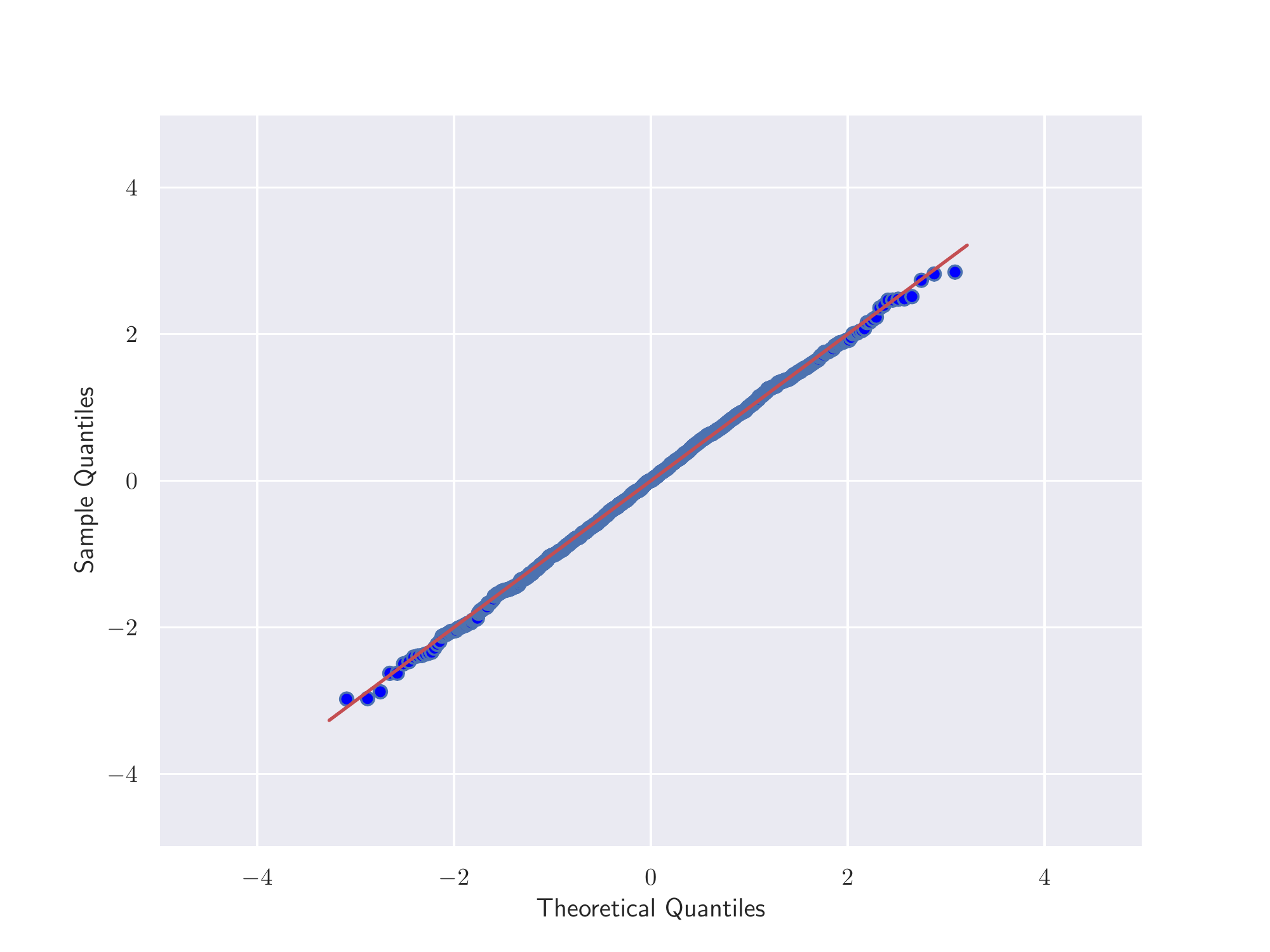}       &
\includegraphics[width=23mm]{qq_t2_0_2_n__-0.5.pdf}       \\ \bottomrule
\end{tabular}
\caption{
    Averages and standard errors over 1000 simulations
    of $(\hat n,\df,|\hat S|, \hat{V}^{1/2}n^{-1/2})$, 
    as well as histograms and QQ-plots for $\Omega_{jj}^{1/2}\xi_j'$
    given in \eqref{xi_j-prime}. 
    The coordinate $j$ is always $j=1$.
    The noise $\varepsilon_i$ is set as iid t-distribution with degree of freedom 2.
    The red line on the QQ-plots is the diagonal line with equation $x=y$.
}
\label{table:simulation_t2}
\end{table}

\begin{figure}[p]
    \centering
    \includegraphics[width=0.75\textwidth]{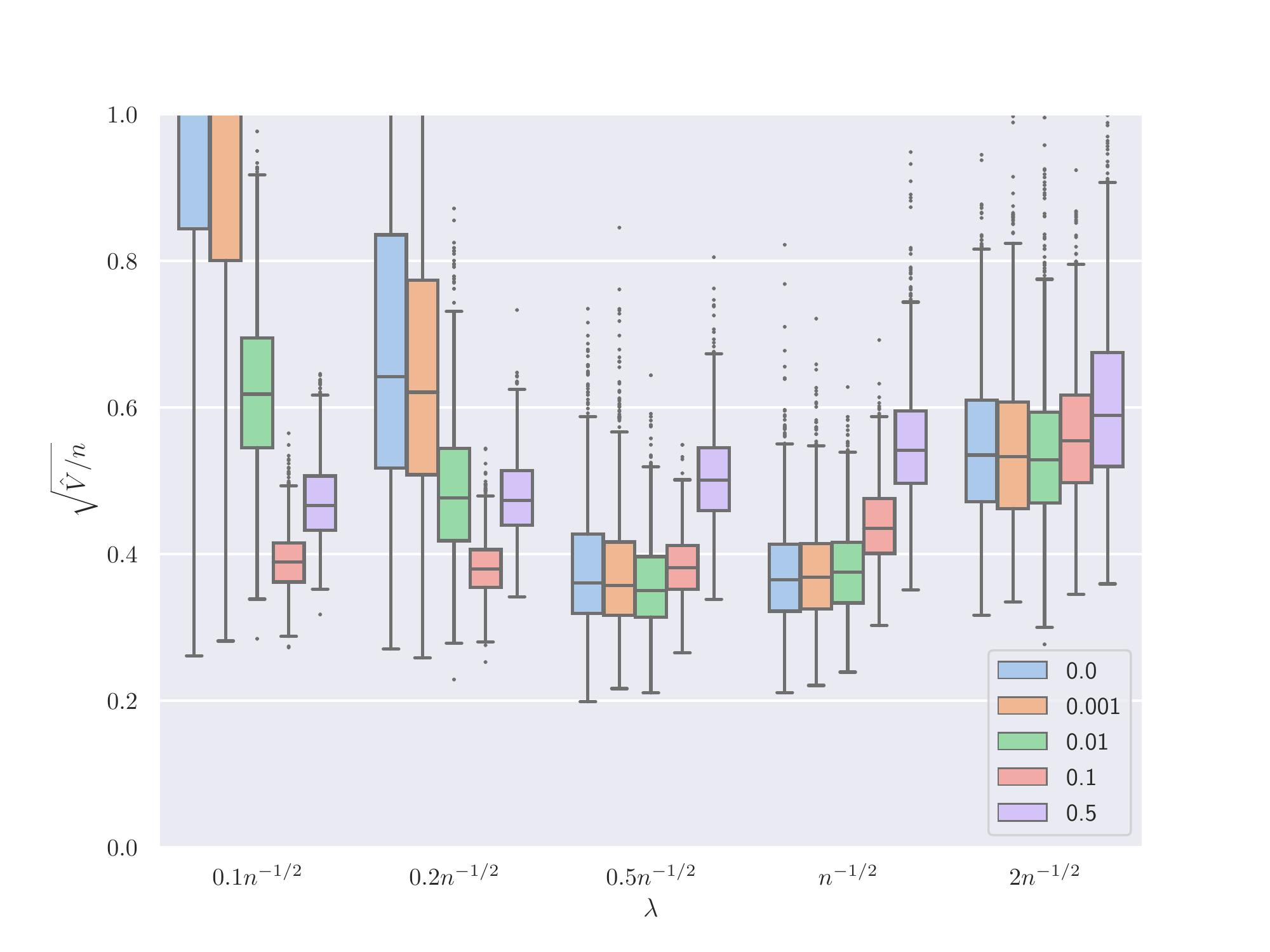}
    \includegraphics[width=0.75\textwidth]{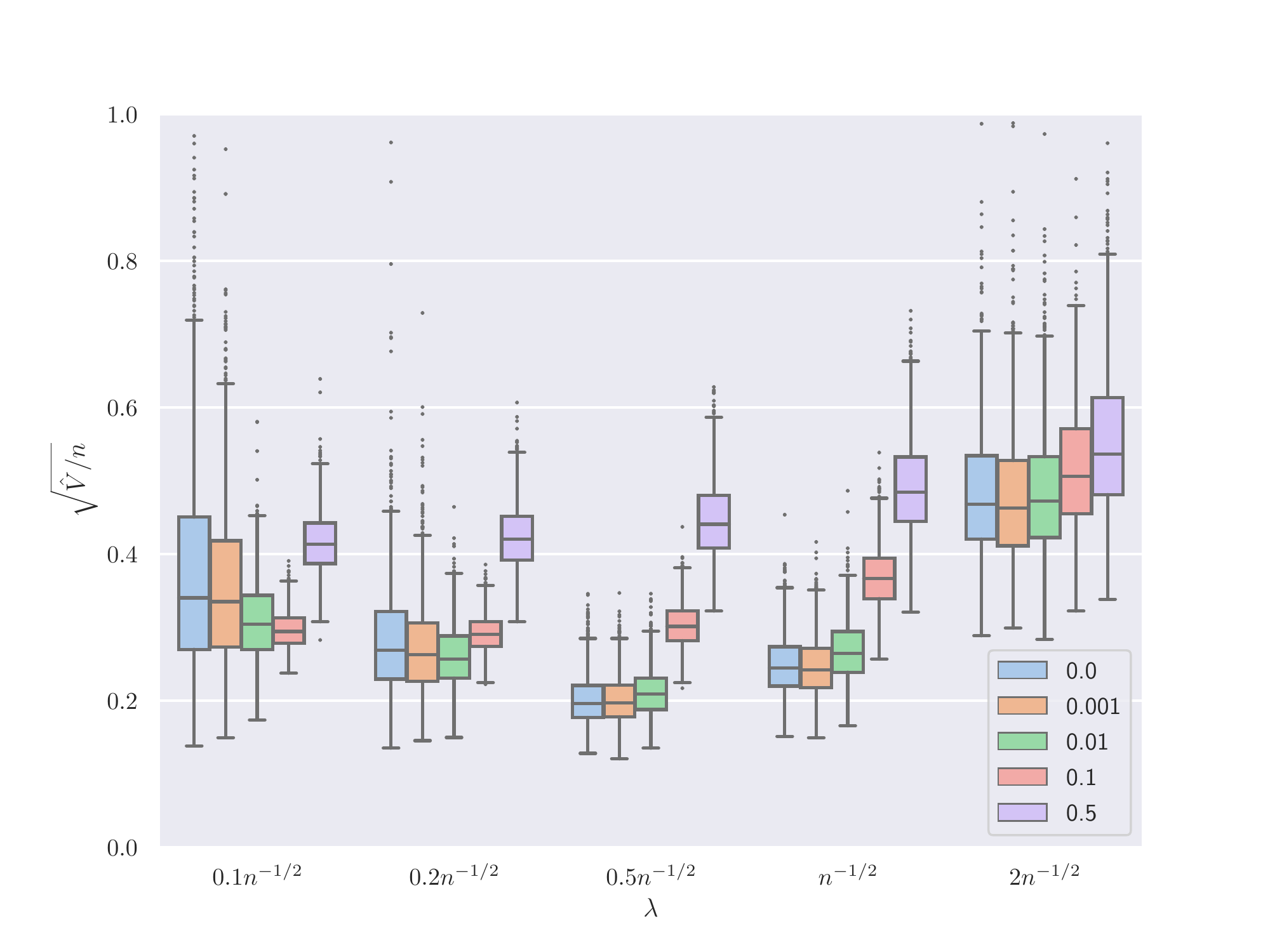}
    \caption{
        Boxplots of simulated $\hat{V}^{1/2} n^{-1/2}$.
        The simulation setup is described in \Cref{sec:simulation}.
        The top plot corresponds to iid standard Cauchy noise $\varepsilon_i$.
        The noise in the second plot is iid
        from the t-distribution with 2 degrees of freedom.
        Different colors correspond to different 
            values of  $\tau \in \{ 0, 10^{-3}, 10^{-2}, 0.1, 0.5 \}$.
            In the x-axis,
            $\lambda$ takes values in $\{0.1n^{-1/2}, 0.2n^{-1/2}, 0.5n^{-1/2}, n^{-1/2}, 2n^{-1/2} \}$.
    }
        \label{fig:boxplots}
\end{figure}

We set $n=200, p=300$, and generate $\veps_i\sim N(0,1)$ and the coordinates of
$\bbeta$ from iid Bernoulli variables with parameter $0.1$.
We compute 1000 simulations of the Z-score $\Omega_{jj}^{1/2}\xi_j'$
from \Cref{thm:main_thm}
for $j=1$ for the $M$-estimator with the Huber loss and the Elastic-Net
penalty \eqref{eq:huber-e-net} with $\sigma=1$
and the four combinations 
$(\lambda,\tau)\in\{n^{-1/2}, 2n^{-1/2}\} \times \{0, 0.1\}$.
The covariance matrix $\bSigma$ is set as $\bA^\top\bA/n$
where $\bA\in\R^{2p\times p}$ has iid Rademacher entries;
$\bSigma$ is generated once and is the same across the 1000 simulations.
The average value and standard error over the 1000 simulations
of $\hat n$, $\df$, $|\hat S|$ and $\hat{V}^{1/2} n^{-1/2}$ are presented in \Cref{table:simulation_cauchy} for $\bep$ with iid standard Cauchy components
and \Cref{table:simulation_t2} 
for $\bep$ with 
iid components from the t-distribution with $2$ degrees of freedom,
together with histograms and QQ-plots against standard normal quantiles
of $\Omega_{jj}^{1/2}\xi_j'$ in \eqref{eq:xi_j_prime-ENet-N01}.

The quantity $\hat{V}^{1/2} n ^{-1/2} = \| \bpsi \|_{2} / (\hat{n} - \df) $ 
featured in the boxplots of \Cref{fig:boxplots}
characterizes the length of our confidence intervals in 
\eqref{eq:normality}.
Computing the values $\hat V^{1/2}$ for different tuning parameters
lets the practitioner pick the tuning parameters leading to the
smallest confidence interval width, although this process
amounts to the construction of multiple confidence intervals
and warrants a Bonferroni multiple testing correction.

The histograms and QQ-plots 
in \Cref{table:simulation_cauchy} and \Cref{table:simulation_t2} 
confirm the normality of $\xi_j'$ for these two heavy-tailed continuous noise distributions.

Since $\hat n - \df$ appears in the denominator of $\hat V$,
the length of confidence intervals can be large 
if $\hat{n} - \df$ is nearly zero and the length is infinite
if $\hat n - \df = 0$.
This explains the large values and large variances
observed in the boxplots of \Cref{fig:boxplots}
for small tuning parameters.

\appendix

\section{Proof of the main result}

\subsection{Supporting propositions}

The proof of \Cref{thm:main_thm} relies on the two intermediary results
given below.
\Cref{prop:bound_ratio_lip} will be proved in \Cref{sec:stein}
and \Cref{lem:inter_step_zj} in \Cref{sec:bound-frobenius}.

\begin{restatable}{proposition}{propisitionSecondSteinCorollary} 
    \label[proposition]{prop:bound_ratio_lip}
	Let $n \geq 3$, $R>0$ and $\bz \sim \text{\rm Unif}(\mathbb{S}^{n-1}(R))$, 
        where $\mathbb{S}^{n-1}(R)$ is sphere of radius $R$ in $\R^n$.
        Assume either:
        \begin{itemize}
            \item
	    $\bff$ is locally Lipschitz on $\mathbb{S}^{n-1}(R)$ with $\E [\| \nabla \bff (\bz) ^\top\|_{F}^{2}]<\infty$.
            \item
	$\bff (\bz)$ is of the form $\tbf (\bz) / \| \tbf (\bz) \|_{2} $ where $\tbf:\R^{n} \to \R^{n}$ is locally Lipschitz and
        satisfies
        $\P(\|\bff(\bz)\|_2 \ne 0)=1$ and
        $\E [ \| \tbf (\bz) \|_{2}^{-2} \| \nabla \tbf (\bz) \|_{F}^{2} ] < + \infty $. 
        \end{itemize}
	Define $\bP_{\bz}^\perp = \bI_n - \bz \bz^{\top} / \| \bz \|_2^2$
and $\xi_{\bff}(\bz) = \bff (\bz) ^\top \bz - R^2 n^{-1} \trace ( \nabla \bff (\bz)^\top \bP_{\bz}^{\perp})$. 
	Then 
\begin{align} \label{eq:inequality_1}
			\E 
			\bigl[ 
			\bigl(
			\xi_{\bff}(\bz)
			- 
			\E [ \bff (\bz) ]^{\top} \bz
			\bigr)^{2}
			\bigr]
		& \leq 
			2 R^4 (n^2 - 2n)^{-1} \E [\| \nabla \bff (\bz) ^\top \bP_{\bz}^{\perp} \|_{F}^{2}],
\\
	\E \bigl[ \big| \| \bff (\bz) \|_{2} - \| \E [ \bff (\bz) ] \|_{2} \big| \bigr]^{2}
& \leq
	R^2 (n-2)^{-1} \E [ \| \nabla \bff (\bz)^\top \bP_{\bz}^{\perp} \|_{F}^{2} ]
.
\label{eq:inequality_2}
\end{align}

\end{restatable}

\begin{restatable}{lemma}{lemmaBoundFrobenius} \label[lemma]{lem:inter_step_zj}
    Let Assumptions~\ref{as:rho}, \ref{as:g} and \ref{as:feature} be fulfilled.
Let $\eta_{n} = \sqrt{2 \log (n) / n} + n^{-1/2}$,  
\begin{align*}
			\mathscr{E}_{j} 
		= 
			\{	\| \bX \bQ_{j} \bSigma^{-1/2} n^{-1/2} \|_{\op} 
			\leq 
				1 + \sqrt{ p / n } + \eta_{n}	
			\} 
		\cap 
			\{  \big| n^{-1/2} \Omega_{jj}^{1/2} \| \bz_{j} \|_{2} - 1 \big|
			\leq
				\eta_{n}
			\},\ j\in [p], 
	\end{align*}
	$u_{*}=\valueOfUstar$, and 
$\E_{j}$ be 
the conditional expectation
given $(\| \bz_{j} \|_{2},\bX \bQ_{j},\bep)$. 
Then, when $ u_* > 0$
	\begin{equation} \label{eq:5221} 
			\sup_{\delta>0} 
			\E\Bigl[
                            \frac{n}{\|\bpsi\|_{2}^{2} +\delta}
				\sum_{j\in[p]}I_{\mathscr{E}_{j}}h_{j}^{2}
                            \Bigr]^{1/2}
		\leq 
			\frac{
                            \bigl[(1+\sqrt{\frac p n})^2+\frac1n\bigr]^{1/2}
                        }{\phi_{\min}(\bSigma)^{1/2}( 1 - \eta_{n} )_{+}^{2}  u_*}
		+
                \frac{\bigl[\frac{2p}{n\tau} + \E[\|\bh\|_2^2]\bigr]^{1/2}}{u_*^{1/2}}. 
	\end{equation} 
\end{restatable}

\subsection{Proof of the main result}
\label{sec:proof_main}

\begin{proof}[of Theorem~\ref{thm:main_thm}]
	Since $\bpsi \neq \bzero_{n}$ for almost every $\bX$
	by Proposition \ref{prop:psi_not_zero},
	$\xi_{j}$ is well-defined with $\P$-probability 1.
	By \Cref{lem:inter_step_zj} and the monotone convergence theorem
        as $\delta\to 0$ for the left-hand side \eqref{eq:5221}, when $u_*>0$  
			\[
				\E\Bigl[
				\frac{n}{\|\bpsi\|_{2}^{2}}
				\sum_{j\in[p]}I_{\mathscr{E}_{j}}h_{j}^{2}
                            \Bigr]
			\leq
			\frac{
                            \bigl[(1+\sqrt{\frac p n})^2+\frac1n\bigr]^{1/2}
                        }{\phi_{\min}(\bSigma)^{1/2}( 1 - \eta_{n} )_{+}^{2}  u_*}
		+
                \frac{\bigl[\frac{2p}{n\tau} + \E[\|\bh\|_2^2]\bigr]^{1/2}}{u_*^{1/2}}. 
			\]
	Under our assumptions, $\| \bSigma \|_{\rm op} \leq {1/\kappa} < +\infty$
	and $\E [\| \bSigma^{1/2} \bh \|_{2}^{2}] \leq \mathscr{R} < +\infty$,
        so that
	there exists some finite constant $\mathscr{N} > 0$ and $\mathscr{A} < +\infty$ 
	independent of $n, p$, 
	such that for $n \geq \mathscr{N}$,
	\begin{equation} \label{eq:main_eq_2}
                \sum_{j \in [p]} 
		\E \Big[ n \| \bpsi \|_{2}^{-2} I_{\mathscr{E}_{j}} h_{j}^{2} \Big]
		\leq \mathscr{A} < +\infty.
	\end{equation}
        By Markov's inequality with respect to the
        uniform distribution on $[p] = \{1,...,p\}$, the set
        \begin{equation}
            \label{J_p}
            J_{n,p} \coloneqq
            \bigl\{j\in [p]:
            \E \bigl[ n I_{\mathscr{E}_{j}} \| \bpsi \|_{2}^{-2} h_{j}^{2} \bigr] \leq \mathscr{A} / a_{p} 
            \bigr\}
            \quad
            \text{ satisfies }
            \quad
	|J_{n,p}| / p \geq 1 - a_{p} / p.
        \end{equation}

		In this paragraph,
            for a given, fixed value of $(\bep,\bX\bQ_j)$
            we view $\bpsi$ as a function
        of $\bz_j$,
        \begin{equation}
        \bpsi: \R^n \to \R^n,
        \qquad
        \bpsi: \bz_j \mapsto \bpsi(\bz_j)=\bpsi(\bep,\bX\bQ_j+\bz_j\bfe_j^\top) 
        \label{bpsi-bz_j}
        \end{equation}
        instead of \eqref{hbpsi-y}, 
        and we denote its Jacobian
        by $\nabla \bpsi(\bz_j)^\top$ at any point $\bz_j$
        where \eqref{bpsi-bz_j} is Fr\'echet differentiable.
        Next, we 
        argue conditionally on $(\bX\bQ_j,\|\bz_j\|_2)$:
        Since $\bz_j/\|\bz_j\|_2$ is independent of $(\|\bz_j\|_2,\bX\bQ_j)$
        and by rotational invariance of the Gaussian distribution,
        conditionally on $(\|\bz_j\|_2,\bX\bQ_j)$ the vector
        $\bz_j$ is uniformly distributed on the sphere
        $\mathbb{S}^{n-1}(\|\bz_j\|_2)$.
	Let $\E_{j}$ denote the conditional expectation of $\bz_{j}$
        given $( \bX \bQ_{j}, \| \bz_{j} \|_2 ) $. 
        By Proposition \ref{prop:Lipschitz_in_X}, 
	the above function is locally Lipschitz.
	From Proposition \ref{prop:Lipschitz_in_X} (iii)
	we have that
	\begin{equation} \label{eq:5098}
			I_{\mathscr{E}_{j}}
			\E _{j} [
                        \| \bpsi \|_{2}^{-2} \| \nabla \bpsi(\bz_{j})\| _{F}^{2}
			]
			\leq
			L (n\tau)^{-1}
			+
			I_{\mathscr{E}_{j}}
			\E_{j} [
			nL^{2} \| \bpsi \|_{2}^{-2} h_{j}^{2} ]
	\end{equation}
        and the right-hand side is finite with probability one
        with respect to $(\bX\bQ_j,\|\bz_j\|_2)$
        thanks to \eqref{eq:main_eq_2} and
	Tonelli's theorem for non-negative measurable functions.
	By Proposition \ref{prop:psi_not_zero}, 
	conditional on almost every $(\bX\bQ_{j}, \|\bz_{j}\|_2)$, 
	$\bpsi (\bz_{j}) \neq \bzero_{n}$ for almost every $\bz_{j} \in \R$.
	We are now in position to apply \Cref{prop:bound_ratio_lip}
	with $\bff = \bpsi / \| \bpsi \|_2$, $\bz = \bz_{j}$ and
	\[
		\xi_{\bff}(\bz_{j})
                := \| \bpsi \|_2^{-1}
                \bigl(\bpsi ^{\top} \bz_{j} 
		- 
                (\| \bz_{j} \|_2^2/n)
		\trace \bigl[\bP_{\bpsi}^\perp 
		(\nabla \bpsi (\bz_{j}))^{\top} \bP_{\bz_j}^\perp
		\bigr]
		\bigr),
	\]
        where $\bP_{\bv}^\perp:=\bI_n - \bP_{\bv}$ with $\bP_{\bv} := \bv\bv^\top/\|\bv\|_2^2$ for any $\bv \in \R^{n}$. By 
        \Cref{prop:bound_ratio_lip}
        and \eqref{eq:5098}, 
	\[ 
		\begin{aligned}
		& I_{\mathscr{E}_{j}}
		\E_{j} [ ( \xi _{\bff} ( \bz_{j} ) - \E_{j} [ \bff (\bz_{j}) ]^{\top} \bz_{j} )^{2} ]
		\\
                    &\leq 
		2
		I_{\mathscr{E}_{j}}
		\| \bz_{j} \|_{2}^{4} (n^{2}-2n)^{-1} 
		\E_{j} [ \| ( \nabla \bff (\bz_{j}) )^{\top} 
		\bP_{\bz_j}^\perp 
		\|_{F}^{2} ]
	\\
		&\leq 
                2
                (1 + \eta_{n})^{4}
                (1-2/n)^{-1} 
		I_{\mathscr{E}_{j}} \Omega_{jj}^{-2}\E_{j}[ \|\bpsi\|_{2}^{-2}\|\nabla\bpsi(\bz_{j})\|_{F}^{2}]
	\\
		&\leq
                2
                (1 + \eta_{n})^{4}
                (1-2/n)^{-1}
		\kappa^{-2} \bigl(
		L (n\tau)^{-1}
		+
		I_{\mathscr{E}_{j}}
		\E_{j} [
		nL^{2} \| \bpsi \|_{2}^{-2} h_{j}^{2} ]\bigr) 
		\end{aligned} 
	\]
	and 
	\[ 
		\begin{aligned}
                &I_{\mathscr{E}_{j}}
		(\|\E_{j}[\bff(\bz_{j})]\|_{2}-1)^{2}
		\\&\leq
		I_{\mathscr{E}_{j}}
		\| \bz_{j} \|_{2}^{2}
		(n-2)^{-1}\E_{j}[\|(\nabla\bff(\bz_{j}))^{\top}
		\bP_{\bz_j}^\perp 
		\|_{F}^{2}].
	\\
		&\leq
                (1 + \eta_{n})^{2}
                (1-2/n)^{-1}
		I_{\mathscr{E}_{j}}
		\Omega_{jj}^{-1} \E _{j} [ \| \bpsi \|_{2}^{-2} \| \nabla \bpsi (\bz_{j}) \|_{F}^{2}] 
	\\
		&\leq 
                (1 + \eta_{n})^{2}
                (1-2/n)^{-1}
		\kappa^{-1} \bigl(
		L (n\tau)^{-1}
		+
		I_{\mathscr{E}_{j}}
		\E_{j} [
		nL^{2} \| \bpsi \|_{2}^{-2} h_{j}^{2} ]
                \bigr),
		\end{aligned} 
	\]
	where the upper bounds follow from 
	$
		I_{\mathscr{E}_{j}} \| \bz_{j} \|_{2}^{4}
		\leq 
		(1 + \eta_{n})^{4} n^{2} \Omega_{jj} ^{-2}
	$.
	Taking $\E$ on both sides, we obtain 
        $\max_{j\in J_{n,p}} \E\bigl[I_{\mathscr{E}_j}(\xi_{\bff} (\bz_{j}) - \E_{j} [ \bff ( \bz_{j} ) ]^{\top} \bz_{j})^2 
        + I_{\mathscr{E}_j}|\| \E_{j} [ \bff (\bz_{j}) ] \|_2 - 1|\bigr]\to 0$.
        Thanks to $\min_{j\in[p]}\P(\mathscr{E}_j) \to 1$
        by \Cref{lem:high_p_Ej}
        this implies that both
        $|\| \E_{j} [ \bff (\bz_{j}) ] \|_2 - 1|$
        and
        $|\xi_{\bff}(\bz_j)^\top\bz_{j}) - \E_{j} [ \bff ( \bz_{j} ) ]^{\top} \bz_{j}|$
        converge in probability to 0 
	uniformly over $j \in J_{n,p}$.

        We now study the asymptotic distribution of $\xi_{\bff}(\bz_{j})$. 
	Since $\bz_{j} \sim N(\bzero_{n}, \Omega_{jj}^{-1} \bI_{n})$ is independent with $\bX\bQ_{j}$ by Proposition \ref{prop:indep_zj},
	without loss of generality, 
	we can assume that 
	$\bz_{j} = \| \bz_{j} \|_{2} \bzeta_{j} / \| \bzeta_{j} \|_{2} $
	for some $\bzeta_{j} \sim N( \bzero_{n}, \bI_{n})$ independent of 
        $(\bX\bQ_{j}, \|\bz_{j}\|_{2})$.
	Then $\E_{j}$ coincides with the conditional expectation of $ \bzeta_{j}$ given $(\bX\bQ_{j}, \|\bz_{j}\|_{2})$.
	After some rearrangement, 
	\begin{equation*} 
	\begin{aligned}
		\xi_{\bff} (\bz_{j})
		= 
		\xi_{\bff} (\bz_{j})
		-
		\E_{j}[\bff(\bz_{j})]^{\top}\bz_{j}
		+
                \|\E_{j}[\bff(\bz_{j})]\|_{2}\frac{\| \bz_{j} n^{-1/2} \|_{2}}{\|\bzeta_{j} n^{-1/2}\|_{2}}\Bigl(\frac{\E_{j}[\bff(\bz_{j})]}{\|\E_{j}[\bff(\bz_{j})]\|_{2}}\Bigr)^\top \bzeta_{j}.
	\end{aligned}
	\end{equation*}
	Uniformly over $j \in J_{n,p}$, we have
	(i)
	$|\xi_{\bff} (\bz_{j}) - \E_{j} [ \bff ( \bz_{j} ) ]^{\top} \bz_{j}|
        \to 0 $
	and 
	$\| \E_{j} [ \bff (\bz_{j}) ] \|_2 \to 1$ in probability,
	(ii)
	$
	\frac{ \E_{j} [ \bff (\bz_{j}) ]^{\top} } { \| \E_{j} [ \bff (\bz_{j}) ] \|_{2} } \bzeta _{j}
        \sim N(0,1)
	$
	and 
	(iii)
	$\| \Omega_{jj}^{1/2} \bz_{j} \|_2 / \| \bzeta_{j} \|_{2}\to 1$
	in probability. 
	By Slutsky's Theorem,
	$ \Omega_{jj}^{1/2} \xi_{\bff} (\bz_{j})$
	converges in distribution to $N(0,1)$ uniformly over $j \in J_{n,p}$.
	That is, for any $t \in \R$,
        $
		\max_{j \in J_{n,p}}	
		|\P (\Omega_{jj}^{1/2} \xi_{\bff} (\bz_{j}) \leq t) - \Phi (t)|
		\to 0.
        $

        It remains to relate $\xi_{\bff}(\bz_j)$ to
        $\xi_j$ defined in \eqref{xi_j}.
        As the term $\|\bpsi\|_2^{-1}\bz_j^\top \bpsi$ 
        present in both $\xi_j$ and $\xi_{\bff}(\bz_j)$ cancel out,
        we have the decomposition
        \begin{align}
            \nonumber
 &               \|\bpsi\|_2 \|\bz_j\|_2^{-2} n \bigl(
                 \xi_j
                 -
                \xi_{\bff} (\bz_{j})
                \bigr)
\\		& =
                (\hbeta_j - \beta_j) \trace[\nabla_{\by}\bpsi^\top]
                +
                \trace (\bP_{\bpsi}^\perp  (\nabla \bpsi (\bz_{j}))^{\top} \bP_{\bz_j}^\perp )
                \label{inside-the-trace-0}
                \\
                &=
                \trace\bigl[\bigl(
                    (\hbeta_j - \beta_j)\nabla_{\by}\bpsi^\top
                    + \nabla \bpsi(\bz_j)^\top 
                \big)
            \bP_{\bpsi}^\perp
                \bigr]
              \\&\qquad\quad+
                  (\hbeta_j - \beta_j) \trace\bigl[ 
                \bP_{\bpsi} \nabla_{\by}\bpsi^\top 
                \bigr] 
            +
            \trace\bigl[\bP_{\bpsi}^\perp \nabla \bpsi(\bz_j)^\top \bP_{\bz_j} 
            \bigr].
            \label{eq:third-term-divergence}
        \end{align}
        The matrix inside the trace in \eqref{inside-the-trace-0}
        is zero thanks to \eqref{diff_ratio}. 
        It follows that only the two terms in \eqref{eq:third-term-divergence}
        remain, hence
        $$
        \E[I_{\mathscr{E}_j}
        (\xi_j - \xi_{\bff}(\bz_j))^2
        ]
        \le
        (2/n)
        \E\bigl[I_{\mathscr{E}_h}
        \|\bz_j\|_2^2
        \|\bpsi\|_2^{-2}
        \bigl(
        (\hbeta_j - \beta_j)^2\|\nabla_{\by}\bpsi\|_{\op}^2
        +
        \|\nabla \bpsi(\bz_j)\|_{\op}^2
        \bigr)
        \bigr].
        $$
        Since $I_{\mathscr{E}_j}\|\bz_j\|^2/n \le \Omega_{jj}^{-1}(1+\eta_n)^2$
        and
        $\| \nabla_{\by}\bpsi \|_{\rm op} \leq L$ by \Cref{prop:Lipschitz_in_X} (i) with $\bX = \tbX$, the first term is bounded from above
        by $L\Omega_{jj}^{-1}(1+\eta_n)^2 \E[
        I_{\mathscr{E}_j}
        \|\bpsi\|_2^{-2}
        h_j^2]$ which converges to 0 uniformly over $j\in J_{n,p}$
        by definition of $J_{n,p}$ in \eqref{J_p}.
        The second term also converges to 0 uniformly over $j\in J_{n,p}$
        thanks to \eqref{eq:5098}.
        Thus $\max_{j\in J_{n,p}}
        \E[I_{\mathscr{E}_j}(\xi_j - \xi_{\bff}(\bz_j))^2] \to 0$
        which implies that
        $|\xi_j - \xi_{\bff}(\bz_j)|\to 0$ in probability uniformly
        over $j\in J_{n,p}$ and Slutsky's theorem completes the proof
        of \eqref{eq:J_p-main-theorem} for $\xi_j$.

        To prove a similar result for $\xi_j'$, 
            it is enough to prove $\Omega_j^{1/2}|\xi_j' - \xi_j|\to^\P0$
            uniformly over $j\in J_{n,p}$ by Slutsky's theorem.
            As 
            $|\xi_j - \xi_j'| = \big|\trace[\nabla_{\by}\bpsi]
            \big|\|\bpsi\|_2^{-1}\big|h_j\big| ~ \big|\|\bz_j\|_2^2/n - \Omega_{jj}^{-1}\big|$
            and $|\trace[\nabla_{\by}\bpsi]|\le n L$, by the Cauchy-Schwarz
            inequality we find
            $$\E[I_{\mathscr{E}_j}|\xi_j - \xi_j'|]
            \le
            \Omega_{jj}^{-1}
            n L \E[I_{\mathscr{E}_j} h_j^2 \|\bpsi\|_2^{-2}]^{1/2}
            \E[(\Omega_{jj}\|\bz_j\|_2^2/n - 1)^2]^{1/2}
            .$$
            Since $\E[(\Omega_{jj}\|\bz_j\|_2^2/n - 1)^2]^{1/2} = \sqrt{2/n}$
            and $\Omega_{jj}\in [\kappa, 1/\kappa]$,
            the previous display converges to 0 uniformly over $j\in J_{n,p}$
            by definition of $J_{n,p}$ in \eqref{J_p}.
\end{proof}

\section{Stein formulae on the sphere}
\label{sec:stein}

The goal of this section is to prove \Cref{prop:bound_ratio_lip}
and to develop Stein formulae for random vectors $\bz$
uniformly distributed on the sphere. 

Let $\mathbb{S}^{n-1}(R)$ be the sphere in $\R^{n}$ 
with center $\bzero$ and radius $R>0$.
We say that $\bz$ is uniformly distributed in $\mathbb{S}^{n-1}(R)$
    and write
    $\bz \sim \text{\rm Unif}(\mathbb{S}^{n-1}(R))$
    if $\bz$ is equal in distribution to $R \bzeta /\|\bzeta\|_2$
    where $\bzeta\sim N(\mathbf{0},\bI_n)$.
We first develops Stein's formulae with respect to
    $\bz \sim \text{\rm Unif}(\mathbb{S}^{n-1}(R))$ 
    for functions $\bff: \bz \in \mathbb{S}^{n-1}(R) \mapsto \bff (\bz) \in \R^{n}$ in Sobolev spaces over $\mathbb{S}^{n-1}(R)$.

        We derive next Stein formulae for functions in Sobolev spaces
        over $\mathbb{S}^{n-1}(R)$.
        One possible construction
        of such Sobolev spaces is obtained by completion
        of the space of infinitely differentiable functions $\mathbb{S}^{n-1}(R)\to \R$ with respect to the desired Sobolev norm as follows.
        Here, $\mathbb{S}^{n-1}(R)$ is viewed as a
        compact Riemannian manifold equipped with the canonical
        metric (the metric induced as a submanifold
        of $\R^{n}$ equipped with the Euclidean metric).
    As it will be convenient for compatibility with the rest of the 
    paper to conserve the partial derivatives
    with respect to the canonical basis in $\R^n$, we adopt the following
    notation.
    For a smooth function $h:\mathbb{S}^{n-1}(R)\to \R$
    and $\Omega\subset\R^n$ an open neighborhood of $\mathbb{S}^{n-1}(R)$,
    define the smooth function
    $\check h:\Omega\to\R$ by
    $\check h(\bx) = h(R \bx/\|\bx\|_2)$ and define the gradient
    of $h$ as that of $\check h$, i.e., for $\bx\in \mathbb{S}^{n-1}(R)$,
    set
    $\nabla h(\bx) = ((\partial/\partial x_{1})\check h(\bx), ...,(\partial/\partial x_{n})\check h(\bx))$.
    For every $\bx\in\mathbb{S}^{n-1}$,
    the gradient $\nabla h(\bx)$ belongs to the hyperplane orthogonal to $\bx$
    which is the tangent space of $\mathbb{S}^{n-1}$ at $\bx$.
    Furthermore if $\bgamma:\R\to\mathbb{S}^{n-1}(R)$ is a smooth
    curve with $\bgamma(0) = \bx, (d/dt)\bgamma(t)|_{t=0} = \bv$ then
    $\bv^\top\bx = 0$ and $\nabla h(\bx)^\top\bv = (d/dt) h(\bgamma(t)) |_{t=0}$.
    For such smooth function $h$ and equipped with its gradient,
    for $\alpha\in\{1,2\}$
    we define the Sobolev norm
    $$
    \|h\|_{1,\alpha}
    =
    \E\bigl[|h(\bz)|^\alpha\bigr]^{1/\alpha}
    +
    \E\bigl[\|\nabla h(\bz)\|_2^\alpha\bigr]^{1/\alpha},
    \qquad
    \bz \sim \text{\rm Unif}(\mathbb{S}^{n-1}(R))
    $$
    and the Sobolev space $W^{1,\alpha}(\mathbb{S}^{n-1}(R))$ as the completion
    of the space of smooth functions over $\mathbb{S}^{n-1}(R)$
    with respect to the above norm.
    This definition is equivalent to the definition
    given in \cite[Definition 2.1]{hebey1996sobolev}.
    By \cite[Proposition 2.3]{hebey1996sobolev},
    since the manifold $\mathbb{S}^{n-1}(R)$ is compact
    the resulting 
    Sobolev spaces do not depend on the chosen metric.
    Equivalently, the Sobolev space
    $W^{1,\alpha}(\mathbb{S}^{n-1}(R))$ is also
    the completion with respect to the above norm
    of the space of once continuously differentiable functions
    on $\mathbb{S}^{n-1}(R)$.

    If $h$ is locally Lipschitz on $\mathbb{S}^{n-1}(R)$ (i.e., every point has a neighborhood in $\mathbb{S}^{n-1}(R)$ on which $h$ is Lipschitz), then
   again by considering an open neighborhood $\Omega\subset\R^n$
   of $\mathbb{S}^{n-1}(R)$, the function 
   $\check h(\bx) = h(R\bx/\|\bx\|_2)$ is locally Lipschitz on $\Omega$. 
   Thus, in this case and by Rademacher's theorem, $\nabla h(\bz)$ is well defined almost everywhere in $\mathbb{S}^{n-1}(R)$ as the gradient of $\nabla \check h(\bx)$,
   and $h \in W^{1,\alpha}(\mathbb{S}^{n-1}(R))$ if and only if
   $\E\big[\|\nabla h(\bz)\|_{2}^\alpha\big]<\infty$. 
   For example, $h \in W^{1,\alpha}(\mathbb{S}^{n-1}(R))$ when $h$ is $L$-Lipschitz on $\mathbb{S}^{n-1}(R)$.  

   Finally, for $\alpha\in\{1,2\}$ define $W^{1,\alpha}(\mathbb{S}^{n-1}(R))^{n}$
   as the space of $\R^n$ valued functions $\bff=(f_1,...,f_n)$
   with all coordinates $f_i$ belonging to $W^{1,\alpha}(\mathbb{S}^{n-1}(R))$,
   equipped with the norm
   $$\|\bff\|_{1,\alpha} = \E[ \|\bff(\bz)\|_2^\alpha]^{1/\alpha} 
   + \E[\|\nabla \bff(\bz)\|_F^\alpha]^{1/\alpha}$$
   where the gradient $\nabla \bff$ is the matrix in $\R^{n\times n}$
   with columns $\nabla f_1,...,\nabla f_n$,

\begin{lemma} [Stein's formula on the sphere] \label[lemma]{lem:stein_on_sphere} 
	Let $n \geq 3$, $R>0$ and $\bz \sim \text{\rm Unif}(\mathbb{S}^{n-1}(R))$.
	Let $\bP_{\bx}^\perp = \bI_n - \bx \bx^{\top} / \| \bx \|_2^2$ for $\bx\neq {\bf 0}$.
        Then, for all $\bff = (f_1,\ldots,f_n) \in W^{1,1}(\mathbb{S}^{n-1}(R))^{n}$, 
	\begin{align}
	\E [ \bff(\bz)^{\top} \bz] = (n-1)^{-1}R^2
	\E [ \trace ( (\nabla \bff (\bz))^{\top}\bP_{\bz}^\perp)], 
	\label{eq:Sphere-First-Order-Stein} 
	\end{align}
	where $\nabla\bff = (\nabla f_1,\ldots,\nabla f_n)$. 
        Moreover, for all $\bff = (f_1,\ldots,f_n) \in W^{1,1}(\mathbb{S}^{n-1}(R))^{n}$, 
	\begin{align}
	& \E[\|\bff(\bz)-\E[\bff(\bz)]\|_2^2]
	\le
	(n-2)^{-1}R^2 
	\E[\|\nabla \bff(\bz)^\top\bP_{\bz}^\perp\|_F^2],
	\label{eq:Sphere-Poincare}
	\\
	\nonumber
	&\E\bigl[\bigl(nR^{-2}\bff(\bz)^{\top}\bz-\trace[
                \nabla \bff(\bz)^\top \bP_{\bz}^\perp]\bigr)^2\bigr] 
	\\&=
	nR^{-2}
		\E\bigl[\|\bff\|_{2}^{2}\bigr]
	+
	n(n-2)^{-1}
	\E \trace\bigl[
		\bigl(
		\nabla \bff(\bz)^\top\bP_{\bz}^\perp
		\bigr)^2
	\bigr]
	-
	2(n-2)^{-1}\E\bigl[\trace(\nabla f(\bz)^\top\bP_{\bz}^\perp)^{2}\bigr]
	\label{eq:Sphere-SOS-equality}
	\\&\le
	nR^{-2}\E\bigl[\|\bff\|_{2}^{2}\bigr]
	+
	(1-2/n)^{-1}
	\E\bigl[\|\nabla \bff(\bz)^\top\bP_{\bz}^\perp\|_F^2\bigr].
	\label{eq:Sphere-SOS-inequality}
	\end{align}
\end{lemma}
    Note that \eqref{eq:Sphere-Poincare} is the classical
    Poincar\'e inequality on the sphere.
    A proof is provided for completeness.

\begin{proof}[of Lemma \ref{lem:stein_on_sphere}] 
    As the operators $T$ and $T_{ij}$ defined by
    $T\bff(\bz) = \bff(\bz)^{\top}\bz$ and $T_{ij}\bff(\bz)=[\nabla \bff(\bz)^\top\bP_{\bz}^\perp]_{ij}$
    for every $i,j=1,...,n$ 
    are all continuous linear mappings 
    from $W^{1,\alpha}(\mathbb{S}^{n-1}(R))^{n}$
    to the $L_\alpha$ space with the norm 
    $\|f\|_{L_\alpha}=(\E[|f(\bz)|^\alpha])^{1/\alpha}$, 
    we assume without loss of generality that 
        all coordinates  of $\bff$
        are continuously differentiable.
    Indeed, if $(\bff^{(q)})_{q\ge 1}$ is a sequence of smooth functions
    over the sphere converging to $\bff$ in $W^{1,\alpha}(\mathbb{S}^{n-1}(R))^{n}$ for $\alpha=1$ and \eqref{eq:Sphere-First-Order-Stein} holds
    for $\bff^{(q)}$ then \eqref{eq:Sphere-First-Order-Stein} also holds
    for the limit by continuity; the same argument applies
    with $\alpha=2$ for \eqref{eq:Sphere-Poincare}-\eqref{eq:Sphere-SOS-equality}-\eqref{eq:Sphere-SOS-inequality}.

    Let $\bzeta\sim N(\mathbf{0},\bI_n)$. 
    Assume without loss of generality $\bz = R\bzeta/\|\bzeta\|_2$ 
    as $R\bzeta/\|\bzeta\|_2\sim \text{\rm Unif}(\mathbb{S}^{n-1}(R))$. 
    Let $\phi(t)$ be a continuously differentiable function in $\R$ with 
    $\phi(t)=0$ for $t\le 0$ and $\phi(t)=1$ for $t\ge 1$. 
    For $\delta >0$ define 
    $\bvarphi_\delta(\bx)=\phi(\|\bx\|_2/\delta)\bff(R\bx/ \|\bx\|_2)$. 
    As $\|\bzeta\|_2$ is independent of $\bz$ and $\bvarphi_\delta(\bx)$ has uniformly bounded gradient, 
    the first order Stein formula for $\bvarphi_{\delta}$ yields 
    \begin{align*}
    \E\big[\phi(\|\bzeta\|_2/\delta)\|\bzeta\|_2\big]\E [ \bff(\bz)^{\top} \bz] 
    = \E\big[R \bvarphi_\delta(\bzeta)^{\top} \bzeta\big] 
    = R\,\E\big[\trace(\nabla \bvarphi_\delta (\bzeta))\big]. 
    \end{align*}
    By the product and chain rules, $\nabla \bvarphi_\delta (\bx)$ is given by 
    $$
    \phi'(\|\bx\|_2/\delta)(\bx/(\|\bx\|_2\delta))\bff(R\bx/ \|\bx\|_2)^\top
    +\phi(\|\bx\|_2/\delta)(R/\|\bx\|_2)\bP_{\bx}^\perp\big(\nabla \bff(R\bx/ \|\bx\|_2)\big)
    $$ 
   with $\phi'(t)=(d/dt)\phi(t)$ and $\phi'(0)=\phi'(1)=0$. 
   As $\sup_{\delta>0, \bx\in\R^n} \|\bx\|_2\|\nabla \bvarphi_\delta (\bx)\|_{\rm F} \le C<\infty$ 
   and $\E[\|\bzeta\|_2^{-2}]=1/(n-2)<\infty$, the dominated convergence theorem gives 
    \begin{align*}
    \E [ \bff(\bz)^{\top} \bz] 
    = \frac{R\lim_{\delta\to 0}\E\big[\trace(\nabla \bvarphi_\delta (\bzeta))\big]}
    {\lim_{\delta\to 0}\E\big[\phi(\|\bzeta\|_2/\delta)\|\bzeta\|_2\big]}
        = \frac{R^2\E\big[1/\|\bzeta\|_2\big]\E\big[\trace(\nabla \bff(\bz)^\top\bP_{\bz}^\perp)\big]}
    {\E\big[\|\bzeta\|_2\big]}, 
    \end{align*}
    which yields \eqref{eq:Sphere-First-Order-Stein} due to 
    $\E[\|\bzeta\|_2]/\E[\|\bzeta\|_2^{-1}] = (n-1)$. 

    Next, as the exchange of limit and expectation is allowed when 
    $\bvarphi_{\delta}\to \bvarphi_{0+}=\bff$, 
    the Gaussian Poincar\'e inequality
    \cite[Theorem 1.6.4]{Bog98}
    yields
    \begin{align*}
    \E[ \|\bff(\bz)-\E[\bff(\bz)]\|_2^2]
    \le
    \lim_{\delta\to 0+}\E[\|\nabla \bvarphi_{\delta}(\bzeta) \|_F^2]=
    \E[R^2 \|\nabla \bff(\bz)^\top\bP_{\bz}^\perp\|_F^2\|\bzeta\|_2^{-2}].
    \end{align*}
    Since
    $\E[
    \|\bzeta\|_2^{-2}
    ]=1/(n-2)$ and
    $(\|\bzeta\|_2,\bz)$ are independent,
    we obtain \eqref{eq:Sphere-Poincare}. 
    Finally by applying the Second Order Stein formula of
    \cite{BZ18} to $\bvarphi_\delta(\bzeta)$ we find
    by dominated convergence
    \begin{align*}
    &\E\bigl[
    \bigl(
        R^{-1}\|\bzeta\|_2
        \bz^\top \bff(\bz)
        -
        R\|\bzeta\|_2^{-1}\trace[\nabla \bff(\bz)^\top\bP_{\bz}^\perp]
    \bigr)^2
    \bigr]
    \\&=
    \lim_{\delta\to 0^+}
    \E\bigl[
    \bigl(
        \bzeta^\top \bvarphi_{\delta}(\bzeta)
        -
        \trace[\nabla \bvarphi_{\delta}(\bzeta)]
    \bigr)^2
    \bigr]
  \\&=
    \lim_{\delta\to 0^+}
    \E\bigl[\|\bvarphi_{\delta}(\bzeta)\|_2^2\bigr]
  +
  \E
  \trace\bigl[
  \bigl(\nabla \bvarphi_{\delta}(\bzeta)^\top\bigr)^2
  \bigr]
  \\&=
  \E\bigl[\|\bff(\bz)\|_2^2\bigr]
  +
  R^2
  \E
  \trace\bigl[
      \bigl(\nabla \bff(\bz)^\top\bP_{\bz}^\perp\bigr)^2
  \bigr]
  \E 
  \bigl[
  \|\bzeta\|_2^{-2}
  \bigr]
  \\&=
  \E\bigl[\|\bff(\bz)\|_2^2\bigr]
  +
  R^2
  \E
  \trace\bigl[
      \bigl(\nabla \bff(\bz)^\top\bP_{\bz}^\perp\bigr)^2
  \bigr]
  /(n-2), 
    \end{align*}
    where the last equality follows from the independence
    of $\bz$ and $\|\bzeta\|_2$.
    We now simplify the left-hand side in order to get rid
    of $\|\bzeta\|_2$. By expanding the square
    and using the independence of $(\bz,\|\bzeta\|_2)$,
    the above display is equal to
    \begin{align*}
        &R^{-2}\E[ \|\bzeta\|_2^2]
        \E[(\bz^\top \bff(\bz))^2]
        -
        2\E\bigl[\bz^\top \bff(\bz) \trace[\nabla \bff(\bz)^\top\bP_{\bz}^\perp]\bigr]
        +
        R^{2}\E[ \|\bzeta\|_2^{-2}]
        \E\bigl[(\trace[\nabla \bff(\bz)^\top\bP_{\bz}^\perp])^2\bigr]
        \\
        &=
        nR^{-2}
        \E[(\bz^\top \bff(\bz))^2]
        -
        2\E\bigl[\bz^\top \bff(\bz) \trace[\nabla \bff(\bz)^\top\bP_{\bz}^\perp]\bigr]
        +
        R^{2}(n-2)^{-1}
        \E\bigl[(\trace[\nabla \bff(\bz)^\top\bP_{\bz}^\perp])^2\bigr]
      \\&=
        \E\Bigl[
        \Bigl(
            \frac{\sqrt n}{R}
            \bz^\top \bff(\bz)
            -
            \frac{R}{\sqrt n}\trace[\nabla \bff(\bz)^\top\bP_{\bz}^\perp]
        \Bigr)^2
        \Bigr]
        + R^2\Bigl(\frac{1}{n-2} - \frac{1}{n}\Bigr)
        \E\Bigl[
        \big(\trace[\nabla \bff(\bz)^\top\bP_{\bz}^\perp]\big)^2
        \Bigr].
    \end{align*}
    Since
    $1/(n-2) - 1/n
    = 2n^{-1}(n-2)^{-1}$, we obtain
    \eqref{eq:Sphere-SOS-equality} after multiplying by $n/R^2$.
The proof is complete since \eqref{eq:Sphere-SOS-inequality}
follows directly from \eqref{eq:Sphere-SOS-equality}
by the Cauchy-Schwarz inequality.
    \end{proof}

\begin{proof}[of \Cref{prop:bound_ratio_lip}]
If $\bff$ is locally Lipschitz and $\E [ \|\nabla \bff (\bz) ^{\top}\|_{F}^{2} ]<\infty$, 
then $\bff\in W^{1,2}(\mathbb{S}^{n-1}(R))$.
We consider the mapping $\bff (\bz) - \E \bigl[ \bff (\bz) \bigr]$ rather than $\bff (\bz)$ in \eqref{eq:Sphere-SOS-equality} and \eqref{eq:Sphere-SOS-inequality}.
	Multiplying $R^4n^{-2}$ on both sides of the inequality \eqref{eq:Sphere-SOS-inequality}, it provides
	\begin{align*} 
			\E 
			[ 
			(
			\xi_{\bff}(\bz)
			- 
			\E [ \bff (\bz) ]^{\top} \bz
			)^{2}
			]
		& \leq
			R^2 n^{-1} \E [ \| \bff - \E [ \bff (\bz) ] \|_{2}^{2} ]
			+
			R^4 (n^2 - 2n)^{-1} 
			\E [ \| \nabla \bff (\bz) ^{\top} \bP_{\bz}^{\perp} \|_{F}^{2} ]
		\\
		& \leq 
			2 R^4 (n^2 - 2n)^{-1} \E [\| \nabla \bff (\bz) ^\top \bP_{\bz}^{\perp} \|_{F}^{2}],
	\end{align*}
where the second inequality follows from \eqref{eq:Sphere-Poincare}.
By the triangle inequality and \eqref{eq:Sphere-Poincare}, 
	\begin{align*} 
            \E [ | \| \bff (\bz) \|_{2} - \| \E [ \bff (\bz) ] \|_{2} |^2 ]
	& \leq
        \E [ \| \bff (\bz) - \E [ \bff (\bz) ] \|_2^2 ]
	\\
	& \leq
		R^2 (n-2)^{-1} \E [ \| \nabla \bff (\bz)^\top \bP_{\bz}^{\perp} \|_{F}^{2} ]
	.
	\end{align*}
If $\bff (\bz) = \tbf (\bz) / \| \tbf (\bz) \|_{2} $ with locally Lipschitz $\tbf$ on $\mathbb{S}^{n-1}(R)$ then  
$\tbf (\bz)/(\delta\vee \| \tbf (\bz) \|_{2})$ is locally Lipschitz for $\delta>0$ 
and converges to $\bff$ in $W^{1,2}(\mathbb{S}^{n-1}(R))^{n}$ as $\delta\to 0+$ 
when $\P(\|\bff(\bz)\|_2 \ne 0)=1$ and $\E [ \| \tbf (\bz) \|_{2}^{-2} \| \nabla \tbf (\bz) \|_{F}^{2} ] < + \infty$. 
Thus, the proof still applies.
\end{proof}


\section{Bounds on $(\hbeta_j-\beta_j)^2 \|\mathbf{\psi}\|_2^{-2}$}
\label{sec:bound-frobenius}

The goal of this section is to prove \Cref{lem:inter_step_zj}.

\subsection{High probability events $\mathscr{E}_j$}

\begin{lemma} 
    [high probability of $\mathscr{E}_{j}$]
    \label[lemma]{lem:high_p_Ej}
	Assume that $\bX$ has iid $N(\bzero, \bSigma)$ rows.
	Then 
        \begin{equation}
            \label{eq:bound-squared-operator-norm}
				\E [ \| \bX \bSigma^{-1/2} \|_{\op}^{2}] 
			\leq
				( \sqrt{n} + \sqrt{p} ) ^{2} + 1
		. 
        \end{equation}
	Furthermore, with $ \eta_{n} = \sqrt{2 \log (n) / n} + n^{-1/2}$
	and for 
	the events
	\begin{align*}
			\mathscr{E}_{j} 
		= 
			\{	\| \bX \bQ_{j} \bSigma^{-1/2} n^{-1/2} \|_{\op} 
			\leq 
				1 + \sqrt{ p / n } + \eta_{n}	
			\} 
		\cap 
			\{  \big| n^{-1/2} \Omega_{jj}^{1/2} \| \bz_{j} \|_{2} - 1 \big|
			\leq
				\eta_{n}
			\},
	\end{align*}
        we have
		$
				\P ( \cap_{j \in [p]} \mathscr{E}_{j} )
			\geq
				1 - (p + 1/2) n^{-1} (\pi \log(n))^{-1/2}
		$.

\end{lemma}

\begin{proof} [of Lemma \ref{lem:high_p_Ej}] 
	Let us first notice that $\bX\bSigma^{-1/2}$ is a random Gaussian matrix with iid standard normal entries.
	Theorem 7.3.1 in \cite{Ver18} provides that
	$\E [ \| \bX \bSigma^{-1/2} \|_{\op} ] \leq \sqrt{n} + \sqrt{p}$.
	Since the operator norm of a matrix is a 1-Lipschitz function of the entries of the matrix,
	by the Gaussian Poincar\'e inequality
        \cite[Theorem 3.20]{BLM13},
	$
		\Var ( \| \bX \bSigma^{-1/2} \|_{\op}  ) \leq 1.
	$
        This proves \eqref{eq:bound-squared-operator-norm}.

	By Theorem II.13 in \cite{DS01}, 
	we have for $ t > 0 $,
	\[
		\P ( \| \bX \bSigma^{-1/2} n^{-1/2} \|_{\op} \geq 1 + \sqrt{ p / n } + t ) 
	\leq \Phi ( - t \sqrt{n} ), 
	\]
	\[ 
		\P ( \big| \Omega_{jj}^{1/2} \| \bz_{j} n^{-1/2} \|_{2} - 1 \big| \geq n^{-1/2} + t )
		\leq
		2 \Phi ( - t \sqrt{n}).
	\]	
	Since $\bX \bQ_{j} \bSigma^{-1/2} = \bX \bSigma^{-1/2} ( \bSigma^{1/2} \bQ_{j} \bSigma^{-1/2} )$
	and $ \| \bSigma^{1/2} \bQ_{j} \bSigma^{-1/2} \|_{\op} \leq 1 $, 
	$ \| \bX \bQ_{j} \bSigma^{-1/2} \|_{\op} \leq \| \bX \bSigma^{-1/2} \|_{\op} $ for all $ j \in [p] $.
	With this we can conclude the above inequalities using a union bound over $ j \in [p] $	  when $t = \sqrt{ 2 \log (n) / n } $,  
	$ 
		\P ( \cap_{j \in [p]} \mathscr{E}_{j} )
		\geq 
		1 - (2p + 1) \Phi ( - \sqrt{ 2 \log (n) } ).
	$
	Thanks to 
        $ \Phi (-t) \leq (2\pi)^{-1/2} \exp ( - t^{2} / 2) / t $ for $t > 0$, 
	we have 
	\[ 
		\P ( \cap_{j \in [p]} \mathscr{E}_{j} )
		\geq
		1 - (p + 1/2) n^{-1} (\pi \log(n))^{-1/2}.
	\]
\end{proof}

\begin{remark}
	Our specific construction of $\mathscr{E}_{j}$ satisfies the following properties:

	\begin{enumerate}

		\item
			$I_{\mathscr{E} _{j}}$ is a function of $\| \bX \bQ_{j} \|_{\op}$ and $\| \bz_{j} \|_{2}$ only, 
			consequently the event $\mathscr{E}_j$ is independent of $\bz_{j}/ \| \bz_{j} \|_{2}$.

		\item 
		$
				\| \bX \bSigma^{-1/2} n^{-1/2} \|_{\op} I_{\mathscr{E}_{j}}
			\leq 
                            1 + \sqrt{p/n} + \eta_n
			+
                        (1 + \eta_n)
                        = 2 + \sqrt{p/n} + 2\eta_n
			.
		$

	\end{enumerate}

\end{remark}

\subsection{Twice continuously differentiable penalty}

\begin{lemma}
    \label[lemma]{lem:VNM}
 Let $L,\tau$ be such that
Assumptions \ref{as:rho} and \ref{as:g} are fulfilled.
Further assume that the Hessian matrix 
$\bG = (\nabla^{2} g(\bb))\big|_{\bb=\hbbeta}$ of $g$ at $\hbbeta$ exists and is symmetric, and define 
\begin{align}
    \bM &= (\bX^{\top} \diag(\bpsi') \bX + n\bG)^{-1},
    \\
    \bV &= \diag(\bpsi') - \diag(\bpsi') \bX \bM \bX^{\top} \diag(\bpsi'). 
\end{align}
Then with the partial order $\bA \preceq \bB$ if and only if the matrix $\bB - \bA$ is positive semi-definite, 
    we have
	\begin{align}
		\label{eq:DXM-bound}
                \| \diag(\bpsi') \bX \bM \|_{\op} & \leq (1/2) L^{1/2} (n \tau)^{-1/2},
	\\
		\label{eq:M-bound}
                \bM & \preceq (n\tau)^{-1} \bI_{p}, 
	\\
		\label{eq:V-bound}
		( L \tau^{-1} \| \bX n^{-1/2} \|_{\op}^{2} 
		+ 1 )^{-1} \diag(\bpsi') & 
		\preceq \bV 
		\preceq \diag(\bpsi') 
		\preceq L \bI_{n}.
	\end{align}

\end{lemma}

\begin{proof} [of Lemma \ref{lem:VNM}]
	Throughout the proof we denote by
	$\bB = \diag(\bpsi')^{1/2} \bX n^{- 1/2} \bG^{-1/2}$.
	By some algebra, we have
	$\bB (\bB^\top \bB + \bI_{p})^{-1} \bB^\top = \diag(\bpsi')^{1/2}\bX \bM \bX^{\top} \diag(\bpsi') ^{1/2}$. 
	For an upper bound of $\diag(\bpsi') \bX \bM$,
	we notice 
        \begin{align*}
		\diag(\bpsi') \bX \bM
                &= 
		\diag(\bpsi') \bX (\bX^\top \diag(\bpsi') \bX + n \bG)^{-1} 
                \\&= 
		\diag(\bpsi')^{1/2} \bB (\bB^\top \bB + \bI_p)^{-1} n^{-1/2}\bG^{-1/2}.
        \end{align*}
	For any matrix $\bB$,
	$\| \bB ( \bB^{\top} \bB + \bI_{p} )^{-1} \|_{\op} \leq
            \max_{t\ge0}t/(t^2+1)
        =
        1/2$. 
	By strongly convexity of $g$, 
	$ \| \bG^{-1/2} \|_{\op} \leq \tau^{-1/2} $.
	Since $\psi$ is $L$-Lipschitz, 
	$ \| \diag(\bpsi')^{1/2} \|_{\op} \leq L^{1/2} $.
	Combining those upper bounds, we obtain
	\eqref{eq:DXM-bound}.
        For the upper bound of $\bM$, 
	we notice 
	$
	\bM 
	 = n^{-1} \bG^{-1/2} (\bB^\top \bB + \bI_p)^{-1} \bG^{-1/2}
	 \preceq n^{-1} \bG^{-1}
         \preceq (n\tau)^{-1} \bI_p.
	$
	This gives \eqref{eq:M-bound}.
	For lower and upper bounds on $\bV$,
	we first notice that 
	by definition of $\bB$, 
	$
			\| \bB \|_{\op}
		\leq
			L^{1/2} 
			\tau^{-1/2}
			\| \bX n^{-1/2} \|_{\op} 
		,
	$
	thus
	\[
			(
				L\tau^{-1}
				\|\bX n^{-1/2}\|_{\op}^{2}
			+
				1
			)^{-1} \bI_{n}
		\preceq
			\bI_{n}-\bB(\bB^{\top}\bB+\bI_{p})^{-1}\bB^{\top}
		\preceq
			\bI_{n}
		.
	\]
	Since $\bV = \diag(\bpsi')^{1/2} (\bI_{n} - \bB(\bB^{\top}\bB + \bI_{p})^{-1}\bB^{\top}) \diag(\bpsi')^{1/2}$,
	we have 
	\[
		( L \tau^{-1} \| \bX n^{-1/2} \|_{\op}^{2} 
		+ 1 )^{-1}
		\diag(\bpsi') 
		\preceq \bV 
		\preceq \diag(\bpsi').
	\] 
	By the $L$-Lipschitz property of $\psi$, we have $\diag(\bpsi') \preceq L \bI_{n}$.
	Thus \eqref{eq:V-bound} holds. 
\end{proof}

\begin{proposition} \label[proposition]{prop:calc_deri}
    Assume that $g$ is strongly convex with 
    parameter $\tau > 0$ and $\psi=\rho'$ is $L$-Lipschitz. 
    Let $\hbbeta = \hbbeta(\bep,\bX)$ be as in \eqref{eq:bbeta}.  
    Let $\bpsi(\bep,\bX)=\bpsi$, $\bpsi'$ and $\bh(\bep,\bX)=\hbbeta-\bbeta$ be as in \eqref{eq:bpsi}. 
    Define $\nabla_{\bz}\bh = (\partial/\partial \bz)\bh(\bep,\bX+\bz\ba^\top)\big|_{\bz=0}$ 
    and $\nabla_{\bz}\bpsi = (\partial/\partial \bz)\bpsi(\bep,\bX+\bz\ba^\top)\big|_{\bz=0}$ 
    for fixed $\ba$, $\bep$ and $\bX$. 
    Let $\bP_{\bx}^\perp = \bI_n - \bx \bx^{\top} / \| \bx \|_2^2$ for $\bx\neq {\bf 0}$.
    Then

    (i) For fixed $\bep$, $\bh(\bep,\bX)$ and $\bpsi(\bep,\bX)$ are
    Lipschitz in $\bX$ 
    on every compact subset of $\R^{n\times p}$. 

    For (ii) and (iii), additionally assume that $g$ is twice continuously differentiable.

    (ii) 
    For almost every $\bX$ and every $\ba \in \R^{p}$,
    \begin{align}
    &    (\nabla_{\bz}\bh)^\top =
        - (\bh^{\top} \ba) \bM \bX^{\top} \diag(\bpsi') + \bM \ba \bpsi^{\top},
        \label{eq:gradient_hbeta}
    \\
    &  (\nabla_{\bz}\bpsi)^\top =
        -
        (\bh^{\top} \ba) \bV
        -
        \diag(\bpsi') \bX \bM \ba \bpsi^\top.
        \label{eq:gradient_bpsi}
    \end{align}
    
    (iii) For almost every $\bX$ and every $\ba \in \R^{p}$, 
	if $\bpsi \ne \mathbf{0}$,
	then 
	\begin{equation} \label{eq:gradient_bpsi_normalized}
	(\nabla_{\bz}(\bpsi/\|\bpsi\|_2))^\top =
	\|\bpsi\|_2^{-1}\bP^\perp_{\bpsi} \big[
		-
		(\bh^{\top} \ba) \bV
		-
		\diag(\bpsi') \bX \bM \ba \bpsi^\top
	\big] 
	\end{equation}
\end{proposition}

We remark that in view of \eqref{bpsi-bz_j}, $\nabla_{\bz} \bpsi = \nabla \bpsi(\bz_j)$ when $\ba=\bfe_j$.

\begin{proof}
    (i)
	We refer our readers to the proof of Proposition \ref{prop:Lipschitz_in_X}.

    (ii)
    As the functions $\bh$ and $\bpsi$ 
    are Lipschitz on every compact, their Fr\'echet derivatives
    exist almost everywhere by Rademacher's theorem.
    Moreover, the chain rule holds
    almost everywhere
    by \cite[Theorem 2.1.11]{Zie89}. 
    Let $\bG, \bV$ and $\bM$ be as in \Cref{lem:VNM}.  
    By differentiating these KKT conditions $\bX^\top \bpsi(\bX) = n (\nabla g(\bb))\big|_{\bb=\bbeta+\bh}$,
    and by the chain rule,
    we obtain that
    for almost every $\bX$, 
\begin{align*}
        n\bG\big(\nabla_{\bz}\bh\big)^\top
        = \ba \bpsi^\top + \bX^\top\big(\nabla_{\bz}\bpsi\big)^\top,\quad 
        \big(\nabla_{\bz}\bpsi\big)^\top = \diag(\bpsi')\big\{ - (\ba^\top\bh) \bI_n
        - \bX\big(\nabla_{\bz}\bh\big)^\top\big\}. 
\end{align*}
    Solving the above system of equations 
    gives \eqref{eq:gradient_hbeta} and \eqref{eq:gradient_bpsi}.

    (iii)
	Since the map 
	$\bv \mapsto \bv / \| \bv \|_2$ with $\bv \in \R^{n}$
	has Fr\'echet derivative 
	$\| \bv \|_2^{-1}\bP_{\bv}^\perp$ 
	at every point $\bv \neq \bzero \in \R^{n}$,
	by chain rule, \eqref{eq:gradient_bpsi_normalized} follows for almost every $\bX$ 
	if $\bpsi(\bX) \neq \bzero_{n}$.
\end{proof}


\begin{proof}[of \Cref{lem:inter_step_zj}
    when $g$ is twice continuously differentiable]
Here, we further assume that $g$ is twice differentiable 
so that $\bV, \bM$ in \Cref{lem:VNM} are well-defined.

By Proposition \ref{prop:Lipschitz_in_X}, 
the map $\bz_{j} \mapsto (\bh, \bpsi)$ is locally Lipschitz,
thus the map of the product $\bz_{j} \mapsto h_j \bpsi ( \| \bpsi \|_2^{2} + \delta )^{-1}$
is also locally Lipschitz.
By Proposition \ref{prop:psi_not_zero},
without loss of generality, we can assume that 
$\bpsi \neq \bzero_{n}$ at almost every point $\bz_{j} \in \R^{n}$.
By the first order Stein's formula on the sphere 
\eqref{eq:Sphere-First-Order-Stein}
for $(n-1)\|\bz_{j}\|_{2}^{-2}\E_{j}[\bz_{j}^{\top}\bff(\bz_{j})]$ 
with $\bff(\bz_{j})=h_{j}\bpsi(\|\bpsi\|_2^{2}+\delta)^{-1}$, 
we have
	\begin{align}
	 & \E_{j}[h_{j}^{2}(\|\bpsi\|_{2}^{2}+\delta)^{-1}\trace(\bV\bP_{\bz_{j}}^{\perp})]\label{eq:3254}\\
	= & -\E_{j}[h_{j}(\|\bpsi\|_{2}^{2}+\delta)^{-1}\bpsi^{\top}\bz_{j}](n-1)\|\bz_{j}\|_{2}^{-2}\label{eq:3254_term1}\\
	 & -2\E_{j}[h_{j}(\|\bpsi\|_{2}^{2}+\delta)^{-1}\be_{j}^{\top}(\diag(\bpsi')\bX\bM)^{\top}\bP_{\bz_{j}}^{\perp}\bpsi]\label{eq:3254_term2}\\
	 & +\E_{j}[(\|\bpsi\|_{2}^{2}+\delta)^{-1}\be_{j}^{\top}\bM\be_{j}\bpsi^{\top}\bP_{\bz_{j}}^{\perp}\bpsi]\label{eq:3254_term3}\\
	 & +2\E_{j}[h_{j}^{2}(\|\bpsi\|_{2}^{2}+\delta)^{-2}\bpsi^{\top}\bV\bP_{\bz_{j}}^{\perp}\bpsi]\label{eq:3254_term4}\\
	 & +2\E_{j}[h_{j}(\|\bpsi\|_{2}^{2}+\delta)^{-2}\bpsi^{\top}\diag(\bpsi')\bX\bM\be_{j}\bpsi^{\top}\bP_{\bz_{j}}^{\perp}\bpsi]\label{eq:3254_term5}.  
	\end{align}
For the terms (\ref{eq:3254_term2})-(\ref{eq:3254_term4}),
by Lemma \ref{lem:VNM} and $\|\bP_{\bz_{j}}^{\perp}\|_{{\rm op}}\leq1$
we find
\begin{align*}
    \eqref{eq:3254_term2} \vee \eqref{eq:3254_term5}
    & \leq\sqrt{L/(\tau n)}\E_{j}[|h_{j}|(\|\bpsi\|_{2}^{2}+\delta)^{-1/2}]
    && \text{by \eqref{eq:DXM-bound}}
    \\
    & \leq\tfrac{1}{2}(\tau n)^{-1}+\tfrac{L}{2}\E_{j}[h_{j}^{2}(\|\bpsi\|_{2}^{2}+\delta)^{-1}],\\
    \eqref{eq:3254_term3} & \leq(\tau n)^{-1}
                          && \text{by \eqref{eq:M-bound}}
,\\
\eqref{eq:3254_term4} & \leq2L\E_{j}[h_{j}^{2}(\|\bpsi\|_{2}^{2}+\delta)^{-1}]
                      && \text{by \eqref{eq:V-bound}}.
\end{align*}
By leaving term (\ref{eq:3254_term1}) unchanged and using the above inequalities,
\begin{equation} \label{eq:8508}
	\begin{aligned}
			\eqref{eq:3254}
		\leq 
			& -(n-1)\|\bz_{j}\|_{2}^{-2}\E_{j}[h_{j}(\|\bpsi\|_{2}^{2}+\delta)^{-1}\bpsi^{\top}\bz_{j}]
			\\
			& +2(\tau n)^{-1}
			\\
			& +3L\E_{j}[h_{j}^{2}(\|\bpsi\|_{2}^{2}+\delta)^{-1}]
		 .
	\end{aligned}
\end{equation}

	We now bound \eqref{eq:3254} from below. 
	Let
	\[
		U_{j}
	=
		(L\tau^{-1}(\|\bX\bQ_{j}n^{-1/2}\|_{\op}
	+	\|\bz_{j}n^{-1/2}\|_{2})^{2}
	+	1)^{-1}.
	\]
	By 
	$
		(L \tau^{-1} \| \bX n^{-1/2}\|_{\op}^{2}+1)^{-1}\diag(\bpsi')
	\preceq \bV
	\preceq \diag(\bpsi')
	\preceq L\bI_{n}
	$
	in \eqref{eq:V-bound} and by
	$
		\| \bX \|_{\op}
		\leq
		\| \bX \bQ_{j} \|_{\op}
		+ 
		\| \bz_{j} \|_{2}
	$
	for all $j\in[p]$, 
	we have that $U_{j}\diag(\bpsi')\preceq\bV\preceq L\bI_{n}$.

Since $U_j \diag(\bpsi')  \preceq \bV$ and
$K^2 \bI_n      \preceq \diag(\bpsi') + \diag\{\psi_i^2 ,i=1,...,n\}$
both holds,
multiplying the second inequality by $U_j$ and summing yields
$K^2U_j \bI_n \preceq U_j \diag \psi_i^2  +  \bV$.
Multiplying both sides $\bP_{\bz_j}^\perp$ to the left and to the right
and using that $\bP_{\bz_j}^\perp \preceq \bI_n$ and $\trace \bP_{\bz_j}^\perp = n-1$ we find
$$K^2U_j \bP_{\bz_j}^{\perp} \preceq U_j\diag \psi_i^2  +  \bP_{\bz_j}^{\perp} \bV \bP_{\bz_j}^{\perp},
\qquad
    K^2U_j(n-1) \le U_j \|\bpsi\|^2 + \trace[\bP_{\bz_j}^{\perp} \bV].
$$
Multiplying by $(\|\bpsi\|_2^2 + \delta)^{-1} h_j^2$ and taking the conditional
expectation $\E_j$,
\begin{align}
    K^2 U_j (n-1)
    \E_j[
        (\|\bpsi\|_2^2 + \delta)^{-1}
        h_j^2
    ]
     &\le
     \E_j[h_j^2] + \E_j[\trace[\bP_{\bz_j}^{\perp} \bV] h_j^2
     (\|\bpsi\|_2^2 + \delta)^{-1}]
     \\&= \E_j[h_j^2] + \eqref{eq:3254}
     .
     \label{eq:lower_bound}
\end{align}
Note that by definition of $u_*$ and the events in
Lemma~\ref{lem:high_p_Ej}, we have when $u_{*} > 0$, 
	\begin{equation} \label{eq:Ej_requirement}
		\mathscr{E}_{j}	
		\subset
		\{ 
			n u_{*}
		\leq 
			K^{2} U_{j} (n-1) 
		- 3L
		\} 
		\cap
		\{ 
			\Omega_{jj} \| \bz_{j} \|_{2}^{2} \geq n (1 - \eta_{n})_{+}^{2}
		\}.
	\end{equation}
	Then combining \eqref{eq:8508} with the previous display,
	multiplying both sides by $I_{\mathscr{E}_j}$ and summing over $j\in[p]$ we find
	\[
			n u_*
			\sum_{j=1}^p
			\E_j\Bigl[
                        \frac{
                            I_{\mathscr{E}_j}
                            h_j^2
                            }{
				\|\bpsi\|_2^2 + \delta
                            }
				\Bigr]
		\le
			\frac{2 p}{\tau n}
			+\sum_{j=1}^p
			\E_j[I_{\mathscr{E}_j} h_j^2]
			-(n-1)\E_j\Bigl[
                            \frac{ \bpsi^\top \bz_j
                            I_{\mathscr{E}_j}
                            \be_j^\top \bh
                            }{
                            (\|\bpsi\|_2^2 + \delta)
                            \|\bz_j\|_2^2
                        } 
			\Bigr].
	\]
	Taking expectations $ \E $, 
	letting $\diag\{I_{\mathscr{E}_{j}}\}$ denote the diagonal matrix 
	with the $j$-th diagonal element $I_{\mathscr{E}_{j}}$,
	we find
	\begin{align*}
	&n u_* \E\bigl[
			(\|\bpsi\|_2^2 + \delta)^{-1}
			\bh^\top \diag \{ I_{\mathscr{E}_j} \} \bh
	\bigr]
        -
        2p/(\tau n) - \E[\|\bh\|_2^2]
	\\&\le
	- 
        (n-1)
        \E\bigl[
		(\|\bpsi\|_2^2 + \delta)^{-1}
		\bpsi^\top \bX \bSigma^{-1}
		\diag \{ \Omega_{jj}^{-1} I_{\mathscr{E}_j} \|\bz_j\|_2^{-2} \} \bh
	\bigr]
	\\&\le
        (n-1)\E\bigl[ \|\bX\bSigma^{-1}
	\diag \{ \Omega_{jj}^{-1} I_{\mathscr{E}_j} \|\bz_j\|_{2}^{-2} \}\|_{\op}^2
	\bigr]^{1/2}
	\E\bigl[
		(\|\bpsi\|_2^2 + \delta)^{-1}
		\bh^\top \diag \{ I_{\mathscr{E}_j} \} \bh
	\bigr]^{1/2}
	\\&\le
	(1 - \eta_{n})_{+}^{-2} \E\bigl[ \|\bX\bSigma^{-1}\|_{\op}^2
	\bigr]^{1/2}
	\E\bigl[
		(\|\bpsi\|_2^2 + \delta)^{-1}
		\bh^\top \diag \{ I_{\mathscr{E}_j} \} \bh
	\bigr]^{1/2}
	\end{align*}
	where the second inequality follows from the Cauchy-Schwarz inequality
	and $(\|\bpsi\|_2^2 + \delta)^{-1/2}\|\bpsi\|_2 \le 1$,
	and the third inequality follows from
	$\Omega_{jj}^{-1} I_{\mathscr{E}_j} (n-1)\|\bz_j\|_2^{-2}
	\le (1 - \eta_{n})_{+}^{-2} $ thanks to \eqref{eq:Ej_requirement}.
	This implies that
	$x = (n\E[
    (\|\bpsi\|_2^2 + \delta)^{-1}
    \bh^\top \diag \{ I_{\mathscr{E}_j} \} \bh
])^{1/2}$ 
	satisfies
	$A x^2 + B x + C \le 0$
where the polynomial coefficients are
	$A=u_*$, $C=-2p(\tau n)^{-1} - \E[\|\bh\|_{2}^2] $ 
	and 
	$
	B =
	- ( 1 - \eta_{n} ) _{+}^{-2}
	\E [ \| \bX \bSigma^{-1} n^{-1} \|_{\op}^{2} ]^{1/2}.
	$
	As $A > 0$,
	inequality 
	$A x^2 + B x + C \le 0$
	implies that $x$ lies between the two real roots
	of the polynomial $AX^2 + BX + C$. In particular, $x$ is smaller than
	the largest root, i.e.,
	$
		x \le (-B + \sqrt{B^2 - 4AC})/(2A)
		\le |B|/A + (|C|/A)^{1/2}
	$.
	Here, 
	$
		|C|/A = (2p (n \tau)^{-1} + \E[\|\bh\|_{2}^2]) / u_*
	$ 
	and
	\[ 
			|B|/A
		\leq 
			u_{*}^{-1}
			( 1 - \eta_{n})_{+}^{-2} \big|
                        \| \bSigma^{-1/2 } \|_{\op}
			\E [ \| \bX \bSigma^{-1/2} n^{-1/2} \|_{\op}^{2} ] ^{1/2}.  
	\]
        The upper bound \eqref{eq:bound-squared-operator-norm}
        on $\E [ \| \bX \bSigma^{-1/2}\|_{\op}^{2} ]$ completes the proof.

\end{proof}

\subsection{Non-smooth strongly convex penalty $g$}

\subsubsection{Almost everywhere differentiability}
In this section, 
we provide the almost everywhere existence of the Jacobian matrices.
We also notice that if our penalty $g$ is not twice differentiable,
the matrices $\bM$ and $\bV$ in \Cref{lem:VNM} are
not well defined.
In this case we do not have explicit formula for the Jacobian matrices $( \partial / \partial \bz_{j}) \bpsi$ and $(\partial / \partial \bz_{j}) \bh $ such as those in terms of $\bV,\diag(\bpsi'),\bM$ in Proposition \ref{prop:calc_deri} (ii) and (iii).
In this section we provide upper bounds of the norms of these Jacobian matrices without using Proposition \ref{prop:calc_deri}.

\begin{proposition} \label[proposition]{prop:Lipschitz_in_X} 
Let $\rho: \R \to \R$ be convex
and continuously differentiable 
with derivative $\psi=\rho'$ being $L$-Lipschitz. 
Let $g: \R^{p} \to \R$ be strongly convex with parameter $\tau > 0$.
Let $\bX, \tbX \in \R^{n\times p}$, $\bep, \tbep \in \R^{n}$ and
correspondingly, 
\begin{equation}
\begin{aligned}
\hbbeta (\bep, \bX) & = \argmin_{\bb \in \R^{p}}
\sum _{i \in [n]} 
\frac{\rho ( \ep_{i}  - \bx_{i}^{\top} ( \bb - \bbeta ))}{n}
+
g ( \bb ),
\quad
                    &\tbbeta = \hbbeta (\tbep, \tbX),
\\
\bh (\bep, \bX) & = \hbbeta (\bep, \bX) - \bbeta, 
                &\tbh = \bh (\tbep, \tbX),
\\
\bpsi (\bep, \bX) & = \psi ( \bep - \bX \bh (\bep, \bX) ), 
                  &\tbpsi = \bpsi ( \tbep, \tbX). 
\end{aligned} 
\label{eq:formal-definition-bpsi-eps-X}
\end{equation}
Then (i)
\begin{equation} \label{eq:3692}
n\tau\| \bh-\tbh\| _{2}^{2}+L^{-1}\| \bpsi-\tbpsi\| _{2}^{2}
\leq(\bh-\tbh)^{\top}(\bX-\tbX)^{\top}\bpsi+(\bep-\tbep+\tbX\bh-\bX\bh)^{\top}(\bpsi-\tbpsi).
\end{equation}

(ii) The map
$(\bep,\bX) \mapsto \bigl(\bh(\bep,\bX),\bpsi(\bep,\bX)\bigr)$ is Lipschitz on every compact of $\R^{n} \times \R^{n \times p}$. 

(iii)
For any $\bep \in \R^{n}$ fixed,
and for any $\bfeta \in \R^{n}$ and $\ba \in \R^{p}$
\begin{align*}
    &n\tau\|\hbbeta(\bep,\bX+\bfeta\ba^{\top})-\hbbeta(\bep,\bX)\|_{2}^{2}  
    +L^{-1}\|\bpsi(\bep,\bX+\bfeta\ba^{\top})-\bpsi(\bep,\bX)\|_{2}^{2}                                                                             
    \\& \leq  (n\tau)^{-1}\|\ba\|_{2}^{2}(\bfeta^{\top}\bpsi)^{2}+L(\bh^{\top}\ba)^{2}\|\bfeta\|_{2}^{2}.
\end{align*}
Furthermore, 
\begin{equation} \label{eq:bound_fnorm_jacobian}
    \sum_{i\in[n]}\|\bpsi(\bep,\bX+\be_{i}\ba^{\top})-\bpsi(\bX)\|_{2}^{2}
	\leq
        L(n\tau)^{-1}\|\ba\|_{2}^{2}\|\bpsi\|_{2}^{2}+nL^{2}(\bh^{\top}\ba)^{2}.
\end{equation}

(iv) If $\bfeta\in\R^n$ is such that
$\bfeta^\top\bpsi(\bep,\bX)=0$ then
$\bpsi(\bep + \bh(\bep,\bX)^\top\be_j \bfeta, \bX + \bfeta\be_j^\top)=\bpsi(\bep,\bX)$,
so that if $\bpsi(\bep,\bX)$ is Fr\'echet differentiable at $(\bep,\bX)$ then
\begin{equation}
\sum_{i=1}^n
\eta_i
\Big[
    \frac{\partial \bpsi}{\partial x_{ij}}(\bep,\bX)
+
\Bigl(\be_j^\top \bh(\bep,\bX)\Bigr)
\frac{\partial \bpsi}{\partial \veps_i}(\bep,\bX)
\Big]
= 0 
\quad\text{ provided that }
\bfeta^\top\bpsi(\bep,\bX)=0.
\label{eq:divergence-intermediary}
\end{equation}
\end{proposition}
The content of the above proposition appeared in \cite[Proposition 5.2]{bellec2020out}
with variables $(\by,\bX)$ instead of $(\bep,\bX)$.
It follows from strong convexity and
the KKT conditions of $\hbbeta$ and $\tbbeta$.
Its proof is provided below for completeness.
An application of the above Proposition~\ref{prop:Lipschitz_in_X} to normalized $\bpsi$ yields the following. 

\begin{corollary} Under the conditions of Proposition \ref{prop:Lipschitz_in_X} and 
with the notation of \Cref{prop:calc_deri}, 
at a point where $\|\bpsi\|_2>0$ and $\bpsi$ is Fr\'echet differentiable,  
\begin{equation} \label{bound_ratio_jacobian}
    \bigg\| \nabla_{\bz}\bigg(\frac{\bpsi}{\|\bpsi\|_2}\bigg)
    \bigg\|_{F}^{2}  
    \leq  \frac{L\|\ba\|_{2}^{2}}{n\tau}
    +\frac{nL^2(\bh^{\top}\ba)^{2}
    }{\|\bpsi\|_2^2},
\end{equation} 
and with the $\nabla_{\by}\bpsi$ in \eqref{hbpsi-y} 
\begin{equation}\label{diff_ratio}
\bigg(\nabla_{\bz}\bigg(\frac{\bpsi}{\|\bpsi\|_2}\bigg)
+ \big(\ba^\top\bh\big)\frac{\nabla_{\by}\bpsi}{\|\bpsi\|_2}\bigg)^\top
\bP_{\bpsi}^\perp= {\bf 0}. 
\end{equation} 
\end{corollary}

\begin{proof}
    For \eqref{bound_ratio_jacobian},
     by the chain rule
    $
    \| 
    \nabla_{\bz} (\bpsi / \| \bpsi \|_2)^\top
    \|_{F}^2
    = 
    \| \bpsi \|_2^{-2} \| \bP_{\bpsi}^{\perp} (\nabla_{\bz} \bpsi)^\top \|_{F}^2
    \leq
    \| \bpsi \|_2^{-2} 
    \| \nabla_{\bz} \bpsi \|_{F}^2
    $.
    By definition of the Frobenius norm
    $
    \| \nabla_{\bz} \bpsi \|_F^2
    = 
    \sum_{i \in [n]} \| \partial \bpsi / \partial z_i \|_2^2.
    $
    By \eqref{eq:bound_fnorm_jacobian} with $\ba$
    replaced by $t \ba$ and taking the limit as $t\to 0$
    we obtain
    $
    \| \nabla_{\bz} (\bpsi / \| \bpsi \|_2) \|_{F}^2
    \leq
    L(n\tau)^{-1}\|\ba\|_{2}^{2}+nL^{2}(\bh^{\top}\ba)^{2}\|\bpsi\|_{2}^{-2}
    $.

For \eqref{diff_ratio}, we have 
$$
\bigg(\bigg(\frac{\partial}{\partial \bfeta}
+ \big(\ba^\top\bh\big) \frac{\partial}{\partial \by}\bigg)\frac{\bpsi}{\|\bpsi\|_2}\bigg)\bP_{\bpsi}^\perp
= \bigg(\frac{\bP_{\bpsi}^\perp}{\|\bpsi\|_2}\bigg(\frac{\partial\bpsi}{\partial \bfeta}
+ \big(\ba^\top\bh\big) \frac{\partial \bpsi}{\partial \by}\bigg)\bigg)\bP_{\bpsi}^\perp={\bf 0}
$$
due to $\bpsi\big(\bep + \bh(\bep,\bX)^\top\ba(\bP_{\bpsi}^\perp\bfeta), \bX + (\bP_{\bpsi}^\perp\bfeta)\ba^\top\big)=\bpsi(\bep,\bX)$ 
by \eqref{eq:3692} as in the proof of Proposition~\ref{prop:Lipschitz_in_X}~(iv). Note that if $F(\by,\bX)$ and $G(\bep,\bX)$
        are 
        functions such that $G(\bep,\bX) = F(\bep + \bX\bbeta,\bX)$
        and $F(\by,\bX) = G(\by - \bX\bbeta,\bX)$ 
        then $(\partial/\partial y_i) F(\by,\bX) = 
        (\partial/\partial \veps_i) G(\bep,\bX)$ whenever
        $F$ is Fr\'echet
        differentiable at $(\by,\bX)$ and $G$ is
        Fr\'echet differentiable at $(\bep,\bX)$
        (i.e., translation by a constant in the variables does not
        change the derivatives).
\end{proof}

\begin{proof} [of Proposition \ref{prop:Lipschitz_in_X} (i)] 
Let $\partial g (\cdot)$ denote the subdifferential of $g$. 
The KKT conditions read 
$\bX^{\top} \bpsi \in n \partial g (\hbbeta)$
and
$\tbX^{\top} \tbpsi \in n \partial g (\tbbeta)$.
Taking the difference and by $\tau$-strong convexity of $g$, we have 
\begin{align*}
n \tau \| \hbbeta - \tbbeta \|_{2}^{2}
& \le
( \hbbeta - \tbbeta )^\top
( \bX^{\top} \bpsi - \tbX^{\top} \tbpsi )
\\&=
( \hbbeta - \tbbeta )^\top
\big[
(\bX - \tbX) ^{\top} \bpsi
+
\tbX{}^{\top} (\bpsi - \tbpsi)
\big] 
\end{align*}
For the second term,
\begin{align*}
    & (\hbbeta-\tbbeta)^{\top}\tbX{}^\top(\bpsi-\tbpsi) 
    \\ & = (
    \tbX\bh-\tbX\tbh 
    )^{\top}(\bpsi-\tbpsi)\\
 & =-(\bep-\bX\bh-(\tbep-\tbX\tbh))^{\top}(\bpsi-\tbpsi)+(\bep-\tbep+\tbX\bh-\bX\bh)^{\top}(\bpsi-\tbpsi)
 .
\end{align*}
Since $\psi$ is non-decreasing and $L$-Lipschitz, 
$L^{-1} \| \bpsi - \tbpsi \|_{2}^{2}
\le
\bigl(\bep-\bX\bh-(\tbep-\tbX\tbh)\bigr)^{\top}(\bpsi-\tbpsi)$
holds.
Combining the above displays we 
    obtain
    \eqref{eq:3692}.
\end{proof}

\begin{proof} [of Proposition \ref{prop:Lipschitz_in_X} (ii)] 
    For fixed values of $(\bep,\bX,\bh,\bpsi)$,
    inequality
    \eqref{eq:3692} divided by $1+\|\bh-\tbh\|_2 + \|\bpsi-\tbpsi\|_2$
    implies that $(\tbh,\tbpsi) \to (\bh,\bpsi)$
    as $(\tbep,\tbX)\to (\bep,\bX)$, hence the function
    $(\bep,\bX) \mapsto (\bh(\bep,\bX),\bpsi(\bep,\bX))$ is
    everywhere continuous.
    This implies that
    $S(K) = \sup_{(\bep,\bX)\in K}
    \big((n\tau)^{-1}\|\bpsi(\bep,\bX)\|_2^2
    +L\|\bh(\bep,\bX)\|_2^2\big)
    $
    is finite for any compact $K\subset \R^n\times \R^{n\times p}$.
    The Cauchy-Schwartz inequality on the right hand side  of \eqref{eq:3692}
    gives 
    for any $(\bep,\bX),(\tbep,\tbX)\in K$
\begin{align*}
&n\tau\| \bh-\tbh\| _{2}^{2}+L^{-1}\| \bpsi-\tbpsi\| _{2}^{2} 
\\& \leq\| \bX-\tbX\| _{\op}(\| \bh-\tbh\| _{2}\| \bpsi\| _{2}+\| \bh\| _{2}\| \bpsi-\tbpsi\| _{2})+\| \bep-\tbep\| _{2}\| \bpsi-\tbpsi\| _{2}\\
                                                             & \leq\| \bX-\tbX\| _{\op}(n\tau\| \bh-\tbh\| _{2}^{2}+L^{-1}\| \bpsi-\tbpsi\| _{2}^{2})^{\frac{1}{2}}S(K)^{\frac12}
  +
  \| \bep-\tbep\| _{2}
  \| \bpsi-\tbpsi\| _{2}, 
\end{align*}
This implies that $(\bep,\bX) \mapsto (\bh,\bpsi)$ is Lipschitz on $K$. 
\end{proof}

\begin{proof} [of Proposition \ref{prop:Lipschitz_in_X} (iii)]
Combined with \eqref{eq:3692} with $\bep = \tbep$,
we have 
\begin{align*}
&n\tau\|\bh-\tbh\|_{2}^{2}+L^{-1}\|\bpsi-\tbpsi\|_{2}^{2} \\& \leq-(\bh-\tbh)^{\top}\ba(\bfeta^{\top}\bpsi)+(\bh^{\top}\ba)\bfeta^{\top}(\bpsi-\tbpsi)\\
 & \leq \|\bh-\tbh\|_{2}\|\ba\|_{2}|\bfeta^{\top}\bpsi|+|\bh^{\top}\ba|\|\bfeta\|_{2}\|\bpsi-\tbpsi\|_{2} \\
 & \leq\bigl(n\tau\|\bh-\tbh\|_{2}^{2}+L^{-1}\|\bpsi-\tbpsi\|_{2}^{2}\bigr)^{1/2}\bigl((n\tau)^{-1}\|\ba\|_{2}^{2}(\bfeta^{\top}\bpsi)^{2}+L(\bh^{\top}\ba)^{2}\|\bfeta\|_{2}^{2}\bigr)^{1/2}
\end{align*}
so that the first inequality holds.
Taking summation over $\bfeta = \be_i$ for $i \in [n]$
gives \eqref{eq:bound_fnorm_jacobian}.
\end{proof}
\begin{proof} [of Proposition \ref{prop:Lipschitz_in_X} (iv)]
    For $\tbep = \bep + \bh(\bep,\bX)^\top\be_j \bfeta$
    and $\tbX = \bX + \bfeta \be_j^\top$
    we have $(\bX -\tbX)^\top \bpsi = 0$ thanks to $\bfeta^\top\bpsi = 0$
    as well as
    $\bep - \tbep +(\tbX - \bX)\bh = 0
    $. Hence the two terms in the right-hand side of
    \eqref{eq:3692} are 0 and $\|\bpsi - \tbpsi\|_2^2= 0$.
    Identity \eqref{eq:divergence-intermediary} then follows
    by definition of the Fr\'echet differentiability.
\end{proof}

\subsubsection{Approximation using smooth penalty $\tilde{g}^{\nu}$}

\begin{lemma} [Approximation of strongly convex functions] \label[lemma]{lem:approximation}
Let $g:\R^{p}\to\R$ be strongly convex with constant $\tau\geq0$.
Then for every $\nu>0$ there exists a real-analytic strongly
convex function $g_{\nu}:\R^{p}\to\R$ with constant $\tau$ such
that $g_{\nu}-\nu\leq g\leq g_{\nu}$. 
\end{lemma}

\begin{proof}
Since 
$g$ is proper, i.e.,  
$-\infty \not \in g (\R^{p})$
and 
$ \{ \bb \in \R^{p}\ |\ g (\bb) < +\infty \} \neq \emptyset$,
by Proposition 10.8 in \cite{BC17}, $g$ is strongly convex with
constant $\tau\geq0$ if and only if $f:=g-(\tau/2)\| \cdot\| _{2}^{2}$
is convex. By Theorem 1 in \cite{Aza13}, there exists a function
$f_{\nu}:\R^{p}\to\R$ real-analytic and convex that
satisfies $f_{\nu}-\nu\leq f \leq f_\nu$. The conclusion follows by letting
$g_{\nu}:=f_{\nu}+(\tau/2)\| \cdot\| _{2}^{2}$.
\end{proof}

\begin{lemma} \label[lemma]{lem:continuity} 
Let $\rho: \R \to \R$ be convex
and continuously differentiable 
with derivative $\psi=\rho'$ being $L$-Lipschitz. 
Let $g, \tilde g: \R^{p} \to \R$ be strongly convex with parameter $\tau, \tilde \tau \geq 0$.
Let $\| g - \tilde g \|_{\infty} 
= \max_{\bx \in \R^{p}} | g (\bx) - \tilde g (\bx) |$.
For $\bb \in \R^{p}$,
let
$
L (\bb;g) = \frac{ 1 } { n } \sum _{i \in [n]} 
\rho ( y_{i} - \bx_{i}^{\top} \bb )
+
g ( \bb )
$
and define
\[
\hbbeta = \argmin_{\bb \in \R^{p}} L ( \bb ; g ),
\quad
\tbbeta = \argmin_{\bb \in \R^{p}} L ( \bb ; \tilde g ),
\quad
\bpsi = \psi ( \by - \bX \hbbeta ), 
\quad
\tbpsi = \psi ( \by - \bX \tbbeta ).
\] 
Then inequality
$((\tau + \tilde \tau)/2) 
\| \tbbeta - \hbbeta \|_{2}^{2} + ( nL )^{-1} \| \bpsi - \tbpsi \|_{2}^{2}
\leq
2 \| g - \tilde g \|_{\infty}
$ holds.
\end{lemma}

\begin{proof} [of Lemma \ref{lem:continuity}] 
Denote by $\partial g (\bb)$ subdifferential of $g$ at $\bb\in\R^p$.
The KKT conditions read
\[
(1/n) \bX^{\top} \bpsi \in \partial g ( \hbbeta ), 
\qquad
(1/n) \bX^{\top} \tbpsi \in \partial \tilde g ( \tbbeta ).
\] 
By the definition of the strongly convexity, the above display implies that 
\[
(\tau / 2) \| \tbbeta - \hbbeta \|_{2}^{2} 
+ (\tbbeta - \hbbeta)^{\top} (1/n) \bX^{\top} \bpsi \leq g ( \tbbeta ) - g ( \hbbeta ), 
\] 
\[
(\tilde \tau / 2)
\| \hbbeta - \tbbeta \|_{2}^{2} 
+ (\hbbeta - \tbbeta)^{\top} (1/n) \bX^{\top} \tbpsi \leq \tilde g ( \hbbeta ) - \tilde g ( \tbbeta ). 
\]
Summing over the above displays, we have 
\[
((\tau + \tilde \tau) / 2)
\| \hbbeta - \tbbeta \|_{2}^{2}
+
( \hbbeta - \tbbeta )^{\top} ( 1/n ) \bX^{\top} ( \tbpsi - \bpsi )
\leq
g ( \tbbeta ) - \tilde g ( \tbbeta ) + \tilde g ( \hbbeta )
- g ( \hbbeta )
\leq 
2 \| g - \tilde g \|_{\infty}
.
\] 
We notice that the second term in the left hand side is 
can be rewritten
\[
    \frac1n
\bigl\langle \by - \bX \hbbeta - ( \by - \bX \tbbeta ), 
\psi ( \by - \bX \hbbeta ) - \psi ( \by - \bX \tbbeta )
\bigr\rangle
= \frac1n \sum_{i=1}^n (a_i - b_i)(\psi(a_i) - \psi(b_i))
\] 
where $a_i = y_i-\bx_i^\top\hbbeta$ and $b_i=y_i - \bx_i^\top\tbbeta$.
Since $\psi$ non-decreasing and $L$-Lipschitz,
inequality
$(a_i - b_i) (\psi(a_i) - \psi(b_i)) 
= |a_i - b_i| |\psi ( a_i ) - \psi ( b_i )|
\geq \frac 1 L (  \psi(a_i)  - \psi(b_i) )^{2}
$ completes the proof.
\end{proof}

\begin{proof}
    [of \Cref{lem:inter_step_zj}
for $g$ $\tau$-strongly convex but not necessarilily
continuously differentiable]
    In this proof, we approximate the non-smooth $g$ with smooth $g$ using Lemma \ref{lem:approximation}.
    Let $g$ be strongly convex with constant $\tau > 0$, not necessarily twice differentiable. 
    By Lemma \ref{lem:approximation}, for all $\nu > 0 $, there exists $\tilde g^{\nu}$ strongly convex with constant $\tau>0$ such that 
    $\| \tilde g ^{\nu} - g \|_{\infty} \leq \nu.$
    Let 
    $\tbbeta^{\nu} = \argmin_{\bb \in \R^{p}} L (\bb; \tilde{g}^{\nu})$ 
    and 
    $\tbpsi^{\nu} = \psi (\by - \bX \tbbeta^{\nu})$ 
    be as in Lemma \ref{lem:continuity}.
    For any $\delta>0$,
    \begin{equation}
            \E\Bigl[
                \frac{n}{\|\tbpsi^\nu\|_{2}^{2} +\delta}
                    \sum_{j\in[p]}I_{\mathscr{E}_{j}}\tilde h_{j}^{2}
                \Bigr]^{1/2}
    \leq 
            \frac{
                \bigl[(1+\sqrt{\frac p n})^2+\frac1n\bigr]^{1/2}
            }{\phi_{\min}(\bSigma)^{1/2}( 1 - \eta_{n} )_{+}^{2}  u_*}
    +
    \frac{\bigl[\frac{2p}{n\tau} + \E[\|\tbh^\nu\|_2^2]\bigr]^{1/2}}{u_*^{1/2}}
    \label{eq:with-tilde-nu}
    \end{equation}
    by
    \eqref{eq:5221}
    since $\tilde g^\nu$ is twice continuously differentiable.
    By Lemma \ref{lem:continuity}, we have
    \[
    \tau \| \tbbeta^{\nu} - \hbbeta \|_{2}^{2} + (nL)^{-1} \| \tbpsi^{\nu} - \bpsi \|_{2}^{2} \leq 2 \nu.
    \] 
    This implies that, as $\nu \to 0+$,
    the pointwise convergence 
    $(\tbh_j^\nu,\tbpsi^\nu) \to (\bh,\bpsi)$
    holds.
    By the dominated convergence theorem,
    a sufficient condition that
    \eqref{eq:with-tilde-nu}
    holds with $(\tbh^\nu,\tbpsi^\nu)$ replaced by its pointwise limit
    $(\bh,\bpsi)$ inside the two expectations
    in \eqref{eq:with-tilde-nu} is that
    $\E\sup_{\nu\in(0,1)} \|\tilde \bh^\nu\|_2^2 < +\infty$.
    By Lemma \ref{lem:continuity}, 
            \begin{equation} \label{eq:2588}
                    \|\tbh^{\nu}\|_{2}^{2}
                    \leq
                    2\|\tbh^{\nu}-\bh\|_{2}^{2}+2\|\bh\|_{2}^{2}
                    <
                    (2\nu/\tau)+2\|\bh\|_{2}^{2}
            \end{equation}
            which provides integrability of $\sup_{\nu\in (0,1)} \|\tbh^\nu\|_2^2$
    as $\E[\| \bSigma^{1/2} \bh \|_{2}^{2}] < +\infty$
    when the right-hand side of \eqref{eq:5221} is finite.
\end{proof}

\section{Auxiliary propositions}

\subsection{Decomposition of the design matrix into independent components}

\begin{proposition} [Independence between $\bX(\bI_{p} - \bb \ba^{\top})$ and $\bX\bb$] \label[proposition]{prop:indep_zj}
Let each row $\bx_i$ of $\bX \in \R^{n \times p}$ satisfy that $\bx_{i} \sim^{iid} N(\bzero, \bSigma). $
Then for any deterministic vectors $\ba, \bb \in \R^{p}$,   
$\bSigma \bb = (\bb^{\top} \bSigma \bb) \ba$ holds
if and only if 
$\bX ( \bI_{p} - \bb \ba^{\top} )$ is independent with $\bX \bb$.
Furthermore, if the above holds and the inverse matrix $\bSigma^{-1}$ exists, then $(\bb^{\top} \bSigma \bb)(\ba^{\top} \bSigma^{-1} \ba) = 1$.

\end{proposition}

\begin{proof} 
From the fact that $(\bX(\bI_p - \bb \ba^{\top}), \bX\bb)$ can be represented in a linear transformation of $n\times p $ iid $N(0,1)$ random variable,
the pair is distributed in a multivariate normal distribution. 
Since the rows of $\bX$ are independent, 
the independence between $\bX (\bI_{p} - \bb \ba^{\top})$ and $\bX \bb$ 
reduces to the independence between $\bx_{i}^{\top} (\bI_{p} - \bb \ba ^{\top})$ 
and $\bx_{i}^{\top} \bb$ for each $i \in [n]$, 
which holds if and only if the two random quantities are uncorrelated
in the sense that 
\[
\E [(\bx_{i}^{\top} (\bI_{p} - \bb \ba^{\top}))(\bx_{i}^{\top} \bb)]
=
\E [\bb ^{\top} \bx_{i} \bx_{i} ^{\top} (\bI_{p} - \bb \ba^{\top})]
=
\bb ^{\top}
\bSigma
(\bI _{p} - \bb \ba ^{\top})
=\bzero.
\] 
If the inverse matrix $\bSigma^{-1}$ exists, then the above display implies $( \bb^{\top} \bSigma \bb ) ( \ba^{\top} \bSigma^{-1} \ba )=1$.
\end{proof}

\subsection{$\psi$ at the residuals is almost surely nonzero}

\begin{proposition} \label[proposition]{prop:psi_not_zero}
    If Assumptions~\ref{as:rho}, \ref{as:g} and~\ref{as:feature} hold
    then $\P(\psi(\by - \bX\hbbeta) \neq \bzero)=1$.
\end{proposition}

\begin{proof} [of Proposition \ref{prop:psi_not_zero}] 
    If $\psi(\by-\bX\hbbeta)=0$ then $\hbbeta$ must be a minimizer
    of the penalty $g$.
    Let $\bb_0$ be a minimizer of $g$, which is unique by strong
    convexity.

Our assumption on the convexity of $\rho$ implies 
that $\psi(x)$ is non-decreasing in $x \in \R$. 
Combined with our assumption $\psi' (x) + \psi^{2} (x) \geq K^{2} > 0$ 
for every point $x \in \R$, 
this implies that 
$\psi (x) = 0$
at no more than one point in $\R$.
(Otherwise, there exists an open interval 
on which $\psi (x) = 0$ and $\psi' (x) + \psi^{2} (x) = 0$.
A contradiction then follows.)

Thus we have
$
\P(\psi(\by-\bX\hbbeta)=\mathbf{0})
\le \P(\psi(\bep - \bX(\bb_0 - \bbeta)) = \mathbf{0} )
= 0$
as $\bep - \bX(\bb_0 - \bbeta)$ has continuous distribution
by \Cref{as:feature} and $\{x\in\R: \psi(x)=0\}$ has zero Lebesgue measure.

\end{proof}

\bibliography{M-estimators-2nd-steins-sphere}

\begin{thebibliography}{33}
\providecommand{\natexlab}[1]{#1}
\providecommand{\url}[1]{\texttt{#1}}
\expandafter\ifx\csname urlstyle\endcsname\relax
  \providecommand{\doi}[1]{doi: #1}\else
  \providecommand{\doi}{doi: \begingroup \urlstyle{rm}\Url}\fi

\bibitem[Azagra(2013)]{Aza13}
Daniel Azagra.
\newblock Global and fine approximation of convex functions.
\newblock \emph{Proceedings of the London Mathematical Society}, 107\penalty0
  (4):\penalty0 799--824, 2013.

\bibitem[Bauschke and Combettes(2017)]{BC17}
Heinz~H Bauschke and Patrick~L Combettes.
\newblock \emph{Convex Analysis and Monotone Operator Theory in Hilbert
  Spaces}.
\newblock Springer, 2017.

\bibitem[Bean et~al.(2013)Bean, Bickel, El~Karoui, and Yu]{bean2013optimal}
Derek Bean, Peter~J Bickel, Noureddine El~Karoui, and Bin Yu.
\newblock Optimal m-estimation in high-dimensional regression.
\newblock \emph{Proceedings of the National Academy of Sciences}, 110\penalty0
  (36):\penalty0 14563--14568, 2013.

\bibitem[Bellec(2020)]{bellec2020out}
Pierre~C Bellec.
\newblock Out-of-sample error estimate for robust m-estimators with convex
  penalty.
\newblock \emph{arXiv preprint arXiv:2008.11840}, 2020.

\bibitem[Bellec and Zhang(2018)]{BZ18}
Pierre~C Bellec and Cun-Hui Zhang.
\newblock Second order stein: Sure for sure and other applications in
  high-dimensional inference.
\newblock \emph{arXiv preprint arXiv:1811.04121}, 2018.

\bibitem[Bellec and Zhang(2019{\natexlab{a}})]{bellec2019second_poincare}
Pierre~C Bellec and Cun-Hui Zhang.
\newblock Second order poincar\'e inequalities and de-biasing arbitrary convex
  regularizers when $p/n\to \gamma$.
\newblock \emph{arXiv preprint arXiv:1912.11943}, 2019{\natexlab{a}}.

\bibitem[Bellec and
  Zhang(2019{\natexlab{b}})]{bellec_zhang2019debiasing_adjust}
Pierre~C Bellec and Cun-Hui Zhang.
\newblock De-biasing the lasso with degrees-of-freedom adjustment.
\newblock \emph{arXiv:1902.08885}, 2019{\natexlab{b}}.
\newblock URL \url{https://arxiv.org/pdf/1902.08885.pdf}.

\bibitem[Bickel(1975)]{bickel1975one}
Peter~J Bickel.
\newblock One-step huber estimates in the linear model.
\newblock \emph{Journal of the American Statistical Association}, 70\penalty0
  (350):\penalty0 428--434, 1975.

\bibitem[Bogachev(1998)]{Bog98}
Vladimir~Igorevich Bogachev.
\newblock \emph{Gaussian measures}.
\newblock Number~62. American Mathematical Soc., 1998.

\bibitem[Boucheron et~al.(2013)Boucheron, Lugosi, and Massart]{BLM13}
St{\'e}phane Boucheron, G{\'a}bor Lugosi, and Pascal Massart.
\newblock \emph{Concentration inequalities: A nonasymptotic theory of
  independence}.
\newblock Oxford university press, 2013.

\bibitem[B{\"u}hlmann et~al.(2013)]{buhlmann2013statistical}
Peter B{\"u}hlmann et~al.
\newblock Statistical significance in high-dimensional linear models.
\newblock \emph{Bernoulli}, 19\penalty0 (4):\penalty0 1212--1242, 2013.

\bibitem[Celentano and Montanari(2019)]{celentano2019fundamental}
Michael Celentano and Andrea Montanari.
\newblock Fundamental barriers to high-dimensional regression with convex
  penalties.
\newblock \emph{arXiv preprint arXiv:1903.10603}, 2019.

\bibitem[Celentano et~al.(2020)Celentano, Montanari, and
  Wei]{celentano2020lasso}
Michael Celentano, Andrea Montanari, and Yuting Wei.
\newblock The lasso with general gaussian designs with applications to
  hypothesis testing.
\newblock \emph{arXiv preprint arXiv:2007.13716}, 2020.

\bibitem[Davidson and Szarek(2001)]{DS01}
Kenneth~R Davidson and Stanislaw~J Szarek.
\newblock Local operator theory, random matrices and banach spaces.
\newblock \emph{Handbook of the geometry of Banach spaces}, 1\penalty0
  (317-366):\penalty0 131, 2001.

\bibitem[Donoho and Montanari(2016)]{donoho2016high}
David Donoho and Andrea Montanari.
\newblock High dimensional robust m-estimation: Asymptotic variance via
  approximate message passing.
\newblock \emph{Probability Theory and Related Fields}, 166\penalty0
  (3-4):\penalty0 935--969, 2016.

\bibitem[El~Karoui(2018)]{el_karoui2018impact}
Noureddine El~Karoui.
\newblock On the impact of predictor geometry on the performance on
  high-dimensional ridge-regularized generalized robust regression estimators.
\newblock \emph{Probability Theory and Related Fields}, 170\penalty0
  (1-2):\penalty0 95--175, 2018.

\bibitem[El~Karoui et~al.(2013)El~Karoui, Bean, Bickel, Lim, and
  Yu]{el_karoui2013robust}
Noureddine El~Karoui, Derek Bean, Peter~J Bickel, Chinghway Lim, and Bin Yu.
\newblock On robust regression with high-dimensional predictors.
\newblock \emph{Proceedings of the National Academy of Sciences}, 110\penalty0
  (36):\penalty0 14557--14562, 2013.

\bibitem[Hebey(1996)]{hebey1996sobolev}
Emmanuel Hebey.
\newblock \emph{Sobolev spaces on Riemannian manifolds}, volume 1635.
\newblock Springer Science \& Business Media, 1996.

\bibitem[Huang(2020)]{huang2020}
Hanwen Huang.
\newblock Asymptotic risk and phase transition of $l_{1}$-penalized robust
  estimator.
\newblock \emph{Ann. Statist.}, 48\penalty0 (5):\penalty0 3090--3111, 10 2020.
\newblock URL \url{https://doi.org/10.1214/19-AOS1923}.

\bibitem[Huber(1964)]{huber1964}
Peter~J. Huber.
\newblock Robust estimation of a location parameter.
\newblock \emph{Ann. Math. Statist.}, 35\penalty0 (1):\penalty0 73--101, 03
  1964.
\newblock URL \url{https://doi.org/10.1214/aoms/1177703732}.

\bibitem[Huber(2004)]{huber2004robust}
Peter~J Huber.
\newblock \emph{Robust statistics}, volume 523.
\newblock John Wiley \& Sons, 2004.

\bibitem[Javanmard and Montanari(2014{\natexlab{a}})]{JavanmardM14a}
Adel Javanmard and Andrea Montanari.
\newblock Confidence intervals and hypothesis testing for high-dimensional
  regression.
\newblock \emph{The Journal of Machine Learning Research}, 15\penalty0
  (1):\penalty0 2869--2909, 2014{\natexlab{a}}.

\bibitem[Javanmard and Montanari(2014{\natexlab{b}})]{JavanmardM14b}
Adel Javanmard and Andrea Montanari.
\newblock Hypothesis testing in high-dimensional regression under the gaussian
  random design model: Asymptotic theory.
\newblock \emph{IEEE Transactions on Information Theory}, 60\penalty0
  (10):\penalty0 6522--6554, 2014{\natexlab{b}}.

\bibitem[Javanmard and Montanari(2018)]{javanmard2018debiasing}
Adel Javanmard and Andrea Montanari.
\newblock Debiasing the lasso: Optimal sample size for gaussian designs.
\newblock \emph{The Annals of Statistics}, 46\penalty0 (6A):\penalty0
  2593--2622, 2018.

\bibitem[Karoui(2013)]{karoui2013asymptotic}
Noureddine~El Karoui.
\newblock Asymptotic behavior of unregularized and ridge-regularized
  high-dimensional robust regression estimators: rigorous results.
\newblock \emph{arXiv preprint arXiv:1311.2445}, 2013.

\bibitem[Miolane and Montanari(2018)]{miolane2018distribution}
L{\'e}o Miolane and Andrea Montanari.
\newblock The distribution of the lasso: Uniform control over sparse balls and
  adaptive parameter tuning.
\newblock \emph{arXiv preprint arXiv:1811.01212}, 2018.

\bibitem[Portnoy(1985)]{portnoy1985asymptotic}
Stephen Portnoy.
\newblock Asymptotic behavior of m estimators of p regression parameters when
  p2/n is large; ii. normal approximation.
\newblock \emph{The Annals of Statistics}, pages 1403--1417, 1985.

\bibitem[Stojnic(2013)]{stojnic2013framework}
Mihailo Stojnic.
\newblock A framework to characterize performance of lasso algorithms.
\newblock \emph{arXiv preprint arXiv:1303.7291}, 2013.

\bibitem[Thrampoulidis et~al.(2018)Thrampoulidis, Abbasi, and
  Hassibi]{thrampoulidis2018precise}
Christos Thrampoulidis, Ehsan Abbasi, and Babak Hassibi.
\newblock Precise error analysis of regularized $ m $-estimators in high
  dimensions.
\newblock \emph{IEEE Transactions on Information Theory}, 64\penalty0
  (8):\penalty0 5592--5628, 2018.

\bibitem[Van~de Geer et~al.(2014)Van~de Geer, B{\"u}hlmann, Ritov, and
  Dezeure]{GeerBR14}
Sara Van~de Geer, Peter B{\"u}hlmann, Ya'acov Ritov, and Ruben Dezeure.
\newblock On asymptotically optimal confidence regions and tests for
  high-dimensional models.
\newblock \emph{The Annals of Statistics}, 42\penalty0 (3):\penalty0
  1166--1202, 2014.

\bibitem[Vershynin(2018)]{Ver18}
Roman Vershynin.
\newblock \emph{High-dimensional probability: An introduction with applications
  in data science}, volume~47.
\newblock Cambridge university press, 2018.

\bibitem[Zhang and Zhang(2014)]{ZhangSteph14}
Cun-Hui Zhang and Stephanie~S Zhang.
\newblock Confidence intervals for low dimensional parameters in high
  dimensional linear models.
\newblock \emph{Journal of the Royal Statistical Society: Series B (Statistical
  Methodology)}, 76\penalty0 (1):\penalty0 217--242, 2014.

\bibitem[Ziemer(1989)]{Zie89}
William~P. Ziemer.
\newblock \emph{Weakly Differentiable Functions: Sobolev Spaces and Functions
  of Bounded Variation}.
\newblock Graduate Texts in Mathematics. Springer, 1989.

\end{thebibliography}
\bibliographystyle{plainnat}

\end{document}